\documentclass[12pt,final]{amsart}
\newif\ifslfour
\slfourtrue

\usepackage{amsmath,amsfonts,amssymb,amsthm,color,tikz,todonotes,xspace,euscript,nicefrac,mathrsfs,calligra,bbm,enumitem,fancyhdr}

\usepackage[russian,english]{babel}
\usepackage[margin=1in]{geometry}
\usepackage{tikz-cd}
\usepackage{floatpag}
\usepackage[all]{xy}
\usepackage[utf8]{inputenc}
\usepackage[normalem]{ulem}
\usepackage{bm}
\usepackage{mathtools}
\usepackage{tensor,todonotes}
\usepackage{latexsym,verbatim,hyperref}
\usepackage{subcaption}
\usepackage{afterpage}
\newcommand{\arxiv}[1]{\href{http://arxiv.org/abs/#1}{\tt
    arXiv:\nolinkurl{#1}}}
\newcommand{\Seq}{\operatorname{Seq}}

\newcommand{\nc}{\newcommand}

\nc{\Z}{\mathbb Z}
\nc{\R}{\mathbb R}
\newcommand{\gl}{\mathfrak{gl}}

\newcommand{\Bv}{\mathbf{v}}

\newcommand{\Bb}{\mathbf{b}}

\newcommand{\mcU}{\mathcal{U}}
\newtheorem{itheorem}{Theorem}

\numberwithin{equation}{section}
\renewcommand{\theequation}{\arabic{section}.\arabic{equation}}

 \usetikzlibrary{decorations.pathmorphing,backgrounds,decorations.markings,babel,decorations.pathreplacing}

\tikzset{wei/.style={draw=red,double=red!40!white,double distance=1.5pt,thin}}
\tikzset{thickc/.style={draw=green, line width=3pt}}
\newcounter{subeqn}
\renewcommand{\thesubeqn}{\theequation\alph{subeqn}}
\newcommand{\subeqn}{\refstepcounter{subeqn}\tag{\thesubeqn}}
\makeatletter
\@addtoreset{subeqn}{equation}
\newcommand{\newseq}{\refstepcounter{equation}}

\numberwithin{equation}{section}
\theoremstyle{plain}
\newtheorem{theorem}{Theorem}[section]

\newtheorem{lemma}[theorem]{Lemma}
\newtheorem{proposition}[theorem]{Proposition}
\newtheorem{prop}[theorem]{Proposition}
\theoremstyle{definition}
\newtheorem{example}[theorem]{Example}

\newtheorem{definition}[theorem]{Definition}

\newtheorem{remark}[theorem]{Remark}
\newtheorem{corollary}[theorem]{Corollary}

\nc{\Bi}{\mathbf{i}}
\newcommand{\excise}[1]{}
\nc{\Bj}{\mathbf{j}}
\nc{\al}{\alpha}
\nc{\la}{\lambda}
\nc{\ev}{\operatorname{ev}}
\newcommand{\poly}{\mathcal{P}}
\nc{\Hom}{\operatorname{Hom}}

\nc{\MaxSpec}{\operatorname{MaxSpec}}
\nc{\bbC}{\mathbb{C}}

\newcommand{\GL}{\textrm{GL}}
\nc{\Cat}{\mathcal{C}}

\nc{\GTc}{\mbox{\rm\foreignlanguage{russian}{ГЦ}}}
\nc{\Wei}{\EuScript{Wei}}
\newcommand{\End}{\textrm{End}}

\nc{\wgmod}{\operatorname{-wgmod}}
\nc{\cO}{\mathcal{O}}
\nc{\Ba}{\mathbf{a}}
\nc{\tT}{\tilde{\mathbb{T}}}
\makeatletter
\renewenvironment{cases}[1][l]{\matrix@check\cases\env@cases{#1}}{\endarray\right.}
\def\env@cases#1{\let\@ifnextchar\new@ifnextchar
  \left\lbrace\def\arraystretch{1.2}\array{@{}#1@{\quad}l@{}}}
\makeatother

\newcommand{\theshorttitle}{Gelfand-Tsetlin modules: canonicity and calculations}
\newcommand{\thetitle}{Gelfand-Tsetlin modules:\\ canonicity and calculations}
\begin{document}\begin{center}
\noindent {\large  \bf \thetitle}
\bigskip

  \begin{tabular}{c@{\hspace{10mm}}c}
     {\sc\large Turner Silverthorne}&{\sc\large Ben Webster}\\
 \it Department of Mathematics,&   \it Department of Pure Mathematics,\\ 
   \it  University of Toronto&  \it 
   University of Waterloo \&\\
 \it Toronto, ON, Canada& \it Perimeter Institute for Mathematical Physics \\
 & \it Waterloo, ON, Canada\\
{\tt turner.silverthorne@utoronto.ca}&{\tt ben.webster@uwaterloo.ca}
 \end{tabular}
\end{center}
\bigskip
{\small
\begin{quote}
\noindent {\em Abstract.}
    In this paper, we give a more down-to-earth introduction to the connection between Gelfand-Tsetlin modules over $\mathfrak{gl}_n$ and diagrammatic KLRW algebras and develop some of its consequences.  In addition to a new proof of this description of the category Gelfand-Tsetlin modules appearing in earlier work, we show three new results of independent interest:  (1) we show that every simple Gelfand-Tsetlin module is a canonical module in the sense of Early, Mazorchuk and Vishnyakova, and characterize when two maximal ideals have isomorphic canonical modules, (2) we show that the dimensions of Gelfand-Tsetlin weight spaces in simple modules can be computed using an appropriate modification of Leclerc's algorithm for computing dual canonical bases, and (3) we construct a basis of the Verma modules of $\mathfrak{sl}_n$ which consists of generalized eigenvectors for the Gelfand-Tsetlin subalgebra.
    
    Furthermore, we present computations of multiplicities and Gelfand-Kirillov dimensions for all integral Gelfand-Tsetlin modules in ranks 3 and 4;  unfortunately, for ranks $>4$, our computers are not adequate to perform these computations.
\end{quote}}

\section{Introduction}

The finite-dimensional representation theory of $\mathfrak{gl}_n$ is one of the jewels of modern representation theory.  One well-known approach to understanding these representations is based on the formulas of Gelfand and Tsetlin \cite{gelfandFinitedimensionalRepresentations1950}; these can be found by considering the Gelfand-Tsetlin subalgebra of $U(\mathfrak{gl}_n)$ and simultaneously diagonalizing its action on any finite-dimensional representation.  

A natural generalization is to apply this approach to infinite-dimensional representations.  Of course, in many cases, such as the regular action on $U(\mathfrak{gl}_n)$, the elements of $\Gamma$ cannot be diagonalized in any meaningful sense. In order to do something even remotely similar to Gelfand and Tsetlin's approach, we need to only consider modules on which $\Gamma$ acts locally finitely.  Introduced in \cite{drozdGelFandZetlin1991}, these are called {\bf Gelfand-Tsetlin} modules and include all generalized weight modules with finite weight multiplicities with respect to the Cartan $\mathfrak{h}=\mathfrak{gl}_n\cap \Gamma$; in particular, they include all modules in category $\cO$ for a Borel containing $\mathfrak{h}$.  

The study of these modules has been quite active in recent years, with several interesting constructions allowing us to get a handle on this category, which had been rather mysterious quite recently; the most relevant recent papers for us will be those of Futorny, Grantcharov, Ramirez, and Zadunaisky \cite{FGRZVerma,futornyGelfandTsetlinModules2018}, Hartwig \cite{Hartwig} and Early, Mazorchuk, and Vishnyakova \cite{EMV}; all these authors have many other papers on the subject, so we will just note that we have not tried to give a comprehensive list here.  In this paper, we'll try to give a more down-to-earth and example-driven explanation of the approach of \cite{WebGT} which is somewhat orthogonal to the papers above; while there will be some overlap, and we will depend on \cite{WebGT} for a few results, this paper is meant to be read independently.  To give a crude summary: it is difficult to define Gelfand-Tsetlin modules because of issues that arise when the denominator and numerator of the Gelfand-Tsetlin formulas develop zeros.  One can organize the algebraic problem of defining how the singular portions of the GT formulas act in terms of defining a representation of a simpler (but still quite tricky) algebra, a {\bf KLRW algebra} $\tT$, which has appeared before in the study of categorification and quiver varieties \cite{Webmerged,Webqui}.  Versions of this theorem are proven in \cite{WebGT,KTWWYO,websterKoszulDuality2019}, but as special cases of more general results; here we give a purely algebraic proof that emphasizes the connection to the Gelfand-Tsetlin action formula.  

In addition to this new proof, we draw out some results about Gelfand-Tsetlin modules which are difficult to see from the perspective of $\mathfrak{gl}_n$.  The main new results of this paper are as follows:
\begin{itheorem}\hfill
\begin{enumerate}
    \item Every simple Gelfand-Tsetlin $\mathfrak{gl}_n$-module is the canonical module (in the sense of \cite{EMV}) of a maximal ideal in $\Gamma$, and we give an algorithm for deciding when two maximal ideals have isomorphic canonical modules.  
    \item The GT multiplicities of all simple GT modules can be computed using a modification of LeClerc's algorithm for the dual canonical basis.
    \item We give a basis of the Verma modules of $\mathfrak{gl}_n$, which consists of generalized eigenvectors of $\Gamma$, which is posed as a problem in \cite[Remark 6.8]{FGRZVerma}.
\end{enumerate}
\end{itheorem}
All of these depend on properties of the algebra $\tT$: \begin{itemize}
    \item part (1) depends on the structure of the polynomial representations of $\tT$,
    \item part (2) depends on a rather deep and surprising property of $\tT$: while it contains elements of negative degree, it's Morita equivalent to a positively graded algebra
    \item part (3) describes the Verma modules in terms of standard modules for the algebra $\tT$.  
\end{itemize} 

Finally, we include explicit calculations of various properties of GT modules.  We've written a GAP program \cite{GTcode} which computes the multiplicities of all simple GT modules, and have included tables of these computations for $\mathfrak{gl}_3$ and $\mathfrak{gl}_4$. Using these data, we also describe the Gelfand-Kirillov dimensions and essential supports of these modules, as well as whether they possess an infinite-dimensional weight space over the Cartan subalgebra.   The complexity of these computations increases very quickly as rank increases, roughly in proportion with $n^2!$, and $\mathfrak{gl}_5$ is far beyond the abilities of our computer; one illustrative point is that the number of integral Gelfand-Tsetlin modules in a regular block grows very fast with $n$:\medskip

\setlength{\tabcolsep}{1em} 
\centerline{\begin{tabular}[]{|c|c|c|c|c|c|c|c|}\hline
              $n$ & 2& 3 & 4 & 5 & 6 & 7&8\\
\hline                    
\# of simples& 3& 20 & 259 & 6005 & 235,546&14,981,789&$1.494\times10^9$ \\ \hline
\end{tabular}}\medskip
\noindent showing the computational challenge of working in higher rank, even though our algorithm is valid there.\setlength{\tabcolsep}{.5em}  \section*{Acknowledgements}
Many thanks to Jon Brundan for generously sharing computer code that our program was based on, and to Pablo Zadunaisky, Elizaveta Vishnyakova, Walter Mazorchuk and Jonas Hartwig for helpful discussions about their work.  

\section{Background}
\label{sec:background}

\subsection{Gelfand-Tsetlin modules}
\label{sec:GT-modules}
Throughout, we let $Z_k=Z(\mcU(\gl_k))$ denote the center of $\mcU(\gl_k)$; we can think of $Z_k$ for $k=1,\dots, n-1$ as a commutative (but not central) subalgebra of $\mcU(\gl_n)$ via embedding a $k\times k$ matrix in the upper left corner of an $n\times n$-matrix. The Capelli determinant gives free generators for $Z(\mcU(\gl_n))$ in terms of the generators $e_{ij}$ of $\mcU(\gl_n)$.
    The {\bf Gelfand-Tsetlin subalgebra} $\Gamma$ is a maximal commutative subalgebra of $\mcU(\gl_n)$ and is given by $\Gamma = \langle Z_1, \ldots, Z_n\rangle$, 

\begin{definition}
  A {\bf Gelfand-Tsetlin module} over $\mcU(\gl_n)$ is a module $V$ where the action of $\Gamma$ is locally finite: for any $v\in V$, the subspace $\Gamma\cdot v$ is finite-dimensional.  

  More generally, a topological module over $\mcU$ (where $\mcU$ has the discrete topology) is called {\bf pro-Gelfand-Tsetlin } if the action is locally pro-finite: for any $v\in V$, the subspace $\Gamma\cdot v$ is pro-finite dimensional.  
\end{definition}
For readers who dislike pro-finite-dimensional vector spaces: the only example we will need is any finitely generated module over a power-series ring in finitely many variables.

Given $\gamma\in \MaxSpec{\Gamma}$, we can define the corresponding (generalized) weight space:
\[\Wei_{\gamma}(V)=\{v\in V|\mathsf{m}_\gamma^Nv=0\text{ for } N\gg 0\}.\] 
To avoid any confusion with weight spaces for the Cartan subalgebra $\mathfrak{h}$, we call the spaces $\Wei_{\gamma}(V)$ {\bf Gelfand-Tsetlin (GT) weight spaces} and those for the Cartan $\mathfrak{h}$-weight spaces.

The restriction of $\gamma$ to $Z_n$ gives a central character $\chi=\gamma|_{Z_n}$, and it will often be important to consider the subsets $\MaxSpec{\Gamma}_{\chi}$ and $\MaxSpec{\Gamma}_{\Z,\chi}$ of all/integral maximal ideals with this restriction fixed to be $\chi$.

\subsubsection{OGZ algebras} For purposes of computation, it will be essential for us to use the fact that $\mcU$ has a polynomial representation on $\Gamma$ itself. First, we embed $\Gamma$ into the polynomial ring $\Lambda$ with generators $\la_{ij}$ for $1\leq j \leq i\leq n$ using the shifted Harish-Chandra homomorphism. Let $K,L$ be the fraction fields of $\Gamma, \Lambda$, respectively.

That is, we send $\gamma\in Z_k$ to the polynomial $f_\gamma(\lambda_{k1},\dots,\lambda_{kk})$ given by the central character of the Verma module with highest weight $(\lambda_1 ,\lambda_2+1,\dots, \lambda_k+(k-1))$ for the standard upper-triangular Borel of $\mathfrak{gl}_k$.  
One can also interpret this isomorphism as formally adjoining the $\lambda_{k*}$ as the roots of the Capelli determinant; indeed, if we consider the column-determinant 
\[C_k(z)=\operatorname{cdet}(A)=\sum_{\sigma\in S_k}(-1)^{\sigma}a_{\sigma(1)1}\cdots a_{\sigma(k)k}\]\[ A=\begin{bmatrix}
e_{11}-z & e_{12} & \cdots & e_{1k}\\
e_{21} & e_{22} -1-z & \cdots &e_{2k}\\
\vdots & \vdots &\ddots & \vdots\\
e_{k1} & e_{k2} & \cdots & e_{kk}-k+1-z
\end{bmatrix}\]
then $C(z)$ is sent under this map to the polynomial $(\la_{k1}-z)\cdots (\la_{kk}-z)$ since all the terms corresponding to non-trivial permutations act trivially on the highest weight of a Verma module.
Thus, the image of $\Gamma$ is the invariant polynomials under the action of $W=S_1\times S_2\times \cdots \times S_{n}$ on $\Lambda$.

The maximal ideal $\mathsf{m}_{\gamma}$ lies under precisely the maximal ideals in a $W$-orbit in $\MaxSpec(\Lambda)$ of $W$.  It will often be useful to choose a point $\Ba=(a_{ij})\in \MaxSpec(\Lambda)\cong \bbC^1\times \bbC^2\times \cdots \bbC^n$ for $1\leq j\leq i\leq n$.  In particular, we will call $\gamma$ integral if $a_{ij}\in \Z$ for all $i,j$, and let $\MaxSpec{\Gamma}_{\Z}$ denote the set of integral maximal ideals.

By \cite[Thm. 1]{mazorchukOGZ}, we have:
\begin{prop}\label{prop:mazorchuk}
    Let $\delta_{ij}^\pm : L \to L$ be the map which sends $\lambda_{ij} \to \lambda_{ij} \mp 1$. Let $X_{i}^\pm : {L} \to {L}$ be the map
    \begin{align}\label{eq:Mazorchuk}
        X_{i}^\pm =\sum_{j=1}^{i}\mp \frac{\prod_{m}(\lambda_{i\pm 1 ,m} - \lambda_{ij} )}{\prod_{m \neq j}(\lambda_{im} - \lambda_{ij} ) } \delta_{ij}^\pm.
    \end{align} 
     The map sending 
     \[e_{i,i} \to \sum_{k=1}^i\lambda_{ik}-\sum_{\ell=1}^{i-1}\lambda_{i-1,\ell}+(i-1)\qquad  e_{i, i+1} \to X_{i}^+\qquad  e_{i+1, i} \to X_{i}^{-}\] 
     defines a faithful left $\mcU(\gl_n)$-module structure on $L$ which preserves the subrings $K$ and $\Gamma$. 
\end{prop}
When $n=3$ the formulas simplify to
\begin{align}\label{u_on_polys:start}
    e_{11} &= \lambda_{11}\\
    e_{12} &= (\lambda_{11} - \lambda_{21})(\lambda_{22}-\lambda_{11})\delta_{11}^+\\
    e_{21} &= \delta_{11}^-\\\label{u_on_polys:end}
    e_{23} &= \frac{ (\lambda_{22} - \lambda_{31})(\lambda_{22} - \lambda_{32})(\lambda_{22} - \lambda_{33})}{\lambda_{21}-\lambda_{22}}\delta_{22}^+ \nonumber \\
           &\quad - \frac{ (\lambda_{21} - \lambda_{31})(\lambda_{21} - \lambda_{32})(\lambda_{21} - \lambda_{33})}{\lambda_{21} - \lambda_{22}}\delta_{21}^+ \\
    e_{32} &= \frac{ (\lambda_{11}-\lambda_{22})}{\lambda_{21}-\lambda_{22}}\delta_{22}^- - \frac{ (\lambda_{11}-\lambda_{21})}{\lambda_{21}-\lambda_{22}}\delta_{21}^-.
\end{align}
This perspective on $\mcU(\gl_n)$ suggests a more general definition where we vary the number of variables in each alphabet in the formula \eqref{eq:Mazorchuk}.  That is, consider $\Bv=(v_1,\dots, v_n)\in \Z_{\geq 0}^n$, and let $\Omega=\{(i,j)\mid 1\leq i\leq n, 1\leq j\leq v_i\}$.  We can then let $\Lambda_{\Bv}$ be the polynomial ring with variables indexed by $\Omega$, and $\Gamma_{\Bv}$ the invariant subring in $\Lambda$ under the natural action of $S_{v_1}\times \cdots \times S_{v_n}$.  We can then consider analogues of the operators \eqref{eq:Mazorchuk} on this polynomial ring:
    \begin{align}\label{eq:Mazorchuk2}
        X_{i}^\pm =\sum_{j=1}^{v_i}\mp \frac{\prod_{m}(\lambda_{i\pm 1 ,m} - \lambda_{ij} )}{\prod_{m \neq j}(\lambda_{im} - \lambda_{ij} ) } \delta_{ij}^\pm.
    \end{align} 
\begin{definition}
  The orthogonal Gelfand-Zetlin\footnote{We know this is an inconsistent spelling of \foreignlanguage{russian}{Цетлин}, but we are keeping consistent with the spelling of \cite{mazorchukOGZ}, since ``OGZ'' is an established acronym.} (OGZ) algebra $\mcU_{\Bv}$ is the algebra of operators on $\Gamma_{\Bv}$ generated by $X_{i}^\pm$ as in \eqref{eq:Mazorchuk2} and multiplication by $\Gamma_{\Bv}$.  
\end{definition}

\subsubsection{Completion and Gelfand-Tsetlin modules}
In the work of Drozd, Futorny and Ovsienko \cite{FOD}, the authors consider a category $\Cat$ whose objects are maximal ideals in $\MaxSpec(\Gamma)$, and whose morphism spaces are defined by 
\[\Hom_{\Cat}(\gamma,\gamma')=\varprojlim \mcU /(\mathsf{m}_{\gamma'}^N\mcU+\mcU\mathsf{m}_{\gamma}^N).\]
These are the same as the space of natural transformations of the weight functors $\Wei_{\gamma}\to \Wei_{\gamma'}.$ Note that this is a category enriched in topological abelian groups.

Recall that a discrete representation of $\Cat$ is a functor to the category of vector spaces where the induced maps on morphisms are continuous, with all vector spaces given the discrete topology.   
By \cite[Thm. 17]{FOD}, there is an equivalence $\mathsf{DFO}$ of the category of Gelfand-Tsetlin modules over $\mcU$ to the discrete representations of $\Cat$, which sends the module $M$ to the representation $\gamma\mapsto \Wei_{\gamma}(M)$.

 Of course, $\Hom_{\Cat}(\gamma,\gamma')=0$ unless $\gamma|_{Z_n}=\gamma'|_{Z_n}$.  Let $\Cat_{\chi}$ and $\Cat_{\Z,\chi}$ be the full subcategories with object set $\MaxSpec(\Gamma)_{\chi}$ and $\MaxSpec(\Gamma)_{\Z,\chi}$ respectively.
 Let $\GTc_{\Z,\chi}$ be the category of Gelfand-Tsetlin modules where $\Wei_{\gamma}(V)\neq 0$ only for $\gamma\in \MaxSpec_{\Z,\chi}(\Gamma)$.  

One can easily extend \cite[Thm. 17]{FOD} to show that:
\begin{proposition}\label{prop:DFO}
  The category $\GTc_{\Z,\chi}$ is equivalent to the category of discrete $\Cat_{\Z,\chi}$-modules, via the functor $\mathsf{DFO}$.
\end{proposition}
The category $\Cat$ does not act on $\Gamma$, but it has a representation $\mathscr{P}$ which sends $\gamma\mapsto \widehat{\Gamma}_{\gamma}$, the completion \[\widehat{\Gamma}_{\gamma}=\varprojlim {\Gamma}/\mathsf{m}_{\gamma}^N.\]
Note that unlike the representations considered earlier, this is not continuous in the discrete topology, but will be in the inverse limit topology on $\widehat{\Gamma}_{\gamma}$.  This is very closely related to the representation $\mathfrak{O}^G$ considered in \cite[\S 3.1]{MVHC}, which is the $\mcU$-module obtained by summing these completions (though they use germs of functions instead of completions; the difference is not relevant for us).

Let us try to carefully describe this action.  Let $\mathbf{a}\in \MaxSpec(\Lambda)$ be a maximal ideal lying over $\mathsf{m}_{\gamma}$ and $W_{\mathbf{a}}$ the stabilizer of $\mathbf{a}$ in $W$.  Since the map $\Gamma\to \Lambda$ induces a covering on spectra (i.e. an \'etale map of schemes), we have that $\widehat{\Gamma}_{\gamma}=\widehat{\Lambda}_{\mathbf{a}}^{W_{\mathbf{a}}}$.  Alternatively, if we sum over this orbit, we see that \begin{equation*}\label{eq:orbit-sum}
\widehat{\Gamma}_{\gamma}=\Big(\bigoplus_{\mathbf{a}'\in W\cdot \mathbf{a}}\widehat{\Lambda}_{\mathbf{a}'}\Big)^{W}.
\end{equation*}  We can think of $\widehat{\Lambda}_{\mathbf{a}}$ as the power series ring $\bbC[[\lambda_{ij}-a_{ij}]]$ for scalars $a_{ij}$, defined by Taylor expansion, so $\widehat{\Gamma}_{\gamma}$ is the set of invariant power series under $W_{\mathbf{a}}$.  Acting by $W$, we can also realize an element of $\widehat{\Gamma}_{\gamma}$ as a choice of power series based at each point in the orbit which is compatible with the action of $W$.

Thus, the operation $\delta_{ij}^\pm $ descends to an isomorphism $\widehat{\Lambda}_{\mathbf{a}}\to \widehat{\Lambda}_{\mathbf{a}\pm \mathbf{e}_{ij}}$ for $\mathbf{e}_{ij}$ the obvious unit vector.  We can interpret the element $X_{i}^\pm$ as a map \[X_{i}^\pm\colon \widehat{\Lambda}_{\mathbf{a}}\to \bigoplus_{j}\widehat{\Lambda}_{\mathbf{a}\pm  \mathbf{e}_{ij}}.\]
In order to induce the correct action on $\widehat{\Gamma}_{\gamma},$ we need to consider this map on orbit sums. Let $U_{\pm i,{\mathbf{a}}}=\cup_{j} W\cdot (\mathbf{a}\pm  \mathbf{e}_{ij})$.  From the formulas, it is clear that $X_{i}^\pm$ induces a map
\[X_{i}^\pm\colon \bigoplus_{\mathbf{a}'\in W\cdot \mathbf{a}}\widehat{\Lambda}_{\mathbf{a}'} \to \bigoplus_{\mathbf{a}''\in U_{\pm i,{\mathbf{a}}}}\widehat{\Lambda}_{\mathbf{a}''}.\]

The desired map on $\widehat{\Gamma}_{\gamma}$ is the restriction of this $W$-equivariant map to the invariant subspaces. The image $X_{i}^\pm(\gamma,\gamma')$ of $X_{i}^\pm$ in $\Hom_{\Cat}(\gamma,\gamma')$ is then the map just described, followed by projection to $\widehat{\Gamma}_{\gamma'}$.  In particular, note that this image is 0 unless $\gamma'$ lies under one of the elements of $U_{\pm i,{\mathbf{a}}}$.  

Since the action on $\Gamma$ is faithful and $\Gamma$ injects into any of its completions,   we can identify morphisms in $\Cat$ by their action on this module.

\subsection{KLRW algebras}
\label{sec:klrw-algebras}

We recall now the second author's construction of the \textbf{KLRW
  algebra} $\tT$.  We will follow the
conventions of \cite{websterThreePerspectives2020}: This algebra is defined by {\bf Stendhal diagrams} consisting of finitely many oriented curves that carry dots in $\R\times [0,1]$ whose projection to the second factor is a diffeomorphism.
  
  \begin{remark}
  To avoid any confusion, let us note that following \cite{websterThreePerspectives2020} means that unlike some earlier papers (for example, \cite{Webmerged}), we allow dots on red strands in addition to black strands.  These are needed to account for modules where the center $Z_n$ acts non-semi-simply.   
  \end{remark}

The diagrams are considered up to isotopy and must be locally of the form
\begin{equation*}
\begin{tikzpicture}
\draw[very thick] (0,0) +(-1,-1) -- +(1,1);
\draw[very thick](0,0) +(1,-1) -- +(-1,1);

  \draw[very thick](3,0) +(-1,-1) -- +(1,1);
\draw[wei, very thick](3,0) +(1,-1) -- +(-1,1);

  \draw[wei,very thick](6,0) +(-1,-1) -- +(1,1);
\draw [very thick](6,0) +(1,-1) -- +(-1,1);

 \draw[very thick](-5,0) +(0,-1) --  +(0,1);

  \draw[very thick](-4,0) +(0,-1) --  node
  [midway,circle,fill=black,inner sep=2pt]{}
  +(0,1);

   \draw[wei,very thick](-3,0) +(0,-1) --  +(0,1);
   
     \draw[wei, very thick](-2,0) +(0,-1) --  node
  [midway,circle,fill=red,inner sep=3pt]{}
  +(0,1);
\end{tikzpicture}
\end{equation*}
We'll adopt the convention (similar to that of \cite{websterThreePerspectives2020}) of labeling black strands with elements of the set $[1,n-1]$ and all red strands with $n$; this is not the most general algebra of this type, but it will suffice for us.  

We call the lines $y=0,1$ the {\bf  bottom} and {\bf top} of the
diagram.  Reading across the bottom and top from left to right, we
obtain a sequence $\Bi=(i_1, \dots, i_V)$ of elements of $[1,n]$
that label both red and black strands, where $V$ is the total number of
strands.

 Let $e(\Bi)$ be the unique
crossingless, dotless  diagram where the sequence at top and bottom are both $\Bi$.    

\begin{definition}  The {\bf degree} of a Stendhal diagram is the sum over
  crossings and dots in the diagram of 
  \begin{itemize}
  \item $-2$ for each crossing  where the strands both have labels $i$
  \item $1$ for each crossing where the strands both have labels $i$ and $j=i\pm 1$
  \item $0$ for each crossing where the strands both have labels $i$ and $j\notin \{i,i\pm 1\}$
  \item $2$ for each dot;
  \end{itemize}
The degree of diagrams is additive under composition.  Thus, the
algebra $\tT$ inherits a grading from this degree function.
\end{definition}

Consider the set $\Omega=\{(i,j) \mid i\in [1,n], j\in [1,v_i]\}$.  
Let $\prec$ be a total order on $\Omega$ such that 
\begin{equation}
    (i,1)\preceq \cdots \preceq (i,v_i). \label{eq:order-1}
\end{equation}  
This is equivalent to choosing a word $\Bi =(i_1,\dots, i_{N})$ where where $N=|\Omega|$ and $i_k=i$ for $v_i$ different indices $k$.  

We will want to weaken this definition a bit and allow $\preceq$ to be a total preorder (that is, a relation which is transitive and reflexive, but not necessarily anti-symmetric).  In this case, we have an induced equivalence relation $(i,k)\approx (j,\ell)$ if $(i,k)\preceq (j,\ell)$ and $(i,k)\succeq (j,\ell)$.  We assume that our preorder satisfies the condition that
\begin{equation}
    (i,k)\not\approx (j,\ell)\text{ whenever } i\neq j.\label{eq:order-2}
\end{equation}
We can still attach a word $\Bi$ to such a preorder; two equivalent elements give the same letter in the word $\Bi$, so it doesn't matter whether they have a chosen order.  We can thus think of a preorder as corresponding to a word in the generators with some subsets where the same letter appears multiple times together grouped together.  We can represent this within the word itself by replacing $(i,\dots, i)$ with $i^{(a)}$.  Thus, for our purposes, the sequences $(3,2,2,3,1,3)$ and $(3,2^{(2)},3,1,3)$ are different words with different associated preorders.  Every such word has a unique {\bf totalization} satisfying \eqref{eq:order-1}. 
We let $\operatorname{tSeq}(\Bv)$ be the resulting set of {\bf pre-sequences}.

\begin{definition}
  Let $\tT$ be the quotient of the formal span of Stendhal diagrams by the following local relations between Stendhal diagrams.  We draw these below as black, but the same relations apply to red strands (always taken with the label $n$):
\begin{equation*}\subeqn\label{first-QH}
    \begin{tikzpicture}[scale=.7,baseline]
      \draw[very thick](-4,0) +(-1,-1) -- +(1,1) node[below,at start]
      {$i$}; \draw[very thick](-4,0) +(1,-1) -- +(-1,1) node[below,at
      start] {$j$}; \fill (-4.5,.5) circle (4pt);
\node at (-2,0){=}; \draw[very thick](0,0) +(-1,-1) -- +(1,1)
      node[below,at start] {$i$}; \draw[very thick](0,0) +(1,-1) --
      +(-1,1) node[below,at start] {$j$}; \fill (.5,-.5) circle (4pt);
   \end{tikzpicture}
 \qquad 
    \begin{tikzpicture}[scale=.7,baseline]
      \draw[very thick](-4,0) +(-1,-1) -- +(1,1) node[below,at start]
      {$i$}; \draw[very thick](-4,0) +(1,-1) -- +(-1,1) node[below,at
      start] {$j$}; \fill (-3.5,.5) circle (4pt);
\node at (-2,0){=}; \draw[very thick](0,0) +(-1,-1) -- +(1,1)
      node[below,at start] {$i$}; \draw[very thick](0,0) +(1,-1) --
      +(-1,1) node[below,at start] {$j$}; \fill (-.5,-.5) circle (4pt);
      \node at (3.5,0){unless $i=j$};
    \end{tikzpicture}
  \end{equation*}
\begin{equation*}\subeqn\label{nilHecke-1}
    \begin{tikzpicture}[scale=.8,baseline]
      \draw[very thick](-4,0) +(-1,-1) -- +(1,1) node[below,at start]
      {$i$}; \draw[very thick](-4,0) +(1,-1) -- +(-1,1) node[below,at
      start] {$i$}; \fill (-4.5,.5) circle (4pt);
\node at (-2,0){$-$}; \draw[very thick](0,0) +(-1,-1) -- +(1,1)
      node[below,at start] {$i$}; \draw[very thick](0,0) +(1,-1) --
      +(-1,1) node[below,at start] {$i$}; \fill (.5,-.5) circle (4pt);
      \node at (1.8,0){$=$}; 
    \end{tikzpicture}\,\,
\begin{tikzpicture}[scale=.8,baseline]
      \draw[very thick](-4,0) +(-1,-1) -- +(1,1) node[below,at start]
      {$i$}; \draw[very thick](-4,0) +(1,-1) -- +(-1,1) node[below,at
      start] {$i$}; \fill (-4.5,-.5) circle (4pt);
\node at (-2,0){$-$}; \draw[very thick](0,0) +(-1,-1) -- +(1,1)
      node[below,at start] {$i$}; \draw[very thick](0,0) +(1,-1) --
      +(-1,1) node[below,at start] {$i$}; \fill (.5,.5) circle (4pt);
      \node at (2,0){$=$}; \draw[very thick](4,0) +(-1,-1) -- +(-1,1)
      node[below,at start] {$i$}; \draw[very thick](4,0) +(0,-1) --
      +(0,1) node[below,at start] {$i$};
    \end{tikzpicture}
  \end{equation*}
\begin{equation*}\subeqn\label{black-bigon}
    \begin{tikzpicture}[very thick,scale=.8,baseline]
      \draw (-2.8,0) +(0,-1) .. controls (-1.2,0) ..  +(0,1)
      node[below,at start]{$i$}; \draw (-1.2,0) +(0,-1) .. controls
      (-2.8,0) ..  +(0,1) node[below,at start]{$j$}; 
   \end{tikzpicture}=\quad
   \begin{cases}
0 & i=j\\
     \begin{tikzpicture}[very thick,yscale=.6,xscale=.8,baseline=-3pt]
       \draw (2,0) +(0,-1) -- +(0,1) node[below,at start]{$j$};
       \draw(1,0) +(0,-1) -- +(0,1) node[below,at start]{$i$};
     \end{tikzpicture} & i\neq j,j\pm 1\\
   \begin{tikzpicture}[very thick,yscale=.6,xscale=.8,baseline=-3pt]
       \draw(2,0) +(0,-1) -- +(0,1) node[below,at start]{$j$};
       \draw (1,0) +(0,-1) -- +(0,1) node[below,at start]{$i$};\fill (2,0) circle (4pt);
     \end{tikzpicture}-\begin{tikzpicture}[very thick,yscale=.6,xscale=.8,baseline=-3pt]
       \draw (2,0) +(0,-1) -- +(0,1) node[below,at start]{$j$};
       \draw (1,0) +(0,-1) -- +(0,1) node[below,at start]{$i$};\fill (1,0) circle (4pt);
     \end{tikzpicture}& i=j-1\\
  \begin{tikzpicture}[very thick,baseline=-3pt,yscale=.6,xscale=.8]
       \draw (2,0) +(0,-1) -- +(0,1) node[below,at start]{$j$};
       \draw(1,0) +(0,-1) -- +(0,1) node[below,at start]{$i$};\fill (1,0) circle (4pt);
     \end{tikzpicture}-\begin{tikzpicture}[very thick,yscale=.6,xscale=.8,baseline=-3pt]
       \draw (2,0) +(0,-1) -- +(0,1) node[below,at start]{$j$};
       \draw (1,0) +(0,-1) -- +(0,1) node[below,at start]{$i$};\fill (2,0) circle (4pt);
     \end{tikzpicture}& i=j+1
   \end{cases}
  \end{equation*}
 \begin{equation*}\subeqn\label{triple-dumb}
    \begin{tikzpicture}[very thick,scale=.8,baseline=-3pt]
      \draw (-2,0) +(1,-1) -- +(-1,1) node[below,at start]{$k$}; \draw
      (-2,0) +(-1,-1) -- +(1,1) node[below,at start]{$i$}; \draw
      (-2,0) +(0,-1) .. controls (-3,0) ..  +(0,1) node[below,at
      start]{$j$}; \node at (-.5,0) {$-$}; \draw (1,0) +(1,-1) -- +(-1,1)
      node[below,at start]{$k$}; \draw (1,0) +(-1,-1) -- +(1,1)
      node[below,at start]{$i$}; \draw (1,0) +(0,-1) .. controls
      (2,0) ..  +(0,1) node[below,at start]{$j$}; \end{tikzpicture}=\quad
      \begin{cases} 
    \begin{tikzpicture}[very thick,yscale=.6,xscale=.8,baseline=-3pt]
     \draw (6.2,0)
      +(1,-1) -- +(1,1) node[below,at start]{$k$}; \draw (6.2,0)
      +(-1,-1) -- +(-1,1) node[below,at start]{$i$}; \draw (6.2,0)
      +(0,-1) -- +(0,1) node[below,at
      start]{$j$};     \end{tikzpicture}& i=k=j+1\\
    -\begin{tikzpicture}[very thick,yscale=.6,xscale=.8,baseline]
     \draw (6.2,0)
      +(1,-1) -- +(1,1) node[below,at start]{$k$}; \draw (6.2,0)
      +(-1,-1) -- +(-1,1) node[below,at start]{$i$}; \draw (6.2,0)
      +(0,-1) -- +(0,1) node[below,at
      start]{$j$};     \end{tikzpicture}& i=k=j-1\\
     0& \text{otherwise}
      \end{cases}
  \end{equation*}
\end{definition}

Note that this algebra breaks up into a sum of subalgebras $\tT_{\Bv}$
where we
fix the number of strands with label $i$ to be $v_i$.  Following
Khovanov and Lauda \cite[\S 2.1]{KLI}, we let $\Seq(\Bv)$ be the set
of sequences in $[1,n]$ where $i$ appears $v_i$ times.

We'll also want to consider the completion of this algebra with respect to its grading.  By \cite[Lem. 2.3]{WebBKnote}, this is equivalent to simply considering the same diagrammatics, but allowing formal power series in the subring of dots.  One fact that will be useful for us: a topological module over $\widehat{\tT}$ which has discrete topology is thus one where dots of high degree act trivially. 

On the other hand, we can consider $\tT$-modules which are {\bf weakly graded}, that is, they possess a filtration such that the subquotients are gradeable. Let $\tT\wgmod$ be the category of finite-dimensional weakly graded $\tT$-modules 
\begin{lemma}\label{lem:wgmod}
  A $\tT$-module $M$ is finite dimensional and weakly graded if and only if it is the restriction to $\tT$ of a topological $\widehat{\tT}$-module which is finitely generated and discrete.
\end{lemma}
\begin{proof}
Since $M$ is finite dimensional and weakly graded, all dots act nilpotently on it, and thus, $M$ is killed by the $N$th power of the unique graded maximal ideal of the ring of dots.  The two-sided ideals generated by the $n$th powers of this ideal for $n\geq N$ form a basis of the neighborhoods of 0 in the grading topology on $\tT$, and so the fact that they act trivially shows that we obtain a continuous action of $\widehat{\tT}$ on $M$ equipped with the discrete topology, which is, of course, finitely generated by a basis of $M$.

If $M$ is a topological $\widehat{\tT}$-module which is finitely generated and discrete, then by discreteness, it must also be finitely generated as a $\tT$-module.  Furthermore, any vector must be killed by a power of the unique graded maximal ideal of the ring of dots, since these powers generate two-sided ideals, which are a basis of neighborhoods of 0, as discussed above.  Since the quotient $R$ of $\tT$ by this ideal is finite-dimensional, $M$ is finite-dimensional.  In particular, it contains a simple submodule, generated by a single vector; this submodule is gradable, since it is a simple over the finite-dimensional graded ring $R$.  Modding out by this simple submodule gives a module of smaller dimension, which is still discrete.  By induction on the dimension of $M$, this shows that $M$ has a composition series of gradeable simple modules.  
\end{proof}

\subsubsection{Polynomial representation} From its realization as a weighted KLR algebra, the algebra $\tT$ inherits a polynomial representation. 
\begin{definition}\label{def:poly-rep}
  The polynomial representation of $\tT$  is the vector space 
  \begin{equation*}
      \poly=\bigoplus_{\Bi}\bbC[Y_1, \dots, Y_{V}]e({\Bi}),
  \end{equation*}  with sum running over total orders on $\Omega$ satisfying \eqref{eq:order-1}.  The action is given by the rules:
\begin{itemize}
    \item $e({\Bi})$ acts by projection to the corresponding summand,
    \item a dot on the $k$th strand from the left acts by multiplication by $Y_k$,
    \item a crossing of the $k$th and $k+1$st strands with $\Bi$ at the bottom and $\Bi'$ at top acts by
    \begin{itemize}
        \item  If $i_k=i_{k+1}$, the divided difference operator \[fe_{\Bi}\mapsto \frac{f^{(k,k+1)}-f}{Y_{k+1}-Y_k}e_{\Bi'}.\]
        \item If $i_k+1=i_{k+1}$, the permutation $(k,k+1)$ followed by a multiplication \[fe_{\Bi}\mapsto (Y_{k}-Y_{k+1}) f^{(k,k+1)}e_{\Bi'}.\]    
        \item Otherwise, the permutation $(k,k+1)$  \[fe_{\Bi}\mapsto f^{(k,k+1)}e_{\Bi'}.\]
    \end{itemize}
\end{itemize}
\end{definition}
As noted below the proof of Theorem 2.8 in \cite{WebwKLR}, this representation is faithful.  
The same formulas define a representation $\widehat{\poly}$ of $\widehat{\tT}$ on the corresponding power series rings, which \cite[Lem. 2.3]{WebBKnote} shows is also faithful.  For later reference, we state this as a lemma:
\begin{proposition}\label{lem:faithful}
  The representations $\poly$ and $\widehat{\poly}$ are faithful over $\tT$ and $\widehat{\tT}$ respectively.   
\end{proposition}

\subsubsection{Induction} Note that if $\Bv_1,\dots, \Bv_p$ are dimension vectors such that
$\Bv=\Bv_1+\cdots+ \Bv_p$, then we have an obvious map of horizontal
composition $\tT_{\Bv_1}\otimes \cdots \otimes \tT_{\Bv_p}\to
\tT_{\Bv}$.  Note that this is not a unital map, since not all
idempotents $e(\Bi)$ are in the image.
Following Khovanov and Lauda \cite{KLI}, we define the {\bf induction functor} for $L_k$ a $\tT_{\Bv_k}$-module for each $k$:
\begin{equation}
  \label{eq:induction}
  L_1\circ \cdots \circ L_p=\tT_{\Bv}\otimes_{\tT_{\Bv_1}\otimes
    \cdots \otimes \tT_{\Bv_p}} (L_1\boxtimes \cdots \boxtimes L_p).
  \end{equation}
  By an obvious extension of \cite[Lem. 2.16]{KLI} using the basis of \cite[Prop. 4.16]{Webmerged}, we have that:
  \begin{lemma}\label{lem:induction-basis}
    The module $L_1\circ \cdots \circ L_p$ has a basis given by the elements of the form $ae(\Bi)\ell$ for $\Bi$ ranging over concatenations of $\Bi_k\in \Seq(\Bv_k)$, the diagram $a$ ranging over longest right coset reps of $S_{|\Bv_1|}\times
    \cdots S_{|\Bv_p|}$ in $S_{|\Bv|}$ with $e(\Bi)$ at the bottom which don't cross any red strands, and $\ell$ ranging over a basis of $e(\Bi)(L_1\boxtimes \cdots \boxtimes L_p)$.  
  \end{lemma}
  
\subsubsection{Thick calculus}\label{sec:thick-calculus} Our work later will be considerably simplified by using ``thick calculus'' in the nilHecke algebra, following the approach of \cite{khovanovExtendedGraphical2012}.  This incorporates the ``full twist'' diagram $D_a$ on $a$ consecutive strands with the same label $i$.  We will use the reflection of the conventions from that paper, so we let $e_a=D_a\delta_a$ for $\delta_a$ a polynomial of degree $\frac{a(a+1)}{2}$ such that $D_a\delta_a D_a=D_a$; in \cite{khovanovExtendedGraphical2012}, they use $y_1^{a-1}y_2^{a-2}\cdots y_{a-1}$, but in the interest of symmetry, we prefer $\delta_a=\frac{1}{a!}\prod_{1\leq i<j\leq a} (y_i-y_j)$.  
A simple calculation shows that in either case, $e_a^2=e_a$, and $e_a$ is the projection whose image is the symmetric polynomials.  For our choice of $\delta_a$, we obtain the unique $S_a$-invariant projection.

Given a pre-sequence $\mathbf{I}\in \operatorname{tSeq}(\Bv)$, we can define an idempotent $e(\mathbf{I})$ by starting with the idempotent $e(\Bi)$ and multiplying by $e_a$ in a block of $a$ strands that are equivalent under $\approx$, that is, on the block of strands corresponding to $i^{(a)}$.  It will also be useful to consider the split and merge diagrams:
\[\tikz[baseline,yscale=1.5]{\draw[thickc] (-1,1) to[out=-90,in=180] node[midway, left]{$a$}node [above, at start]{$i$} (0,0); \draw[thickc] (1,1) to[out=-90,in=0] node[midway, right]{$b$} node [above, at start]{$i$}(0,0); \draw[thickc] (0,0) to[out=-90,in=90] node[midway, left]{$a+b$}node [below, at end]{$i$}(0,-1); }=\tikz[baseline,yscale=1.5]{\draw[very thick] (-1.6,1) to[out=-90,in=90] node[above, at start]{$i$} node[below, at end]{$i$}(-.9,-1); \draw[very thick] (-.9,1) to[out=-90,in=90] node[above, at start]{$i$} node[below, at end]{$i$} (-.2,-1);\draw[very thick] (1.6,1) to[out=-90,in=90] node[above, at start]{$i$} node[below, at end]{$i$}(.9,-1); \draw[very thick] (.9,1) to[out=-90,in=90] node[above, at start]{$i$} node[below, at end]{$i$}(.2,-1); \node at (1.25,.8) {$\cdots$};\node at (-1.2,.8) {$\cdots$};\node at (.55,-.8) {$\cdots$};\node at (-.55,-.8) {$\cdots$};\node[draw,very thick,fill=white,inner xsep=30pt] at (0,-.4){$e_{a+b}$}; \draw [decorate,decoration={brace,amplitude=5pt,mirror,raise=1.5ex}]
  (-1,-1.2) -- (1,-1.2) node[midway,yshift=-2em]{$a+b$ strands};\draw [decorate,decoration={brace,amplitude=5pt,mirror,raise=1.5ex}]
  (-.8,1.2) -- (-1.6,1.2) node[midway,yshift=2em]{$a$ strands};\draw [decorate,decoration={brace,amplitude=5pt,mirror,raise=1.5ex}]
  (1.6,1.2) -- (.8,1.2) node[midway,yshift=2em]{$b$ strands};}\qquad \qquad \tikz[baseline,yscale=1.5]{\draw[thickc] (-1,-1) to[out=90,in=180] node[midway, left]{$a$}node [below, at start]{$i$} (0,0); \draw[thickc] (1,-1) to[out=90,in=0] node[midway, right]{$b$}node [below, at start]{$i$}(0,0); \draw[thickc] (0,0) to[out=90,in=-90] node[midway, left]{$a+b$}node [above, at end]{$i$}(0,1); }=\tikz[baseline,yscale=1.5]{\draw[very thick] (.2,1) to[out=-90,in=90] node[above, at start]{$i$} node[below, at end]{$i$}(-1.5,-1); \draw[very thick] (.8,1) to[out=-90,in=90] node[above, at start]{$i$} node[below, at end]{$i$} (-.9,-1);\draw[very thick] (-.2,1) to[out=-90,in=90] node[above, at start]{$i$} node[below, at end]{$i$}(1.5,-1); \draw[very thick] (-.8,1) to[out=-90,in=90] node[above, at start]{$i$} node[below, at end]{$i$}(.9,-1); \node at (1.2,-.8) {$\cdots$};\node at (-1.15,-.8) {$\cdots$};\node at (.5,.8) {$\cdots$};\node at (-.45,.8) {$\cdots$}; \node[draw,very thick,fill=white,inner xsep=15pt] at (.9,-.4){$e_{b}$}; \node[draw,very thick,fill=white,inner xsep=15pt] at (-.9,-.4){$e_{a}$}; \draw [decorate,decoration={brace,amplitude=5pt,mirror,raise=1.5ex}]
  (-1.6,-1.2) -- (-.8,-1.2) node[midway,yshift=-2em]{$a$ strands};\draw [decorate,decoration={brace,amplitude=5pt,mirror,raise=1.5ex}]
  (.8,-1.2) -- (1.6,-1.2) node[midway,yshift=-2em]{$b$ strands};\draw [decorate,decoration={brace,amplitude=5pt,mirror,raise=1.5ex}]
  (1,1.2) -- (-1,1.2) node[midway,yshift=2em]{$a+b$ strands};}\]
Note that we only define these here for groups of strands with the same label; the extension of these to other type A cases is discussed in \cite{SWschur}, but that introduces a lot of complications which are not needed here.  We will need to use the notation for crossings of thick strands with different labels:
\[\tikz[baseline,yscale=1.5]{\draw[thickc] (-1,-1) to node[pos=.3, left]{$a$}node [below, at start]{$i$} node [above, at end]{$i$} (1,1); \draw[thickc] (1,-1) to node[pos=.3, right]{$b$}node [below, at start]{$j$} node [above, at end]{$j$} (-1,1);  }=\tikz[baseline,yscale=1.5]{\draw[very thick] (.8,1) to node[above, at start]{$i$} node[below, at end]{$i$}(-1.6,-1); \draw[very thick] (1.5,1) to node[above, at start]{$i$} node[below, at end]{$i$} (-.9,-1);\draw[very thick] (-.8,1) to node[above, at start]{$j$} node[below, at end]{$j$}(1.6,-1); \draw[very thick] (-1.5,1) to node[above, at start]{$j$} node[below, at end]{$j$}(.9,-1); \node at (1,-.8) {$\cdots$};\node at (-.95,-.8) {$\cdots$};\node at (.97,.8) {$\cdots$};\node at (-.92,.8) {$\cdots$}; \node[draw,very thick,fill=white,inner xsep=15pt] at (.8,-.5){$e_{b}$}; \node[draw,very thick,fill=white,inner xsep=15pt] at (-.8,-.5){$e_{a}$}; \draw [decorate,decoration={brace,amplitude=5pt,mirror,raise=1.5ex}]
  (-1.7,-1.2) -- (-.8,-1.2) node[midway,yshift=-2em]{$a$ strands};\draw [decorate,decoration={brace,amplitude=5pt,mirror,raise=1.5ex}]
  (.8,-1.2) -- (1.7,-1.2) node[midway,yshift=-2em]{$b$ strands};\draw [decorate,decoration={brace,amplitude=5pt,mirror,raise=1.5ex}]
  (1.6,1.2) -- (.7,1.2) node[midway,yshift=2em]{$a$ strands};\draw [decorate,decoration={brace,amplitude=5pt,mirror,raise=1.5ex}]
  (-.7,1.2) -- (-1.6,1.2) node[midway,yshift=2em]{$b$ strands};}\]

On the polynomial representation above, an idempotent with thick strands acts by projection to the appropriate subring of symmetric polynomials.  
\begin{lemma}[\mbox{\cite[Prop. 3.4]{SWschur}}]\label{lem:thick-action}
The split diagram above acts by the inclusion of $S_{a+b}$-invariant polynomials into $S_a\times S_b$-invariants.  The merge diagram above acts by the Demazure operator
\[ \sum_{\sigma\in S_{a+b}/(S_a\times S_b)}\sigma\cdot \frac{f}{\prod_{k=1}^a\prod_{\ell=a+1}^{a+b}(y_k-y_\ell)}\]
\end{lemma}

\subsection{Standard modules}
\label{sec:standard-modules}

We will now discuss how to construct simple modules over $\tT$.  As
discussed in \cite[\S 5]{WebGT}, these are indexed by {\bf red-good}
words; we call a word $\Bi$ red-good if there is a simple module $L$ over
$\tT$ such that $\Bi$ is lex-minimal amongst words where $e(\Bi)L$.  By \cite[Th. 5.9]{WebGT}, we have a combinatorial characterization of
these words.  Let $\GL$ be the set of words of the form  $(k,k-1,\cdots, k-p)$ for
$k\leq n-1$, and $0\leq p <k$, and $\GL'$ be the set of words of the form  $(n,n-1,\cdots, n-p)$ for
$0\leq p<n $; as noted in \cite[\S 6.6]{Lecshuf}, these together form
the good Lyndon words of the $A_{n}$ root system in the obvious order
on nodes in the Dynkin diagram (which we identify with $[1,n]$).
\begin{theorem}[\mbox{\cite[Th. 5.9]{WebGT}}]
  A word $\Bi$ is red-good if it is the concatenation
  $\Bi=a_1\cdots a_p b_1\cdots b_{v_n}$ of words for $a_k\in\GL$, and
  $b_k\in \GL'$ satisfying $a_1\leq a_2\leq \cdots \leq a_{p}$ in
  lexicographic order.
\end{theorem}

For the purposes of constructing the characters of simple modules, we
need a slightly different construction of these simples than that
given in \cite{WebGT}.  The approach we'll take is based on the theory of standard modules of Brundan, Kleshchev, and MacNamara \cite{brundanHomologicalProperties2014, mcnamaraFiniteDimensional2015}; we need to modify their approach to account for the presence of red strands.  

\begin{definition}
Consider a red-good word $\Bi$ and let $\Delta(\Bi)$ be the universal module generated by a vector $v$
  satisfying the equations 
  \begin{equation}\label{eq:Delta}
      e(\Bi)v=v\qquad \text{and} \qquad \psi_kv=0\text{ for any $k$ with
  $i_k-1=i_{k+1}$.}
  \end{equation}  
  Let $\bar{\Delta}(\Bi)$ be the quotient of $\Delta(\Bi)$ by the
  further relation that 
  \begin{equation}\label{eq:bar-Delta}
      y_kv=0 \text{ for all $k$.}  
  \end{equation}
By analogy with the modules defined in \cite[(3.5) \& (2.11)]{brundanHomologicalProperties2014}, we call these {\bf standard} and {\bf proper standard} modules.
\end{definition}
Note that despite the names, we do not claim that these algebras have a properly stratified structure or any analogous property; they are ``standard'' only in the sense of being quotients of indecomposable projectives by all maps from projectives lower in an order.

It's easy to check that:
\begin{lemma}\label{lem:delta-induction}
\begin{align*}
  \Delta(\Bi)&=\Delta(a_1)\circ \cdots \circ \Delta(a_p)\circ
  \Delta(b_1)\circ \cdots \circ \Delta(b_{v_n}) \\ \bar{\Delta}(\Bi)&=\bar{\Delta}(a_1)\circ \cdots \circ \bar{\Delta} (a_p)\circ
  \bar{\Delta} (b_1)\circ \cdots \circ \bar{\Delta} (b_{v_n}).
\end{align*}
\end{lemma}  

First, consider the case where $\Bi$ is a single
Lyndon word $a=(k,\dots,p)$ for $1\leq p\leq k\leq n$.   In this case, we have an endomorphism of $\Delta(\Bi)$, induced by sending $v\mapsto y_kv$.  Since in this case
$\psi_k^2=y_k-y_{k+1}$, we find that this is the same endomorphism for
all $k$.  We denote it by $\mathsf{y}$.

Now, consider the case where $\Bi$ is of the form $(a,a)$ for a single
Lyndon word, so $\Delta(\Bi)=\Delta(a)\circ \Delta(a)$.  In this case, $ \Delta(\Bi)$ has two dot
endomorphisms $\mathsf{y}_1, \mathsf{y}_2$ given by multiplication by
a dot on the first or second copy of $a$.  If $a\in \GL'$, then
clearly every element of $e(\Bi) \Delta(\Bi)$ is of the form
$\bbC[\mathsf{y}_1, \mathsf{y}_2]v$; on the other hand, if $a\in \GL$,
then we have a diagram which switches the strands attached to the two
different copies of $a$.  This is a vector $v'$ which satisfies the
same relations as $v$, and thus defines a map $\psi\colon
\Delta(\Bi)\to \Delta(\Bi)$.  We have
\[v'=\tikz[baseline,yscale=1.5]{\draw[very thick] (.5,1) to node[above, at start]{$k$} (-1.5,-1); \draw[very thick] (1.5,1) to node[above, at start]{$p$}  (-.5,-1);\draw[very thick] (-.5,1) to node[above, at start]{$p$} (1.5,-1); \draw[very thick] (-1.5,1) to node[above, at start]{$k$} (.5,-1); \node at (.8,.8) {$\cdots$};\node at (-.8,.8) {$\cdots$};\node at (.7,-.7) {$\cdots$};\node at (-.7,-.7) {$\cdots$}; \node[draw,very thick,fill=white,inner xsep=45pt] at (0,-1){$v$}; }\]
More generally, when $\Bi=a^m$, then we always have an action of
$\bbC[\mathsf{y}_1,\dots, \mathsf{y}_m]$ induced by the action on factors of the induction, and if $a\in \GL$, we also
have an endomorphisms $\psi_1,\dots, \psi_{m-1}$ defined by applying $\psi$ to an adjacent pair of factors.  These define an action of the nilHecke algebra constructed in \cite[Lemma 3.7]{brundanHomologicalProperties2014}.
\begin{proposition}
  If $\Bi=a^m$ for $a\in \GL'$, then
  $\End(\Delta(\Bi))=\bbC[\mathsf{y}_1,\dots, \mathsf{y}_m]$ and
  $\Delta(\Bi)$
  has a unique graded simple quotient.

  If $\Bi=a^m$ for $a\in \GL$, then $\End({\Delta}(\Bi))=NH_m$, and
  $\bar{\Delta}(\Bi)$ is the unique graded simple quotient of $\Delta(\Bi)$ up to
  isomorphism.  
\end{proposition}
\begin{proof}
We prove this by induction on $m$.  If $m=1$, then obviously any diagram in $e(\Bi) \Delta(\Bi) $ is
either a
polynomial multiple of $v$ or it contains a bigon where two
strands cross.  In the latter case, this diagram is 0 in
$\Delta(\Bi)$.  Thus, multiplying by $v$ defines a surjection
$\bbC[\mathsf{y}]\to e(\Bi) \Delta(\Bi)$, and this surjection is split
by the map that maps $f(y_1,\dots, y_r)v\mapsto f(\mathsf{y})$, and
kills all other diagrams.  Since $ \bbC[\mathsf{y}]$ is indecomposable
as a module over itself, this surjection is an isomorphism.

In general, by Lemma \ref{lem:induction-basis}, we have that $e(\Bi)\Delta(\Bi)$ has a basis given by a basis of $\bbC[\mathsf{y}_1,\dots, \mathsf{y}_m]v$, which is isomorphic to $e(\Bi)(\Delta(a)\boxtimes\cdots \boxtimes\Delta(a))$, times the set of shortest coset representatives that cross no red strands.  If $a\in \GL'$, only the identity has this property, and $e(\Bi)\Delta(\Bi)\cong \bbC[\mathsf{y}_1,\dots, \mathsf{y}_m]$ via multiplication on $v$.

If $a\in \GL$, then there is a shortest coset rep for each permutation in $S_m$, given by the action of a composition of $\psi_k$'s acting on $v$.  Thus, in this case we have an isomorphism $e(\Bi)\Delta(\Bi)\cong NH_m$ defined by action on $v$.  

We can readily convert both of these to statements about endomorphisms:  every endomorphism of $\Delta(\Bi)$ sends $v$ to an element of $e(\Bi)\Delta(\Bi)$, and since $v$ is a generator, this is an injective map $\End(\Delta(\Bi))\to e(\Bi)\Delta(\Bi)$, which we have already checked is surjective, and thus an isomorphism.

Finally, note that for any simple quotient $L$ of $\Delta(\Bi)$, we must have a surjective, non-zero map $\End(\Delta(\Bi))\cong e(\Bi)\Delta(\Bi)\to e(\Bi)L$ which is a left module homomorphism.  In fact, $e(\Bi)L$ must be a simple quotient of $\End(\Delta(\Bi))$ as a left module over itself.  The uniqueness of this quotient up to isomorphism follows from the fact that $\End(\Delta(\Bi))$ is a matrix ring over a polynomial ring (which is graded local).  Note that $\bar{\Delta}(\Bi)$ corresponds to the quotient of $\End(\Delta(\Bi))$ by the left ideal generated by $\mathsf{y}_1,\dots, \mathsf{y}_m$; in either case, this gives an irrep over $\End(\Delta(\Bi))$. 
By the usual argument (see \cite[Lemma 5.9]{Webmerged}), the module  $\bar{\Delta}(\Bi)$ has a unique simple quotient.  
Finally note that if $a\in \GL$, then any diagram in the basis of $\bar{\Delta}(\Bi)$ corresponds to a right coset representative which can be multiplied by a diagram on the left to give a non-zero element of $NH_m$.  This shows that $\bar{\Delta}(\Bi)$ has no proper non-zero submodules in this case.
\end{proof}

\begin{theorem}\label{thm:all-simples}\hfill
  \begin{enumerate}
      \item For any red-good word $\Bi$, the standard module $\bar{\Delta}(\Bi)$ has a unique simple quotient $L(\Bi)$.
      \item As $\Bi$ ranges over red-good words, these quotients give a complete irredundant list of the simple graded $\tT$-modules.
      \item  The composition factors of the kernel of the map $\bar{\Delta}(\Bi)\to L(\Bi)$ are all of the form $L(\Bj)$ for $\Bj >\Bi$ in lexicographic order.
  \end{enumerate}
\end{theorem}
\begin{proof}\hfill
\begin{enumerate}[wide]
    \item We can write \[\bar{\Delta}(\Bi)=\bar{\Delta}(a_1^{m_1})\circ \cdots \circ \bar{\Delta} (a_p^{m_p})\circ
  \bar{\Delta} (b_1)\circ \cdots \circ \bar{\Delta} (b_{v_n})\] where we assume that $a_1<a_2<\cdots <a_p$ enumerate all elements of $\GL$, and $m_i\geq 0$.  It's an easy combinatorial result (see \cite[Lem. 15]{Lecshuf}) that every non-trivial shuffle of these words is above $\Bi$ in lex order, so $e(\Bi)\bar{\Delta}(\Bi)\cong \overline{NH}_{m_1}\otimes \cdots \otimes \overline{NH}_{m_p}$ as a module over the tensor product over nilHecke algebras.  By the simplicity of this module, we can apply \cite[Lemma 5.9]{Webmerged} again and show that $\bar{\Delta}(\Bi)$ has a unique simple quotient.
  
  \item Let $\mathcal{N}_{L}=\{\Bi\in \Seq(\Bv)\mid e(\Bi)L\neq 0\}$.  As observed above, $\Bi$ is lex-minimal amongst $\Bj$ so that $e(\Bj)\bar{\Delta}(\Bi)\neq 0$.  Thus, it is lex-minimal in $\mathcal{N}_{L(\Bi)}$.  This shows that we can recover $\Bi$ from $L(\Bi)$ and this list is irredundant.  It is complete by \cite[Thm. 5.9]{WebGT}.  Since that proof is a bit indirect, we will include a proof here ``by hand,'' which is not difficult in this special case.  We only need to show that any simple $L$ contains a vector that satisfies the relations \eqref{eq:Delta} and \eqref{eq:bar-Delta} for some word $\Bi$.  This is clear if $\Bi$ is red-good and lex-minimal in $\mathcal{N}_{L}$, since $\psi_kv$ is in the image of a lower lex idempotent if $i_k-1=i_{k+1}$; this defines a map $\Delta(\Bi)\to L$ which factors through $\bar{\Delta}(\Bi)$ with simplicity. Thus, we need only show that if $\Bj\in \mathcal{N}_{L}$ is a word which is not red-good, then there is a lex-lower word $\Bj'\in \mathcal{N}_{L}$.  This repeatedly uses the observations that 
  \begin{enumerate}
      \item if $s\neq r\pm 1$ and $e(\dots, r,s,\dots)L\neq 0$ then $e(\dots, s,r,\dots)L\neq 0$ as well.  This is immediate from the relation (\ref{black-bigon}).
      \item Similarly,  $e(\dots, r,r\pm 1,r,\dots)L\neq 0$ if and only if $e(\dots, r,r,r\pm 1\dots)L\neq 0$ or $e(\dots, r,r,r\pm 1\dots)L\neq 0$.  This follows from (\ref{triple-dumb}) using the fact that the LHS of the equation is non-zero if and only the RHS is.  
  \end{enumerate}  
First, let $k$ be the smallest index such that there are two consecutive 
appearances of $k$ in $\Bj$ with more than one copy of $k-1$ between then, or between the last copy of $k$ and the end of the word.  If there is no appearance of $k-2$ between the first two appearances of $k-1$, then we can use (a) to create a word $(\dots, k,k-1,k-1,\dots)\in \mathcal{N}_{L}$. By (b), we also have $e(\dots, k-1,k,k-1,\dots)\in \mathcal{N}_{L}$, which is lower in lexicographic order than $\Bj$.  
Similarly, if there is a copy of $k-2$ between them, we can obtain an idempotent of the form $(\dots, k,k-1,k-2,k-1,\dots)\in \mathcal{N}_{L}$; by (b), either $(\dots, k,k-1,k-1,k-2,\dots)\in \mathcal{N}_{L}$ or $(\dots, k,k-2,k-1,k-1,\dots)\in \mathcal{N}_{L}$.  In the first case, we have $(\dots, k-1,k,k-1,k-2,\dots)\in \mathcal{N}_{L}$  and in the second case,  we have $(\dots, k-2,k,k-1,k-1,\dots)\in \mathcal{N}_{L}$.  Both of these are lower than $\Bj$ in lex-order.
After finitely many of these moves, we will arrive at a word where there is at most one copy of $k-1$ between any two appearances of $k$, or between the last copy of $k$ and the end of the word.   If we switch any consecutive pair $(\dots, r,s,\dots)$ where $r>s+1$, then we will arrive at a word which is a concatenation of words from $\GL$ and $\GL'$.  

Consider two consecutive such words
$(\dots, r,\dots, s, r',\dots, s',\dots)$.  Note that if $r'<r$, then $r'+1$ and $r'$ violate the interlacing condition, so we must have $r'\geq r$.  If $r'>r$, or $s\leq s'$, then these are correctly lexicographically ordered.  Thus, we need only show that if $r=r'<n$ and $s>s'$, then this is not lex-minimal in $\mathcal{N}_{L}$.  We can apply (a) to move $r'$ to the left and obtain a pattern of the form $(\dots, r,r-1,r,\dots)$.  Since $r<n$, we can apply (b) to find that $(\dots, r-1,r,r,\dots)\in \mathcal{N}_{L}$ or $(\dots, r,r,r-1,\dots)\in \mathcal{N}_{L}$.  The former is lex-lower, so we can assume that $(\dots, r,r,r-1,\dots)\in \mathcal{N}_{L}$.  We can now move the $r-1$ left and apply the same argument again, to show that $(\dots,r,r,r-1,r-1,\dots, s,s,s-1,\dots, s',\dots)\in \mathcal{N}_{L}$.  Now, we apply (b) again to see that $(\dots,r,r,r-1,r-1,\dots, s,s-1, s-2,\dots, s',s,\dots)\in \mathcal{N}_{L}$.  Again, we can repeat this observation to see that $(\dots, r,\dots, s', r,\dots, s,\dots)\in \mathcal{N}_{L}$.  These are now in lexicographic order, and applying this inductively, we can arrive at a red-good word.  
\item The kernel $K$ is killed by $e(\Bj)$ for any word $\Bj\leq \Bi$, so $L(\Bj)$ for such a word cannot be a composition factor.\qedhere
\end{enumerate} 
\end{proof}

\subsection{Computation of characters of simple modules}

A well-known and beautiful property of KLR algebras is that they categorify the canonical basis of $U(\mathfrak{n})$ for the underlying diagram; this was shown by Varagnolo and Vasserot \cite{varagnoloCanonicalBases2011a} for an arbitrary symmetrizable Cartan matrix.  As discussed in \cite[\S 4.2]{brundanHomologicalProperties2014}, this means that the classes of simple modules can be identified with the dual canonical basis of this space and therefore can be computed using an algorithm of LeClerc \cite[\S 5.5]{Lecshuf}.  

Let us describe a generalization of this algorithm for simple modules over $\tT$.  We will avoid any discussion of canonical bases; readers wishing for that perspective can refer to \cite{WebCB}. As shown in \cite[Lemma 1.15]{WebCB}, this canonical basis property is a consequence of a positivity property called {\bf mixedness} (\cite[Def. 1.11]{WebCB}).  We will also not require the details of that definition (which is made to work in greater generality than we need), but we will discuss the components of it relevant here.

First, we give every simple module a preferred grading (in terms of the weight function in \cite[Def. 1.11]{WebCB}, this is the grading with weight 0).  Let us set up basic notation for graded vector spaces.  As usual, we let $V_n$ be the degree $n$ component of a graded vector space $V$, and let $\dim_qV=\sum_{i\in \Z}V_iq^i\in \Z[q,q^{-1}].$ The dual space $V^*=\operatorname{Hom}_{\bbC}(V,\bbC)$ has grading so that $V_n^*=(V^*)_{-n}$; in particular $\dim_qV^*=\overline{\dim_qV}$, where $\bar{q}=q^{-1}$.

The module $\bar{\Delta}(\Bi)$ comes with a natural grading where the degree of $v$ is 0, but this is not the one we will use.  
There is a duality functor $L\mapsto L^*$ on the category of left $\tT$-modules defined by considering duality as $\bbC$-vector spaces and twisting by the anti-automorphism given by reflection in the $x$-axis (denoted $\psi$ in \cite[\S 2.1]{KLI}).  
This has the property that $e(\Bi)L^*=(e(\Bi)L)^{*};$ in particular, each simple $L$ and its dual $L^*$ have the same red-good word and thus are isomorphic.  
However, this is not true as graded modules: $e(\Bi)L(\Bi)$ has a basis indexed by the elements of $S_{m_1}\times \cdots \times S_{m_p}$, with these elements having degree -2 times the length of the corresponding element.  Since the generating function for the elements with respect to length is $\prod_{s=1}^n \frac{1-q^{s}}{1-q}$ (\cite[Th. 2.3]{MR4596199}), we have:
\begin{equation}\label{eq:bad-gd}
\dim_qe(\Bi)L(\Bi)=\prod_{r=1}^\ell\prod_{s=1}^{m_r}\frac{1-q^{-2s}}{1-q^{-2}}.
\end{equation}
Thus, in the obvious grading $e(\Bi)L(\Bi)$ is not self-dual. 

However, we can correct for this.  Writing $L(\Bi)$ as a quotient of $\bar{\Delta}(\Bi)$, we let $v$ denote the (unique) vector in $e(\Bi)L$ satisfying the relations \eqref{eq:Delta} and \eqref{eq:bar-Delta}:
\begin{lemma}
In the grading where $\deg(v)=\sum m_r(m_r-1)/2$, the module $L(\Bi)$ is self-dual as a graded module. 
\end{lemma}From now on, when we write $L(\Bi)$, we consider it as a graded module with this unique self-dual grading, and similarly, give ${\Delta}(\Bi),\bar{\Delta}(\Bi)$ the grading where $\deg(v)=\sum m_r(m_r-1)/2$. 
\begin{proof} Given this new grading, we correct \eqref{eq:bad-gd} to 
\begin{equation}\label{eq:good-gd}
    \dim_qe(\Bi)L(\Bi)=\prod_{r=1}^\ell\prod_{s=1}^{m_r}\frac{q^{s}-q^{-s}}{q-q^{-1}}.
\end{equation}
This is bar-invariant, so $e(\Bi)L(\Bi)$ and $e(\Bi)L(\Bi)^*$ are isomorphic as graded vector spaces.  Since a gradable simple over a graded algebra which is finite-dimensional in each degree has a unique grading up to shift, this completes the proof.  
We can also explicitly check this by showing that the unique line in $e(\Bi)L^*$ which is killed by $\bbC[\mathbf{y}]$ (i.e. the unique functional on $e(\Bi)L$ killing the image of $y_k$ for all $k$) is non-zero on the diagram giving the longest element of $e(\Bi)L(\Bi)$ and so has degree $-\deg(v)+\sum m_r(m_r-1)$.  This induces an isomorphism of graded modules $L(\Bi)\cong L^*(\Bi)$.
\end{proof}
 Following LeClerc \cite[(19)]{Lecshuf}, we let \begin{equation}\label{eq:kappa}
    \kappa_{\Bi}=\dim_qe(\Bi)L(\Bi)=\prod_{r=1}^p[m_r]!\qquad  [m]!=\prod_{s=1}^m\frac{q^{s+1}-q^{-s-1}}{q-q^{-1}}.
\end{equation}

We can refine Theorem \ref{thm:all-simples} by incorporating grading shifts; for a graded module $M$, let $M\{a\}$ be the same underlying module with all degrees shifted upward by $a$.
\begin{theorem}\label{th:positive-shift}
  The graded composition factors of the kernel of the map $\bar{\Delta}(\Bi)\to L(\Bi)$ are all of the form $L(\Bj)\{a\}$ for $a>0$ and $\Bj >\Bi$ in lexicographic order.
\end{theorem}
This is the only consequence of the mixedness of $\tT$ that we will require.
\begin{proof}
Since $\bar{\Delta}(\Bi)$ has a unique simple quotient, it must be indecomposable, and thus a quotient of the indecomposable projective with head $L(\Bi)$.  All the composition factors of this module are shifted positively in grading by the mixed property of this category of modules (see \cite[Def. 1.11]{WebCB}), shown in the proof of \cite[Th. 8.8]{WebCB}.  This ultimately derives from the fact that the multiplicity of $L(\Bj)\{a\}$ in this projective cover is the dimension of $\operatorname{Ext}^{a}$ between two perverse sheaves on a moduli space of quiver representations (thanks to \cite[Cor. 4.12]{WebwKLR}). 
\end{proof}
 These will ultimately correspond to the dimensions of Gelfand-Tsetlin weight spaces in simple Gelfand-Tsetlin modules. As in \cite{Lecshuf}, we consider the free associative algebra on $\Z$ on symbols $w_1,\dots, w_n$, with $w[\Bi]=w_{i_1}\cdots w_{i_n}$.  
\begin{definition}
  Let $M$ be a finite-dimensional graded $\tT_{\Bv}$-module.  The {\bf character} $\ch(M)$ is the formal sum $\mathcal{F}$ of words $w[\Bi]$, weighted by the graded dimension of $e(\Bi)M$:
  \[\ch(M)=\sum_{\Bi\in \Seq(\Bv)}\dim_{q}(e(\Bi)M)\cdot w[\Bi].\]
\end{definition}
Using Theorems \ref{thm:all-simples} and \ref{th:positive-shift}, we can give an algorithm to calculate the characters of simple $\tT$-modules, which is only a minor modification of that given in \cite[\S 5.5]{Lecshuf}.
For a given $k$, assume that $\ch(L(\Bi_1)),\dots, \ch(L(\Bi_{k-1}))$ are already known.  
Let $\ch_{k}=\ch(\bar{\Delta}(\Bi_k))$, and define $\ch_r$ for $r\in [1,k-1]$ as follows: Let $\alpha_r$ be the coefficient of $w[\Bi_r]$ in $\ch_{r+1}$, and assume that there is a polynomial $\gamma_r\in q\Z[q]$ such that $\rho_r=\alpha_r -\kappa_{\Bi_r}\gamma_r$ is bar-invariant (such a polynomial is necessarily unique).  Then we let $\ch_r=\ch_{r+1}-\gamma_r\ch(L(\Bi_r))$.  
\begin{lemma}
    In this case $\ch(L(\Bi_k))=\ch_0$ and furthermore, $\gamma_r$ is the graded multiplicity of $L(\Bi_r)$ in $\bar{\Delta}(\Bi_k)$ and $\rho_r=\dim_q e(\Bi_r)L(\Bi_k)$.   
\end{lemma}
\begin{proof}
We prove by induction that $\chi_r$ is the sum of the characters of the composition factors of $\bar{\Delta}(\Bi_k)$ which are of the form $L(\Bi_k)$ or $L(\Bi_s)$ for $s<r$.  For $r=k$, this is clear, since all composition factors of $\bar{\Delta}(\Bi_k)$ are of the form $L(\Bi_s)$ for $s<r$.

Now assume that this is true for $r+1$.  Let $\gamma'_r$ be the graded multiplicity of $L(\Bi_r)$ in $\bar{\Delta}(\Bi_k)$, and $\rho_r'=\dim_q e(\Bi_r)L(\Bi_k)$. Note that by the self-duality of $L(\Bi_k)$, we have $\bar{\rho}_r'=\rho'_r$ and by Theorem \ref{th:positive-shift}, we have $\gamma'_r\in q\Z[q]$. We see that the coefficient $\alpha_r$ is the sum of the contributions of $L(\Bi_k)$ and $L(\Bi_r)$, since all the composition factors $L(\Bi_s)$ for $s<r$ are killed by $e(\Bi_r)$.  This gives the equation $\alpha_r=\rho_r'+\kappa_{\Bi_r}\gamma'_r$.  This shows that $\rho_r'$ and $\gamma_r'$ are uniquely characterized by this equation (as argued by Leclerc in \cite[5.5.3]{Lecshuf}), and are the $\rho_r,\gamma_r$ introduced before.  Thus, $\ch_r$ is indeed the sum of the characters of the composition factors of the form $L(\Bi_k)$ or $L(\Bi_s)$ for $s<r$. 
\end{proof}
Applying this lemma inductively gives us an algorithm for computing the character $\ch(L(\Bi))$ for all red-good words:
\begin{enumerate}
    \item Enumerate the red-good words $\Bi_1,\dots, \Bi_{\ell}$ in decreasing lexicographic order.
    \item Calculate the character $\ch(\bar{\Delta}(\Bi_k))$ for each $k$.  This is a straightforward computation using shuffles, though its combinatorial complexity grows quickly.
    \item For each $k=1,\dots, \ell$, we compute the sequence of characters \[\ch_k=\ch(\bar{\Delta}(\Bi_k)),\ch_{k-1},\dots, \ch_0=\ch(L(\Bi_k)).\]
\end{enumerate}
We have written a program in GAP (based on an earlier program written by Brundan) which performs this computation, though unfortunately, it is only practical to use for the number of strands $\leq 11$.  

\subsection{Canonical modules}
\label{sec:canonical-1}

In work of Early, Mazorchuk and Vishnyakova \cite{EMV}, they introduce certain special simple modules over orthogonal Gelfand-Zetlin algebras, which they call {\bf canonical modules}.  As discussed in \cite[\S 2.5]{WebGT}, these modules have counterparts over $\tT$: Consider the submodule $C'(\Bi)$ in the polynomial representation $\poly$ for each $\Bi\in \Seq(\Bv)$ generated by $e(\Bi)$.  By \cite[Lem. 2.24]{WebGT}, this module has a unique simple quotient $C(\Bi)$.   It will be more convenient for us to consider the dual construction working with the restricted dual of $\poly$, the span of the homogeneous elements in the full vector space dual $\poly^*$.  Note that this is naturally a right module, but we use the reflection anti-automorphism to make it a left module; we could instead construct a right polynomial representation and dualize that to obtain a left module, which will be different from $\poly^*$ and potentially result in a different notion of canonical module (this is one of many points where one needs to make choices in this definition, raising some questions about the name ``canonical'').  

For each $\Bi$, we have a unique element $\ev_{\Bi}\in \poly^* $ of degree 0 which sends $e(\Bj)\mapsto \delta_{\Bi,\Bj}$.  Let $N'(\Bi)$ be the submodule generated by $\ev_{\Bi}$; by an argument as in \cite[Prop. 3.6]{EMV}, this has a unique simple quotient $N(\Bi)$.  
\begin{lemma}
    The simple modules $C(\Bi)$ and $N(\Bi)$ are canonically dual, and thus (non-canonically) isomorphic.
\end{lemma}
\begin{proof}
    Note that the kernel of the map $C'(\Bi)\to C(\Bi)$ is a homogeneous space, and thus its intersection with $\bbC[\mathbf{Y}]$ is spanned by positive degree elements.  That is, this kernel is killed by $\ev_{\Bi}$; since it is a submodule, it has trivial pairing with $N'(\Bi)$.  Thus, we have an induced map $C(\Bi)\to N'(\Bi)^*$, which is non-zero because $\ev_{\Bi}$ and $e(\Bi)$ have non-trivial pairing.  The image of this map is thus a simple submodule isomorphic to $C(\Bi)^*$.  On the other hand, since $N(\Bi)$ is the unique simple quotient of $N'(\Bi)$, its dual is the socle of $N'(\Bi)^*$, which must be the image of $C(\Bi)$.
\end{proof}

Now, assume that $\Bi$ is a red-good word.  
\begin{theorem}\label{lem:good-canonical}
  The element $\ev_{\Bi}$ in $N(\Bi)$ satisfies the relations \eqref{eq:Delta} and \eqref{eq:bar-Delta}, and thus defines an isomorphism $L(\Bi)\cong N(\Bi).$
\end{theorem}
\begin{proof}
    The relations \eqref{eq:bar-Delta} are obvious: any polynomial in the image of multiplication by $Y_i$ has trivial evaluation at the origin.  
    
    We wish to check \eqref{eq:bar-Delta}. We assume that $i_k-1=i_{k+1}$ and consider the evaluation of $\psi_k\ev_{\Bi}$ against $f(\mathbf{Y})\cdot e(s_k\Bi)$.  By definition, this is the same as the evaluation at 0 of \[\psi_kf(\mathbf{Y})\cdot e(s_k\Bi)=(Y_{k+1}-Y_k)f^{s_k}e(\Bi).\] Thus, we obtain 0, showing that $\ev_{\Bi}$ satisfies the desired relations.  This defines a map $\bar\Delta(\Bi)\to N(\Bi)$ which gives the desired isomorphism.  
\end{proof}
This shows that every simple $\tT$ module is a canonical module.
We can determine the canonical modules for all other cases by the following proposition:
\begin{lemma}\label{lem:canonical-moves}
If the words $\Bi$ and $\Bj$ differ by transformations of one of the forms
     \begin{equation}\label{eq:canonical-1}
         \Bi=(\dots, r,s,\dots) \qquad \Bj=(\dots,s,r,\dots)\quad \text{  with $r\neq s\pm 1$,}  
     \end{equation}
          \begin{equation}\label{eq:canonical-2}
     \Bi=(\dots, r,r-1,r \dots)\qquad \Bj=(\dots,r,r,r-1,\dots) \quad \text{ with }1\leq r<n     \end{equation}\begin{equation}\label{eq:canonical-3}
   \Bi=(\dots, r,r+1,r \dots)\qquad \Bj=(\dots,r+1,r,r,\dots)\quad \text{ with }1\leq r<n    
 \end{equation}then we have $N(\Bi)\cong N(\Bj)$.
\end{lemma}
\begin{proof}
    Of course, it's enough to show that $N'(\Bi)=N'(\Bj)$ in each case, or equivalently that $\ev_{\Bj}\in N'(\Bi)$ and $\ev_{\Bi}\in N'(\Bi)$.  In the case of \eqref{eq:canonical-1}, we just have that $\psi_k\ev_{\Bi}=\ev_{\Bj}$ where $k$ is the first index shown, so the result is clear.

    In the  case of \eqref{eq:canonical-2}, we note that $\psi_{k+1}(f\cdot e(\Bj))=f^{s_{k+1}}\cdot e(\Bi)$ and \[\psi_k\psi_{k+1}(f\cdot e(\Bi))=\psi_k(Y_{k+2}-Y_{k+1})f^{s_{k+1}}e(\Bj)=(Y_{k+2}-Y_k)\psi_k(fe(\Bj))+f^{s_{k+1}}e(\Bj).\]
    Dualizing, these calculations show that \[\psi_{k+1}\ev_{\Bi}=\ev_{\Bj} \qquad \psi_{k+1}\psi_{k}\ev_{\Bj}=\ev_{\Bi}.\]  
    A similar calculation for \eqref{eq:canonical-3} shows that in this case 
    \[\psi_{k}\ev_{\Bi}=\ev_{\Bj} \qquad \psi_{k}\psi_{k+1}\ev_{\Bj}=\ev_{\Bi}.\qedhere\]  
\end{proof}

\begin{remark}
As mentioned above, the asymmetry between $i+1$ and $i-1$ here is a reflection of our conventions.  In particular, the polynomial representation $\poly$ from Definiton \ref{def:poly-rep} breaks the symmetry between these, or put differently, depends on the orientation of the underlying Dynkin diagram.  We could swap the signs in equations \eqref{eq:canonical-2} and \eqref{eq:canonical-3} by choosing a different polynomial representation or various other adjustments of conventions.
\end{remark}

Let us describe how these results can be used to reduce the general case to the case of a red-good word.  
\begin{theorem}
    We have $N(\Bi)\cong N(\Bj)$ if and only if $\Bi$ and $\Bj$ differ by a finite chain of the moves in \eqref{eq:canonical-1}--\eqref{eq:canonical-3}.
\end{theorem}
\begin{proof}
Lemma \ref{lem:canonical-moves} covers the ``if'' direction.  
By Theorem \ref{thm:all-simples}(2) and Lemma \ref{lem:good-canonical}, if $\Bi$ is arbitrary, there is some red-good $\Bj$ so that $N(\Bi)\cong L(\Bj)\cong N(\Bj)$.  Thus, we need only prove that $\Bi$ can be transformed by a chain of these moves to be red-good.  We can prove this by induction on lexicographic order: we need only show that if $\Bi$ is not red-good, then it is equivalent by a chain of moves to a lex-lower word.  
  This follows the  pattern of the proof of Theorem \ref{thm:all-simples}(2). As before, choose $k$ minimal so that there is a pair of $k+1$'s (or the last $k+1$ and the end of the word) with more than one $k$ between them.  Applying \eqref{eq:canonical-1}, we can push the $k$'s to the left, so we have either $(\dots, k+1,k,k,\dots)$ or $(\dots, k+1,k,k-1,k,\dots)$.  We can apply \eqref{eq:canonical-3} in the first case or \eqref{eq:canonical-2} followed by \eqref{eq:canonical-3} in the second to transform these to $(\dots, k,k+1,k,\dots)$ or $(\dots, k,k+1,k,k-1,\dots)$, which are lex lower.
  
   This allows us, as in the proof of Theorem \ref{thm:all-simples}(2), to reduce to the case where we have a concatenation of words in $\GL$ and $\GL'$ with their first entries weakly increasing.  Thus, just as in that proof, we only need to deal with the case of two consecutive words $(\dots, r,r-1,\dots ,s,r,\dots s,s-1,\dots)$.  Moving the second appearance of $r$ right and applying \eqref{eq:canonical-2}, then moving the second instance of $r-1$, etc., we see that this has the same canonical simple as $(\dots, r,r,r-1,r-1,\dots, s,s,s-1,\dots )$.  Now, we start applying \eqref{eq:canonical-2} to the block $(s,s,s-1)$, and then to the block $(s+1,s+1,s)$, and so on up to $(r,r,r-1)$.  After several more applications of \eqref{eq:canonical-1}, we arrive at $(\dots, r,r-1,\dots ,s,s-1,r,\dots s,\dots)$, which is lower in lexicographic order.  This completes the proof.
\end{proof} 

\begin{example}
  Consider the word $(4,4,3,2,1,2,3,5,4,4,5,5,3,5,5)$.  Let us apply the arguments above in this case.  We will perform \eqref{eq:canonical-1} freely between lines, and underline the portion of the word where we are doing a move.
  \begin{align*}
      (4,4,\uline{3,2,1,2},3,5,4,4,5,5,3,5,5)&\to (4,4,\uline{2,3,2,1},3,5,4,4,5,5,3,5,5)\\
       (2,4,\uline{4,3,2,3},1,5,4,4,5,5,3,5,5)&\to (2,4,\uline{3,4,3,2},1,5,4,4,5,5,3,5,5)\\
       (2,4,3,4,3,2,1,\uline{5,4,4},3,5,5,5,5)&\to
       (2,4,3,4,3,2,1,\uline{4,5,4},3,5,5,5,5)
  \end{align*}
We have now finished the first stage of the proof, and indeed the resulting word is a concatenation of words in $\GL$ and $\GL'$.  However, they are not weakly increasing in lex-order.  So, we need to apply the transformations:
 \begin{align*}
       (2,\uline{4,3,4},3,2,1,4,5,4,3,5,5,5,5)&\to
       (2,\uline{4,4,3},3,2,1,4,5,4,3,5,5,5,5)\\
       (2,4,4,\uline{3,3,2},1,4,5,4,3,5,5,5,5)&\to
       (2,4,4,\uline{3,2,3},1,4,5,4,3,5,5,5,5)\\
       (2,\uline{4,4,3},2,1,3,4,5,4,3,5,5,5,5)&\to
       (2,\uline{4,3,4},2,1,3,4,5,4,3,5,5,5,5)
  \end{align*}
  After a final application of \eqref{eq:canonical-1}, we arrive at $(2,4,3,2,1,4,3,4,5,4,3,5,5,5,5)$, a red-good word obtained by concatenating the words $(2),(4,3,2,1),(4,3),(4)$ in $\GL$ and $(5,4,3),(5),(5),$ $(5),(5)$ in $\GL'$.
\end{example}

 \section{Gelfand-Tsetlin modules}

\subsection{Equivalence to KLR} Both \cite{KTWWYO} and \cite{WebGT} show an equivalence from a block of the category of Gelfand-Tsetlin modules to the KLRW algebra.  Here we will give a slightly different description of the equivalence.  The works cited above both had to use some complex auxiliary constructions to construct the equivalence. Thus, we wanted to give as ``down to earth'' a description of this equivalence as we could manage.  While we are primarily interested in the case of $\mcU(\gl_n)$, almost everything in the section below applies to general OGZ algebras $\mcU$, so we will work in this generality for some dimension vector $\Bv=(v_1,\dots, v_n)$, with $\Omega=\{(i,j) | 1\leq j\leq v_i\}$.   In particular, we let $Z_n$ denote the symmetric polynomials in the variables $\la_{n,1},\dots, \la_{n,v_n}$. 

Throughout this section, we fix $\chi$, an integral element of $\MaxSpec(Z_n)$.  Consider $\gamma,\gamma'\in \MaxSpec(\Gamma)_{\Z,\chi}$ with $\mathbf{a}$ and $\mathbf{a}'$  the dominant lifts of these weights, with \[a'_{\ell m}=a_{\ell m}\pm\delta_{\ell i}\delta_{mj} \] for some $(i,j)\in \Omega$.  By the assumption of dominance, $a_{i m}=a'_{im}\leq a_{i j}$ if $m<j$, and $a_{i m}=a'_{im}\geq  a'_{i j}$ if $m>j$.

Since $\gamma'$ only depends on the choice of $\gamma$ and $a_{ij}$, we can write $\gamma_a=\gamma'$ if $a_{ij}=a$, with this symbol undefined if $a$ doesn't have this form for $(i,j)\in \Omega$
We can break $X^{\pm}_i$ into a sum $\sum_{a\in \bbC} X^{\pm}_{i,a}$ where $X^{\pm}_{i,a}=\sum_{\gamma\in \MaxSpec(\Gamma)}X^{\pm}_i(\gamma,\gamma_a)$. 

The respective completions of $\Gamma$ at $\gamma$ and $\gamma'$ act as endomorphisms of the corresponding objects in $\Cat$.  
Consider the rational function \[\Phi_{i,j,\Ba}^{\pm}(\boldsymbol{\la})=\frac{\prod_{a_{i,m}\neq a_{i,j}}(\lambda_{im} - \lambda_{ij} ) }{\prod_{a_{i\pm 1,m}\neq a_{i,j}}(\lambda_{i\pm 1 ,m} - \lambda_{ij} )}\] as an element of the ring $\widehat{\Lambda}_{\Ba}$. Note that these are exactly the factors of the Gelfand-Tsetlin formulas that are invertible when specialized at $a_{**}$.  
It will be useful for us to identify the rings $\widehat{\Lambda}_{\Ba}$ and $\widehat{\Lambda}_{\Ba'}$ with each other via the shift map $\delta^{\pm}$, and use $\Phi_{i,j,\Ba}^{\pm}$ to denote the corresponding element of $\widehat{\Lambda}_{\Ba'}$ as well.

Recall the equality of power series:
\begin{multline}\label{eq:power-series}
1+\lambda_{ij}u+\lambda_{ij}^2u^2+\cdots=    \frac{1}{1-\lambda_{ij}u}
=\frac{\prod_{m< j}(1-\lambda_{ij}u)}{\prod_{m\leq j}(1-\lambda_{ij}u)}\\=1+\cdots + \Big(\sum_{k=0}^p (-1)^ke_k(\lambda_{i1},\dots \lambda_{i,m-1})h_{p-k}(\lambda_{i1},\dots \lambda_{i,m})\Big)u^p+\cdots
\end{multline}

This is a useful factorization since while $\Phi_{i,j,\Ba}^{\pm}(\boldsymbol{\la})$ is not invariant under $W_{\Ba}$ or $W_{\Ba'}$, this shows how to write it as an element of the subring generated by the invariants of both of these rings, since $e_k(\lambda_{i1},\dots \lambda_{i,m-1})$ lies in $\Lambda_{\Ba'}^{W_{\Ba'}}$, and $h_{p-k}(\lambda_{i1},\dots \lambda_{i,m})$ in $\Lambda_{\Ba}^{W_{\Ba}}$, and we can apply this formula to every appearance of $\lambda_{ij}-a_{ij}$ in the power series expansion of $\Phi_{i,j,\Ba}^{\pm}$.  

That is, this gives a formula 
\begin{equation}\label{eq:fg}
    \Phi_{i,j,\Ba'}^{\pm}(\boldsymbol{\la})=\sum_{r=1}^s f_r^{\pm}\delta^{\pm}_{ij}(g_r^{\pm}),
\end{equation} with $f_r\in \widehat{\Gamma}_{\gamma'}$ and $g_r\in \widehat{\Gamma}_{\gamma}$. Note that these are independent of the choice of $j$.

\begin{definition}\label{def:tildeX}
  Let \[\tilde{X}_{i}^{\pm}(\gamma,\gamma')=\sum_{r=1}^s f_r^{\pm}{X}_{i}^\pm(\gamma,\gamma')g_r ^{\pm}.\]
\end{definition}

This definition might require a bit of explanation.  This slightly strange formula looks much clearer in the completed polynomial representation of $\mcU$, where by \eqref{eq:Mazorchuk}, the morphisms $\mp\tilde{X}_{i}^{\pm}(\gamma,\gamma')$ act by the formulae
\begin{align}\notag
       \mp \tilde{X}_{i}^\pm &=\sum_{a'_{i,p}=a\pm 1}\sum_{r=1}^s\mp \frac{\prod_{m}(\lambda_{i\pm 1 ,m} - \lambda_{ip} )}{\prod_{m \neq p}(\lambda_{im} - \lambda_{ip} ) } f_r\delta_{ip}^\pm g_r\\ \label{eq:tildesum}
        &=\sum_{a'_{i,p}=a\pm 1} \frac{\displaystyle\prod_{a'_{i\pm 1,m}= a'_{i,p}}(\lambda_{i\pm 1 ,m} - \lambda_{ip} )}{\displaystyle\prod_{\substack{m\neq p\\a'_{i,m} = a'_{i,p}}}(\lambda_{im} - \lambda_{ip} ) } \delta_{ip}^\pm 
    \end{align} 
    In essence, $\tilde{X}_{i}^\pm $ is what we have left when we consider the formulae \eqref{eq:Mazorchuk} and remove as many polynomial factors as we can do safely with the left and right actions of $\Gamma$.  
    
    It is helpful to think about this in the case where $\gamma$ and $\gamma'$ correspond to Gelfand-Tsetlin patterns in the usual sense (i.e. they appear in the spectral decomposition of a finite-dimensional module).  In this case, these operators just act by shift in the completed polynomial representation, and thus, they give a canonical set of isomorphisms between all the GT weight spaces in a finite-dimensional representation.  This is usually represented by choosing a basis with one vector in each weight space, in which case $\tilde{X}_{i}^\pm $ just sends basis vectors to basis vectors. 
    
    Thus, one can translate the question of defining a Gelfand-Tsetlin module into one of defining how $\tilde{X}_{i}^\pm $ acts.  In cases where the values of $a_{ij}$ and $a_{k\ell}$ are far apart, we will see behavior like we discussed for finite-dimensional representations, and $\tilde{X}_{i}^\pm$ will be an isomorphism between GT weight spaces.  In other cases, it will not be, and one has to track the relations between these morphisms.  As we will see below, the algebra $\tT$ is an algebraic object which precisely does this. 

Let us make this connection more precise.  Given $\gamma\in \MaxSpec_{\Z,\chi}(\Gamma)$, there is a unique minimal partial preorder on the set $\Omega$ such that:
\begin{enumerate}
    \item If $a_{i,j}<a_{k,m}$, then $(i,j)\prec (k,m)$.
    \item If $a_{i,j}=a_{k,m}$ with $k<i$, then $(i,j)\prec (k,m)$.
\item If $a_{i,j}=a_{i,m}$, then $(i,j)\approx (i,m)$.
    \end{enumerate}
    The resulting preorder gives an element $\mathbf{I}_{\gamma}\in \operatorname{tSeq}(\Bv)$. This has a corresponding idempotent $e(\mathbf{I})\in \tT=\tT_{\Bv}$ defined in Section \ref{sec:thick-calculus}.  For each $a\in \Z$ and $i\in [1,n]$, we have a strand of thickness $t_{i,a}=\# \{j\in [1,v_i] | a_{i,j}=a\} $, and these strands are ordered first with $a$ weakly increasing left-to-right and $i$ decreasing within strands with a fixed $a$.  
    On the polynomial representation of $\tilde{\mathbb{T}}$, this idempotent acts by projection to the polynomials invariant under the permutations in $W_{\bf{a}}$, the stabilizer of $\Ba$ in $W$.
    The association of the value $a_{ij}$ to this strand is called a {\bf longitude} in \cite{KTWWYO}.

\begin{definition}
  We call $\gamma$ {\bf critical} if $a_{i,j}=a_{i,m}$ for some $j,m$, and {\bf non-critical} otherwise.  The preorder $\prec$ on $\Omega$ is an honest order if and only if $\gamma$ is non-critical.
\end{definition}

Let $e(\gamma)=e(\mathbf{I}_{\gamma})$; given any $\tT$-module $M$, we can thus consider the image $e(\gamma)M$.      Let $\wp\colon \Omega \to [1,N]$ be a bijection which matches the preorder $\prec$ with a refinement of the usual total order; this is not unique, since equivalent elements of $\Omega$ can be in any order, but this will not affect the constructions below.  Sometimes, it will be convenient for us to write $y_{ij}$ in place of $y_{\wp(i,j)}$.  

\begin{lemma}\label{lem:map-to-action}
For each weakly graded and finite dimensional $\tT$-module $M$, there is a representation of the category $\Cat_{\Z,\chi}$ which sends $\gamma\mapsto e(\gamma)M$,  sending a $W_{\bf{a}}$-invariant power series $f(\lambda_{ij}-a_{ij})$ to $f(y_{\wp(i,j)})$, the morphism
$(-1)^{t_{a+1,i}+1}\tilde{X}_{i}^{+}(\gamma,\gamma')$ to the diagram: 
\begin{equation}\label{eq:tXp}
\tikz[baseline,  yscale=2,xscale=3]{\draw[very thick] (-1.5,-.8) to[out=90,in=-90]  (1.5,.8) ; \draw[thickc]  (-2,-1)--(-2,1) node[below,at start]{$i+1$} node[pos=.7,left,scale=.8]{$t_{a,i+1}$};\draw[thickc] (-1.5,-1)--(-1.5,1) node[below,at start]{$i$}node[pos=.7,left,scale=.8]{$t_{a,i}-1$} node[pos=.1,left,scale=.8]{$t_{a,i}$};\draw[thickc] (-1,-1)--(-1,1) node[pos=.7,left,scale=.8]{$t_{a,i-1}$} node[below,at start]{$i-1$};
    \draw[thickc] (1,-1)--(1,1)node[below,at start]{$i+1$}node[pos=.3,left,scale=.8]{$t_{a+1,i+1}$}; \draw[thickc] (1.5,-1)--(1.5,1)node[below,at start]{$i$} node[pos=.9,left,scale=.8]{$t_{a+1,i}+1$} node[pos=.3,left,scale=.8]{$t_{a+1,i}$};\draw[thickc] (2,-1)--(2,1)node[below,at start]{$i-1$} node[pos=.3,left,scale=.8]{$t_{a+1,i-1}$}; \node[scale=1.5] at (0,.5){$\cdots$};\node[scale=1.5] at (0,-.5){$\cdots$}; }
\end{equation}
    where the strands drawn are the thick strands corresponding to $a=a_{i,j}$ and $a+1$ with $i,j$ as above, and 
    $(-1)^{t_{a-1,i-1}}\tilde{X}_{i}^{-}(\gamma,\gamma')$ to the diagram 
    \begin{equation}\label{eq:tXm}
\tikz[baseline, yscale=2,xscale=3]{\draw[very thick] (1.5,-.8) to[out=90,in=-90]  (-1.5,.8) ;\draw[thickc]  (-2,-1)--(-2,1) node[below,at start]{$i+1$} node[pos=.3,left,scale=.8]{$t_{a-1,i+1}$};\draw[thickc] (-1.5,-1)--(-1.5,1) node[pos=.3,left,scale=.8]{$t_{a-1,i}$} node[pos=.9,left,scale=.8]{$t_{a-1,i}+1$}node[below,at start]{$i$};\draw[thickc] (-1,-1)--(-1,1) node[pos=.3,left,scale=.8]{$t_{a-1,i-1}$} node[below,at start]{$i-1$};
    \draw[thickc] (1,-1)--(1,1)node[below,at start]{$i+1$}node[pos=.7,left,scale=.8]{$t_{a,i+1}$}; \draw[thickc] (1.5,-1)--(1.5,1)node[below,at start]{$i$} node[pos=.7,left,scale=.8]{$t_{a,i}-1$} node[pos=.1,left,scale=.8]{$t_{a,i}$};\draw[thickc] (2,-1)--(2,1)node[below,at start]{$i-1$} node[pos=.7,left,scale=.8]{$t_{a,i-1}$}; \node[scale=1.5] at (0,.5){$\cdots$};\node[scale=1.5] at (0,-.5){$\cdots$}; }
    \end{equation}
    By construction, the polynomial representation $\mathscr{P}$ of $\mathcal{C}$ is the pullback of the completed polynomial representation $\widehat{\poly}$ of $\tT$.
\end{lemma}

It will be useful to rewrite the formulas for this functor in terms of the generators $X_i^\pm$.
In order to write this, it will be useful to have the power series \[\varphi^{\pm}_{ij,\Ba}(\mathbf{y})=\frac{\displaystyle \prod_{a_{i\pm 1,m}\neq a_{i,j}}(y_{i\pm 1,m}+a_{i\pm 1,m}-y_{ij}-a_{ij})}{\displaystyle \prod_{a_{i\pm 1,m}\neq a_{i,j}}(y_{im}+a_{im}-y_{ij}-a_{ij})} \in \bbC[[y_1,\dots, y_N]]\]

The functor of Lemma \ref{lem:map-to-action} sends $X_i^\pm(\gamma,\gamma')$ to the operators:
\begin{equation}\label{eq:tXp-KLR}
\tikz[baseline,  yscale=2,xscale=3]{\draw[very thick] (-1.5,-.95) to[out=90,in=-90]  (1.5,.95) ; \draw[thickc]  (-2,-1)--(-2,1) node[below,at start]{$i+1$} node[pos=.7,left,scale=.8]{$t_{a,i+1}$};\draw[thickc] (-1.5,-1)--(-1.5,1) node[below,at start]{$i$}node[pos=.7,left,scale=.8]{$t_{a,i}-1$} node[pos=.1,left,scale=.8]{$t_{a,i}$};\draw[thickc] (-1,-1)--(-1,1) node[pos=.7,left,scale=.8]{$t_{a,i-1}$} node[below,at start]{$i-1$};
    \draw[thickc] (1,-1)--(1,1)node[below,at start]{$i+1$}node[pos=.3,left,scale=.8]{$t_{a+1,i+1}$}; \draw[thickc] (1.5,-1)--(1.5,1)node[below,at start]{$i$} node[pos=.9,left,scale=.8]{$t_{a+1,i}+1$} node[pos=.3,left,scale=.8]{$t_{a+1,i}$};\draw[thickc] (2,-1)--(2,1)node[below,at start]{$i-1$} node[pos=.3,left,scale=.8]{$t_{a+1,i-1}$}; \node[scale=1.5] at (0,.5){$\cdots$};\node[scale=1.5] at (0,-.5){$\cdots$}; 
    \node[draw, very thick, fill=white, inner xsep=180pt] at (0,0){$(-1)^{t_{a+1,i}+1}\varphi^+_{ij,\Ba'}$};}
\end{equation}
\begin{equation}\label{eq:tXm-KLR}
\tikz[baseline, yscale=2,xscale=3]{\draw[very thick] (1.5,-.95) to[out=90,in=-90]  (-1.5,.95) ;\draw[thickc]  (-2,-1)--(-2,1) node[below,at start]{$i+1$} node[pos=.3,left,scale=.8]{$t_{a-1,i+1}$};\draw[thickc] (-1.5,-1)--(-1.5,1) node[pos=.3,left,scale=.8]{$t_{a-1,i}$} node[pos=.9,left,scale=.8]{$t_{a-1,i}+1$}node[below,at start]{$i$};\draw[thickc] (-1,-1)--(-1,1) node[pos=.3,left,scale=.8]{$t_{a-1,i-1}$} node[below,at start]{$i-1$};
    \draw[thickc] (1,-1)--(1,1)node[below,at start]{$i+1$}node[pos=.7,left,scale=.8]{$t_{a,i+1}$}; \draw[thickc] (1.5,-1)--(1.5,1)node[below,at start]{$i$} node[pos=.7,left,scale=.8]{$t_{a,i}-1$} node[pos=.1,left,scale=.8]{$t_{a,i}$};\draw[thickc] (2,-1)--(2,1)node[below,at start]{$i-1$} node[pos=.7,left,scale=.8]{$t_{a,i-1}$}; \node[scale=1.5] at (0,.5){$\cdots$};\node[scale=1.5] at (0,-.5){$\cdots$}; 
    \node[draw, very thick, fill=white, inner xsep=180pt] at (0,0){$(-1)^{t_{a-1,i-1}}\varphi^-_{ij,\Ba'}$};
    }
    \end{equation}
\begin{proof}
First, note that since $M$ is finite dimensional and weakly graded, the dots $y_k$ all act nilpotently.  This shows that the claimed action of the polynomial rings $\widehat{\Gamma}_\gamma$ is well defined.  

Now, we need to check that if a relation holds between $\tilde{X}_{i}^{\pm}(\gamma,\gamma')$ and the polynomials $f(\lambda_{ij}-a_{ij})$ in $\Cat_{\Z,\chi}$, then it holds in our claimed representation.  Since the polynomial representation $\poly$ is faithful by Lemma \ref{lem:faithful}, it is enough to verify this for the induced action of $\tilde{X}_{i}^{\pm}(\gamma,\gamma')$ and $f(\lambda_{ij}-a_{ij})$ on $\poly$.  

The diagram \eqref{eq:tXp} acts in the polynomial representation by 
\begin{enumerate}
    \item[(A)] first a split, which gives the natural inclusion of symmetric polynomials by Lemma \ref{lem:thick-action}, 
    \item[(B)] a crossing of strands of different labels; the only crossing which contributes is the highest one shown in the diagram, between strands with labels $i$ and $i+1$.  This gives multiplication by the product \[\prod_{a'_{i+ 1,m}= a'_{i,j}}(y_{\wp(i+ 1,m)}-y_{\wp(i,j)}).\]
    \item[(C)] finally a merge, which by Lemma \ref{lem:thick-action} corresponds to the application of the Demazure operator  \begin{equation}
        \label{eq:Demazure}\sum _{\sigma \in S_{i,a+1}/\operatorname{St}(j)}\sigma\cdot  \frac{f}{ \displaystyle\prod_{\substack{m\neq j\\a'_{i,m} = a+1}} y_{\wp(i,j)} -y_{\wp(i,m)}}
    \end{equation} where $S_{i,a+1}$ is the group of permutations of $(i,m)\in\Omega$ with $a'_{i,m}=a+1$, and $\operatorname{St}(j)$ is the stabilizer of $j$ in that group.  
\end{enumerate}

Under the map on polynomials $\lambda_{i\pm 1 ,m} - \lambda_{ij}\mapsto y_{\wp(i\pm 1,m)}-y_{\wp(i,j)}$ and $\lambda_{i ,m} - \lambda_{ij}\mapsto y_{\wp(i,m)}-y_{\wp(i,j)}$.  We can use this to match this formula with the composition of (A), (B), and (C) above.   The cosets of $S_{i,a+1}/\operatorname{St}(j)$ are identified with the set $\{p\mid a'_{i,p}=a\pm 1\}$; this allows us to match the terms in the sums \eqref{eq:Demazure} and \eqref{eq:tildesum} applied to $(-1)^{t_{a+1,i}+1}\tilde{X}_{i}^\pm $.  To see that the individual terms are the same, note that each of the shifts $\delta^+_{ip}$ corresponds to a different choice of lift $\Ba'$ of $\gamma'$ and thus embedding of $\widehat{\Gamma}_{\gamma}$ and $\widehat{\Gamma}_{\gamma'}$ in a common polynomial ring.  The permutations $\sigma$ account for this difference, since $\delta^+_{ip}$ changes the value of $a_{ip}$ from $a$ to $a+1$, and thus moves $(i,p)$ to be larger than all other $(i,m)$ with $a_{im}=a$; these give a set of coset representatives.  Thus, for $\sigma$ the unique coset that switches $j$ to $p$, we match the terms 
\begin{equation}\label{eq:term-match}
    (-1)^{t_{a+1,i}}\frac{\displaystyle\prod_{a'_{i\pm 1,m}= a'_{i,p}}(\lambda_{i\pm 1 ,m} - \lambda_{ip} )}{\displaystyle\prod_{\substack{m\neq p\\a'_{i,m} = a'_{i,p}}}(\lambda_{im} - \lambda_{ip} ) } \delta_{ip}^\pm \qquad \text{and}\qquad \sigma\cdot  \frac{\displaystyle f\prod_{a'_{i+ 1,m}= a'_{i,j}}(y_{\wp(i+ 1,m)}-y_{\wp(i,j)})}{ \displaystyle\prod_{\substack{m\neq j\\a'_{i,m} = a+1}} y_{\wp(i,j)}-y_{\wp(i,m)} }
\end{equation}  
Here, we incorporate the factor from (B), which matches the numerator of \eqref{eq:tildesum}; note that the power of $(-1)^{t_{a+1,i}}$ is incorporated into the denominator on the RHS.

The argument for $(-1)^{t_{a-1,i-1}}\tilde{X}_{i}^{-}(\gamma,\gamma')$ is sufficiently similar that we will not write it out here.  We just note that in the analogue of \eqref{eq:term-match}, the factors in the denominators will not have the same sign mismatch, but the terms in the numerator will,  necessitating the factor $(-1)^{t_{a+1,i}}$.

This calculation shows that the result holds for the representation $\poly$, and the result is the same as the polynomial representation of $\Cat_{\Z,\chi}$.  Thus, for any relation in $\Cat_{\Z,\chi}$ involving only $\tilde{X}_{i}^{\pm}(\gamma,\gamma')$ and the polynomials $f(\lambda_{ij}-a_{ij})$ holds in the case of $\poly$, and as argued at the start of the proof, the faithfulness of $\poly$ shows that the same holds for any $\tT$-module.  
\end{proof}

It immediately follows that we have a $\mcU$-module structure on the direct sum $\mathbb{GT}(M)\cong \oplus_{\gamma}e(\gamma)M$ which is Gelfand-Tsetlin by continuity in the discrete topology.   If we consider a representation of $\mathcal{C}$ which is complete in a not necessarily discrete topology, we can still define $\mathbb{GT}(M)$, which will be a pro-Gelfand-Tsetlin module (an inverse limit of Gelfand-Tsetlin modules).  By construction, we have $\mathscr{P}=\mathbb{GT}(\widehat{\poly}).$

In fact, we can state this result more categorically.  Let $\mathcal{T}_{\Z,\chi}$ be the category whose objects are elements of $\MaxSpec_{\Z,\chi}(\Gamma)$, with morphisms $\Hom_{\mathcal{T}_{\Z,\chi}}(\gamma,\gamma')=e(\mathbf{I}_{\gamma'})\widehat{\tT}e(\mathbf{I}_{\gamma})$.  We can think of Lemma \ref{lem:map-to-action} as defining a functor $\Theta\colon \Cat_{\Z,\chi}\to \mathcal{T}_{\Z,\chi}$.
\begin{proposition}\label{prop:equivalence}
The functor $\Theta$ defines an equivalence of categories $\Cat_{\Z,\chi}\cong \mathcal{T}_{\Z,\chi}$.
\end{proposition}
\begin{proof}
We already know that this functor is essentially surjective, since the object sets are the same.  We also know that it is faithful, since as argued in the proof of Lemma \ref{lem:map-to-action}, the pullback of the module $\widehat{\poly}$ along this functor is faithful.  We need only to show that it is full.

Consider any diagram in $e(\mathbf{I}_{\gamma'})\widehat{\tT}e(\mathbf{I}_{\gamma})$.  For simplicity, we draw these pictures so that the terminal corresponding to $(i,j)$ is at the $x$-value $a_{ij}+\epsilon$, for choices of $\epsilon$ that keep the strands in the correct order. Flipping vertically if necessary, we can assume that the rightmost terminal of the top is to the right of the rightmost for the bottom or at the same $x$-value.

If the $x$-values are the same, then we make sure that the thicker terminal at that position is at the top.  If the thicknesses are the same as well, then it might be that our diagram connects these two terminals without interacting with any other strands.  In this case, we can use \cite[Prop. 2.9]{khovanovExtendedGraphical2012} to close up any splits and reduce to the case of a single thick strand decorated with a symmetric polynomial.  Now ignore this strand and consider the next pair of terminals and flip if necessary to make sure that the further right one is at the top and the thicker one if the $x$-values are the same.  We iterate this process until we obtain a pair of terminals which aren't connected by a strand with no splits or merges.  We will show surjectivity by induction on the sum of the thickenesses of these strands at the far right.  Of course, if this sum is $N$, then our diagram is just the identity on an object times an element of $\widehat{\Gamma}_\gamma$ and thus obviously in the image.

We now apply \cite[Prop. 4.5]{SWschur} to the portion of the diagram left of the straight strands we've already found. These diagrams correspond to double cosets in the symmetric group. We can think about this most canonically by thinking about the set $\mathsf{B}$ of pairs of terminals at the bottom of the diagram and integers in $\{1,\dots, t\}$ where $t$ is the thickness of the terminal, and similarly the set $\mathsf{T}$ for the top.  Diagrams thus correspond to orbits in the set of bijections $\mathsf{B}\to \mathsf{T}$ modulo pre- and post-composition with permutations preserving terminals.  We construct the corresponding diagram by having the thickness of the strand joining terminals be the number of elements sent from one terminal to the other by the bijection.  

In particular, one way of choosing our diagram is to first grab all the strands that correspond to elements mapping to the terminal at $x=a$ at the top, and join those to the strands that start at $x=a$ at the bottom, and then completing the rest of the diagram.  Essentially, whenever we have a crossing of strands that meet the terminal at $x=a$ at the top or bottom, we push the crossing right until it has been squashed into the straight line joining these terminals.  The result will look like:
\begin{equation}
    \tikz[yscale=2,xscale=3,baseline]{\draw[thickc] (1.5,1) -- (1.5,-1);\draw[thickc] (1.3,1) -- (1.3,-1); \node at (1,0){$\cdots$}; \draw[thickc]  (.7,1) to node[above, at start]{$x=a$} (.7,-.2) to (.7,-1); \draw[thickc] (.7,-.2) to[out=-90,in=70] node[below, at end]{$x=a'$} (.2,-1); \draw[thickc] (.7,-.2) to[out=-90,in=50] (-1,-1);
    \draw[thickc] (0,1) to (0,.2) to[out=-90,in=100] (.5,-1); 
    \draw[thickc] (-.3,1) to (-.3,.2) to[out=-90,in=100] (.2,-1); 
    \draw[thickc] (-.6,1) to (-.6,.2) to[out=-90,in=80] (-1,-1);
    \draw[thickc] (.7,-.2) to[out=90,in=-50] (.1,.3);
    \node[draw, very thick, fill=white, inner xsep=40pt,inner ysep=15pt] at (-.3,.6) {D}; \node at (-.25,-.85) {$\cdots$};\node at (-.5,-.5) {$\cdots$};
    }    
\end{equation}
for some diagram $D$, with all dots pushed to the top of the diagram.

We let $a$ be the longitude value on the top of the rightmost terminal we consider (as shown in the picture above).  We let $a'$ be the value on the rightmost terminal at the bottom with $x$-value $<a$ which connects to the terminal at $x=a$.  Such a terminal exists by assumption.  Note that the strand starting there must cross all strands with longitude between $a'$ and $a$, including those with longitude $a'$ and label $i'<i$ as well as those with longitude $a$ and label $i'>i$.  Thus, at the bottom of the diagram, there is a portion that looks like:
\begin{equation}\label{eq:tXp-two}
\tikz[baseline,  yscale=2,xscale=3]{\draw[thickc] (-1.5,-.8) to[out=90,in=-90] node[pos=.35, above]{$t$}
node[pos=.65,fill=white, draw,very thick,inner sep=4pt]{$p$}(1.5,.8) ; \draw[thickc]  (-2,-1)--(-2,1) node[below,at start]{$i+1$};\draw[thickc] (-1.5,-1)--(-1.5,1) node[below,at start]{$i$} ;\draw[thickc] (-1,-1)--(-1,1) node[below,at start]{$i-1$};
    \draw[thickc] (1,-1)--(1,1)node[below,at start]{$i+1$}; \draw[thickc] (1.5,-1)--(1.5,1)node[below,at start]{$i$};\draw[thickc] (2,-1)--(2,1)node[below,at start]{$i-1$}; \node[scale=1.5] at (0,.5){$\cdots$};\node[scale=1.5] at (0,-.5){$\cdots$}; }
\end{equation}
where the middle strand has thickness $t$. The strands at the far right have longitude $a'+1$; these might all have thickness 0, in which case we just have a split and no merge.  Similarly, we might have thickness 0 at longitude $a'$ at the top of the diagram, in which case we don't have a split, but just a merge.  In many cases, we will just have a strand moving one unit to the right and no other crossings.  

We now wish to write the diagram \eqref{eq:tXp-two} in terms of \eqref{eq:tXp}.  Note that using the method of \eqref{eq:power-series}, we can write the diagram \eqref{eq:tXp} with any number of dots on the middle strand by multiplying at the top and bottom by elements of $\Gamma$.  

Now, we can rewrite \eqref{eq:tXp-two} using \cite[(2.70)]{khovanovExtendedGraphical2012} by adding a split and merge in the middle down to $t$ strands of weight 1, placing $t-1$ dots on 1 strand, $t-2$ on another, etc.  Using \cite[Prop. 2.4]{khovanovExtendedGraphical2012}, we can push this split and merge to the vertical strands and rewrite this as a ``ladder''
\begin{equation}\label{eq:tXp-three}
\tikz[baseline,  yscale=5,xscale=3]{\draw[very thick] (-1.5,.1) to[out=60,in=180] (0,.5) to[out=0,in=-120] node[at start, circle,draw,fill=black,inner sep=2pt,label=above:{$t-1$}]{} (1.5,.9) ; 
\draw[very thick] (-1.5,-.2) to[out=60,in=180] (0,.2) to[out=0,in=-120] node[at start, circle,draw,fill=black,inner sep=2pt,label=above:{$t-2$}]{} (1.5,.6) ;\draw[very thick] (-1.5,-.6) to[out=60,in=180] (0,-.2) to[out=0,in=-120] node[at start, circle,draw,fill=black,inner sep=2pt,label=above:{$1$}]{} (1.5,.2) ;\draw[very thick] (-1.5,-.9) to[out=60,in=180] (0,-.5) to[out=0,in=-120]  (1.5,-.1) ;\draw[thickc]  (-2,-1)--(-2,1) node[below,at start]{$i+1$};\draw[thickc] (-1.5,-1)--(-1.5,1) node[below,at start]{$i$} ;\draw[thickc] (-1,-1)--(-1,1) node[below,at start]{$i-1$};
    \draw[thickc] (1,-1)--(1,1)node[below,at start]{$i+1$}; \draw[thickc] (1.5,-1)--(1.5,1)node[below,at start]{$i$};\draw[thickc] (2,-1)--(2,1)node[below,at start]{$i-1$}; \node[scale=1.5] at (0,0){$\cdots$}; \node[inner ysep=80pt,fill=white,draw,very thick] at (.4,0){$p$};}
\end{equation}
As we've already observed, this is in the image of the functor.  Thus, we can reduce to showing that the portion of the diagram above is in the image.  If $a'<a-1$, then the result is of the same form but has increased $a'$.  After a finite number of applications, we will have $a'=a-1$, and the next application will increase the thickness of the terminal with label $i$ and longitude $a$ at the bottom, and reduce the number of times this strand splits.  We can continue this process until we have removed all splits in this strand.
There are now two possibilities for the remainder: 
\begin{enumerate}
    \item The terminal with label $i$ and longitude $a$ is thicker at the bottom than at the top.  So our diagram now has the form.
\begin{equation}
    \tikz[yscale=2,xscale=3]{\draw[thickc] (1.5,1) -- (1.5,-1);\draw[thickc] (1.3,1) -- (1.3,-1); \node at (1,0){$\cdots$}; \draw[thickc]  (.7,1) to node[above, at start]{$x=a$} (.7,-.2) to (.7,-1); \draw[thickc] (0,1) to (0,.2) to[out=-90,in=100] (.5,-1); 
    \draw[thickc] (-.3,1) to (-.3,.2) to[out=-90,in=100] (.2,-1); 
    \draw[thickc] (-.6,1) to (-.6,.2) to[out=-90,in=80] (-1,-1);
    \draw[thickc] (.7,-.2) to[out=90,in=-50] (.1,.3);
    \node[draw, very thick, fill=white, inner xsep=40pt,inner ysep=15pt] at (-.3,.6) {D}; \node at (-.4,-.85) {$\cdots$};
    }    
\end{equation}In this case, we can flip the diagram, and apply this argument again.
    \item There is a strand with no splits or merges connecting these terminals, so the diagram has the form:
    \begin{equation}
    \tikz[yscale=2,xscale=3,baseline]{\draw[thickc] (1.5,1) -- (1.5,-1);\draw[thickc] (1.3,1) -- (1.3,-1); \node at (1,0){$\cdots$}; \draw[thickc]  (.7,1) to node[above, at start]{$x=a$} (.7,-.2) to (.7,-1); \draw[thickc] (0,1) to (0,.2) to[out=-90,in=100] (.5,-1); 
    \draw[thickc] (-.3,1) to (-.3,.2) to[out=-90,in=100] (.2,-1); 
    \draw[thickc] (-.6,1) to (-.6,.2) to[out=-90,in=80] (-1,-1);
    \node[draw, very thick, fill=white, inner xsep=40pt,inner ysep=15pt] at (-.3,.6) {D}; \node at (-.4,-.85) {$\cdots$};
    }    
\end{equation}In this case, we exclude this strand, and continue the argument, having increased the total thickness of the straight strands at the right.
\end{enumerate}
This completes the proof since after a finite number of applications of this argument, we'll arrive at a diagram of all straight strands, which we've already covered.  
\end{proof}
The category $\mathcal{T}_{\Z,\chi}$ is equivalent to a standard construction: let $A$ be an algebra and $e_1,\dots, e_m$ a list of idempotents.  The category $\mathcal{A}$ has the object set $\{e_1,\dots, e_m\}$, morphism spaces $\Hom_{\mathcal{A}}(e_i,e_j)=e_jA e_i$, and the composition given by multiplication.  This is equivalent to the subcategory of projective right $A$-modules with objects $Ae_i$ for $i=1,\dots, m$.  

Of course, for any left $A$-module $M$, there is a representation of the category $\mathcal{A}$ sending $e_m\mapsto e_mM$; this functor has a left adjoint $A\otimes_{\mathcal{A}}-$.  
A basic exercise in ring theory shows that: \begin{lemma}
These functors make the category of representations of $\mathcal{A}$ the quotient category of left $A$-modules by the Serre subcategory of modules $M$ such that $e_iM=0$ for all $M$, with the categories being equivalent if $e_1,\dots, e_m$ generate the ring $A$ as a 2-sided ideal. 
\end{lemma}

The category $\mathcal{T}_{\Z,\chi}$ is equivalent to this category for $A=\widehat{\tT}$ and the set of idempotents being the (finite) set of idempotents of the form $e(\Bi_{\gamma})$.  

Note that this set depends on the choice of $\chi$.  If $a_{n,1}\ll \cdots \ll a_{n,n}$, then every idempotent $e(\Bi)$ appears, but this is not the case for $\chi$ where $a_{n,i}-a_{n,i-1}<n$.
In either case, we can define a functor $\mathbb{KLR}\colon \GTc_{\Z,\chi}\to \tT\wgmod$ by: \begin{enumerate}
    \item using the equivalences of Propositions \ref{prop:DFO} and  \ref{prop:equivalence} to construct the corresponding discrete $\Cat_{\Z,\chi}$-module,
    \item transport structure to a $\mathcal{T}_{\Z,\chi}$-module,
    \item apply the functor $\widehat{\tT}\otimes_{\mathcal{T}_{\Z,\chi}}-$.
\end{enumerate}  In the case where $a_{n,1}\ll \cdots \ll a_{n,n}$, we can give a slightly different definition of this functor: for each $\Bi$, choose a corresponding $\gamma(\Bi)$ and we find that
\[\mathbb{KLR}(V)=\bigoplus_{\Bi}\Wei_{\gamma(\Bi)}(V)\] with $\tT$ acting on the right-hand side by the equivalence of Proposition \ref{prop:equivalence}.  Note that we could alternatively apply \cite[Th. 4.4]{WebGT} to prove the existence of this action.
\begin{theorem}[\mbox{\cite[Th. 5.6]{WebGT}}]\label{thm:GT-KLR}
If $\chi$ is a regular integral character (i.e. $a_{n,1}< \cdots < a_{n,n}$), the functors $\mathbb{KLR}$ and $\mathbb{GT}$ define an equivalence of categories  $\widehat{\tT}\wgmod\cong \GTc_{\Z,\chi}$.
\end{theorem}
In the singular case, a version of this theorem holds, but the algebra $\tT$ needs to be replaced with another KLR-type algebra which doesn't precisely fit the definition we gave: one needs to incorporate thick red strands; this is covered in detail in \cite[\S 4.1]{websterThreePerspectives2020}. 

Note that we have come most of the way to giving a new proof of this theorem---since Propositions  \ref{prop:DFO} and \ref{prop:equivalence} give equivalences, the only question is whether the functor $\widehat{\tT}\otimes_{\mathcal{T}_{\Z,\chi}}-$ is an equivalence, that is, if the idempotents $e(\gamma)$ for $\gamma\in \MaxSpec_{\Z,\chi}(\Gamma)$ generate $\widehat{\tT}$.

This fact is shown based on the theory of KLR algebras in \cite[Thm 5.9]{WebGT}: we know a complete classification of graded simple $\widehat{\tT}$-modules and can show that none are killed by all idempotents of the form $e(\gamma)$ for $\gamma\in \MaxSpec(\Gamma)_{\Z,\chi}$.  However, as noted in the proof of \cite[Cor. 4.10]{websterThreePerspectives2020}, a little work on the side of representation theory allows us to avoid using this fact for $\mcU=\mcU(\gl_n)$---if $a_{n,1}\ll \cdots \ll a_{n,n}$, then every idempotent $e(\Bi)$ appears, and so $\widehat{\tT}$ and $\mathcal{T}_{\Z,\chi}$ are Morita equivalent.  However, translation equivalences show that every regular integral block has the same number of simple Gelfand-Tsetlin modules, so it suffices to show this equivalence for one such block, and it must hold for all of them.

\begin{remark}
Those who have tried to work in a ``hands-on'' way with Gelfand-Tsetlin modules in higher rank know that the key problem with writing explicit bases and structure coefficients for these modules is the difficulties around points where many $a_{i,*}$ coincide.  Theorem \ref{thm:GT-KLR} shows that we can compress all of this difficulty into the Morita equivalence between $\widehat{\tT}$ and $\mathcal{T}_{\Z,\chi}.$  This is certainly not trivial to write out explicitly, but does reduce this question to an established combinatorial commutative algebra of symmetric polynomials; this connection was already noted in \cite{FGRZGalois}.

This approach has the notable effect that one can on some level ignore critical weights (those where $a_{i,j}=a_{i,k}$ for some $i$); their structure is entirely captured by the action of the KLR algebra on the weight spaces for non-critical weights. This is quite difficult to see in the principal block, which is usually the regular integral block that representation theorists prefer to work with, since in this case, many idempotents cannot be written as $e(\gamma)$.  
\end{remark}

Let us note how certain properties of Gelfand-Tsetlin modules transfer to modules over the KLR modules.  For a word $\Bi$, we let the {\bf Gelfand-Kirillov dimension} of $\Bi$ be the number of entries in the word that are not between a pair of $n$'s (i.e., either to the left or to the right of all $n$'s).  This is justified by the fact that if we look at the number of maximal ideals of distance $<\delta$ from any fixed point (identifying maximal ideals with points in $\bbC^1\times \cdots\times \bbC^n$) grows approximately as a polynomial of this degree;  one can see this from the fact that the coordinates corresponding to $n$'s are fixed by $\chi$, but otherwise, the coordinates can grow arbitrarily within a particular cone.  In terms of affine geometry, we can think of this as slicing a cone with an affine space.
\begin{theorem}\label{thm:GK}
  The Gelfand-Kirillov dimension of $\mathbb{GT}(M)$ is the maximum of the GK dimensions of the words $\Bi$ such that $e(\Bi)M\neq 0$.
\end{theorem}
\begin{proof}
 We consider $\mcU$ to be generated by $X_i^{\pm}$ and by some generating set of $\Gamma$.
We claim that given a $\Gamma$-invariant generating subspace $M'$ of the module $\mathbb{GT}(M)$, there are constants $c$ and $C$, such that
\begin{enumerate}
\item 
  the span of all length $\ell$ monomials in generators on $\mcU$ contains the GT weight spaces in a ball of radius $c\ell$ around any fixed point $x\in \bbC^1\times \cdots\times  \bbC^n$, and
 \item it is contained in the sum of those contained in a ball of radius $C\ell $. 
 \end{enumerate}  We can see the bound (2) by the fact that $X_i^\pm$ can only change one of the coordinates by $\pm 1$, and $\Gamma$ cannot change it at all.  To see the bound (1), we need to note that all the projections to weight spaces can be written as a polynomial of degree less than some fixed constant independent of which weight space is under discussion, so the morphism $\tilde{X}_i^\pm(\gamma,\gamma')$ can be written using a number of generators independent of $\gamma,\gamma'$.  Since all morphisms $\Hom_{\mathcal{C}}(\gamma,\gamma')$ are spanned as a left $\Gamma$-module by finitely many monomials in the $\tilde{X}_i^\pm(\gamma,\gamma')$, whose length is bounded above by a multiple of the distance between $\gamma$ and $\gamma'$, a constant with the desired properties $c$ exists.

 Thus, we can determine the Gelfand-Kirillov dimension by considering the growth of the number of lattice points in the support of $\mathbb{GT}(M)$, weighted by their dimensions; however, since only finitely many different $e(\gamma)$, these dimensions are uniformly bounded above, so we can count lattice points in the support, ignoring multiplicity.  The sum of finitely many positive functions has growth rate given by the maximum of the growth rates that appear, and these are given by the GK dimensions of the different idempotents $e(\Bi)$ such that $e(\Bi)M\neq 0$.  This completes the proof.
\end{proof}

\subsection{The non-integral case}
Much recent work on Gelfand-Tsetlin modules has mainly focused on the cases which are relatively far from being integral, with only a small number of the differences $a_{ij}-a_{ik}$ being integers.  Thus, it is reasonable to ask what we will find in the non-integral case.

{For each maximal ideal $\gamma\in \MaxSpec_{\chi}(\Gamma)$ and each $\bar{a}\in \bbC/\Z$, define a vector $\Bv^{(\bar{a})}$ by $v_i^{(\bar{a})}=\#\{ j |a_{ij}\equiv \bar{a}\mod \Z\}$ for $\bar{a}\in \bbC/\Z$.  The ideal $\gamma$ is integral if and only if $\Bv^{(\bar{0})}=\Bv$.}

\begin{lemma}
For any indecomposable Gelfand-Tsetlin module $M$ over $\mcU_{\Bv}$, the vector $\Bv^{(\bar{a})}$ will be the same for every maximal ideal in the support of $M$ for each $\bar{a}$.
\end{lemma}
\begin{proof}
The morphism space between maximal ideals in $\mathcal{C}$ will only be non-zero if the vectors $\Bv^{(\bar{a})}$ are equal.  This writes $\mathcal{C}$ as a direct sum of subcategories where we fix $\Bv^{(\bar{a})}$, and any indecomposable module kills all but one of these subcategories.  
\end{proof}
Let $\mathscr{S}$ be the set of maximal ideals corresponding to some fixed choice of $\Bv^{(\bar{a})}$  for each $\bar{a}\in \bbC/\Z$; by shift by any element of the corresponding coset, we can identify this with the integral maximal ideals for $\otimes \Gamma_{\bar{a}}$.  
By \cite[Cor. 4.7]{WebGT}, we have that:
\begin{proposition}\label{prop:integrality}
The block of Gelfand-Tsetlin modules supported in $\mathscr{S}$ over the OGZ algebra $\mcU_{\Bv}$ is equivalent to the integral modules over the tensor product $\bigotimes_{\bar{a}\in \bbC/\Z} \mcU_{\Bv^{(\bar{a})}}$, which we can describe using Theorem \ref{thm:GT-KLR} as the module category $\bigotimes_{\bar{a}\in \bbC/\Z} \tT_{\Bv^{(\bar a)}}\wgmod$.    
\end{proposition}
In the interest of writing this out explicitly, consider the polynomial  \[\phi_{i,j}^{\pm}(\boldsymbol{\la})=\frac{\prod_{\bar{a}_{i,m}\neq \bar{a}_{i,j}}(\lambda_{im} - \lambda_{ij} ) }{\prod_{\bar{a}_{i\pm 1,m}\neq \bar{a}_{i,j}}(\lambda_{i\pm 1 ,m} - \lambda_{ij} )}\]
This lies in $\widehat{\Lambda}_{\Ba}$; as with $\Phi_{i,j,\Ba}^{\pm}$, we can write it as \begin{math}
    \phi_{i,j}^{\pm}=\sum_{r=1}^p s_r^{\pm}\delta^{\pm}_{ij}(t_r^{\pm})
\end{math} for $s_r$ and $t_r$ which are well defined in all completions of $\Gamma$ at points in $\mathscr{S}$.  
We can break $X^{\pm}_i$ as a morphism in $\mathcal{C}$ into a sum $X^{\pm}_i=\sum_{\bar{a}\in \bbC/\Z}X^{\pm}_{i,\bar{a}}$ where $X^{\pm}_{i,\bar{a}}=\sum_{a\in  \bar{a}+\Z} X^{\pm}_{i,a}$; informally, this is the sum of the components of $X_i^+$ that change an $a_{ij}\in \bar{a}+\Z$ into a $a_{ij}\pm 1$.   

Analogously to Definition \ref{def:tildeX}, we can define
\[\bar{X}^{\pm}_{i,\bar{a}}=\sum_{r=1}^ps_r^{\pm}X^{\pm}_{i,\bar{a}}t_r^{\pm}.\] We leave to the reader confirming that: 
\begin{lemma}\label{lem:polys}
Identifying the completed polynomial representation of $\mcU_{\Bv}$ at points in $\mathscr{S}$ with the completion of the polynomial representation of the tensor product $\bigotimes_{\bar{a}\in \bbC/\Z} \mcU_{\Bv^{(\bar{a})}}$ at integral points,
the operators $\bar{X}^{\pm}_{i,\bar{a}}$ act by the formulas \eqref{eq:Mazorchuk2} for the alphabets corresponding to $\Bv^{(a)}$.
\end{lemma}   Thus, for any Gelfand-Tsetlin module whose support lies in $\mathscr{S}$, the operators $\bar{X}^{\pm}_{i,\bar{a}}$ define an action of  $\bigotimes_{\bar{a}\in \bbC/\Z} \mcU_{\Bv^{(\bar{a})}}$ on the same underlying vector space; this defines the equivalence of \cite[Cor. 4.7]{WebGT}.

\subsection{Canonical modules}

Consider $\gamma\in \MaxSpec(\Gamma)_{\chi}$ with $\Bi=\Bi(\gamma)$ the corresponding idempotent in $\tT$; if $\gamma$ is non-integral, we have a sequence of words $\Bi^{(\bar{a})}$, one for each $\bar{a}\in \bbC/\Z$.  Early, Mazorchuk and Vishnyakova define a {\bf canonical module} $N(\gamma)$ as follows: Let $N'(\gamma)$ be the submodule generated $\mathrm{ev}_\gamma$ as an element of the (full) dual space $\Gamma^*$, and $N(\gamma)$ its unique simple quotient.  As defined in \cite{EMV}, this module is a right module, so we will find it more convenient to work with its restricted dual $C(\gamma)$, which is a left module.  We can realize $C(\gamma)$ by considering the representation of the category $\mathcal{C}$ on the completed polynomial rings $\widehat{\Gamma}_{\gamma'}$, letting $C'(\gamma)$ be the submodule generated by $\widehat{\Gamma}_{\gamma}$ and $C(\gamma)$ its unique simple quotient as a topological module over $\mathcal{C}$.  

\begin{remark}
The reader may recall that for $\tT$, we defined a {\it left} module $N(\Bi)$ which is dual to $C(\Bi)$; here, we used an anti-automorphism of $\tT$ to switch between right and left modules.  An obvious question is whether there is a compatible duality on the category of Gelfand-Tsetlin modules. The usual duality on Lie algebra modules induced by the anti-automorphism $X\mapsto -X$ is obviously unsuitable since it switches highest and lowest weight modules.  The anti-automorphism $X\mapsto X^T$ is compatible with the inclusion $\mathfrak{gl}_k\hookrightarrow \mathfrak{gl}_n$ and acts by the identity on $\Gamma$.  Thus, if we define the contragredient $M^*$ using this anti-automorphism, it will have the property that $\Wei_{\gamma}(M)^*=\Wei_{\gamma}(M^*)$ for all modules $M$. This allows us to define the left module version of $N(\gamma)$, which is non-canonically isomorphic to $C(\gamma)$.

Note that while we have not carefully verified that this contragredient matches the one we defined for $\tT$-modules, it seems virtually certain that this is indeed the case.
\end{remark}

By Lemma \ref{lem:polys}, the equivalence of Proposition \ref{prop:integrality} sends canonical modules to canonical modules, and we can determine all canonical modules by understanding those in the integral case.

We will do this by comparison with the $\tT$-module $C(\Bi)$ defined in Section \ref{sec:canonical-1}; in the non-integral case, we use this to mean the tensor product of the corresponding canonical modules over $\tT_{\Bv^{(\bar{a})}}$.
\begin{proposition}
We have an isomorphism $C(\gamma)\cong \mathbb{GT}(C(\Bi))$.
\end{proposition}
\begin{proof}
As discussed above, we need only consider the case where $\gamma$ is integral.

First, by Lemma \ref{lem:map-to-action} (and the discussion below), we have an isomorphism between the representations $\gamma\mapsto \widehat{\Gamma}_{\gamma}$ and $\mathbb{GT}(\widehat{\poly})$.   Since this sends $1\in \widehat{\Gamma}_{\gamma}$  to $e(\Bi)$, we also have an isomorphism $C'(\gamma)\cong \mathbb{GT}(\widehat{C'(\Bi)})$, where we complete $C'(\Bi)$ with respect to its grading.  Taking the unique simple quotient of both sides shows that $C(\gamma)\cong \mathbb{GT}(C(\Bi))$.
\end{proof}
While we have only stated this result for integral $\gamma\in \MaxSpec(\Gamma)_{\Z,\chi}$, we can also apply the result to arbitrary $\gamma $ using Lemma \ref{lem:polys}.
Thus, the results of Section \ref{sec:canonical-1} apply immediately to Early, Mazorchuk, and Vishnyakova's canonical modules:
\begin{corollary}
Every simple Gelfand-Tsetlin module over the OGZ algebra $\mcU$ is isomorphic to a canonical module for some maximal ideal $\gamma$ and two integral maximal ideals $\gamma,\gamma'$  have isomorphic canonical modules if and only if their corresponding words $\Bi^{\bar{a}}(\gamma),\Bi^{\bar{a}}(\gamma')$ differ by a finite chain of the moves in \eqref{eq:canonical-1}--\eqref{eq:canonical-3} for each ${\bar{a}}\in \bbC/\Z$.
\end{corollary}
In the integral case, we only have a single pair of words $\Bi(\gamma)=\Bi^{\bar{0}}(\gamma),\Bi(\gamma')=\Bi^{\bar{0}}(\gamma')$ we need to compare.

\subsection{Verma modules}  

One major focus of the paper \cite{FGRZVerma} is to understand Verma modules in a way compatible with the Gelfand-Tsetlin subalgebra.  Here, we present one solution to this problem.  Recall that for each red-good word $\Bi$, we have introduced a standard module $\Delta(\Bi)$ as an induction from the words in its red-good factorization.  

Throughout this section, we'll assume that $\Bi=b_1\cdots b_n$ for $b_k=(n,n-1,\dots,r_k)\in GL'$.  Note that in order to obtain the dimension vector $(1,\dots, n)$, the map $\sigma(k)= r_k$ must be a permutation $\sigma\in S_n$.  Now choose a central character $\chi$, and let $\mu$ be the unique weight such that $\mu-\rho$ is anti-dominant and the Verma module $M(\mu)$ with lowest weight $\mu$ for the negative Borel has central character $\chi$.  As usual, we let $e_{i,i}m=\mu_im$ for $m$ the lowest weight vector; anti-dominance after adding $\rho$ means that \[\mu_1\leq\mu_2+1\leq \cdots \leq \mu_{n}+(n-1).\]  Using the Harish-Chandra homomorphism, we see that a symmetric polynomial $f(\la_{n1},\dots, \la_{nn})$ acts by the scalar $f(\chi_1,\chi_2,\dots, \chi_n)$ where $\chi_i=\mu_i+(i-1)$.  Of course, the other Verma modules with the same central character are of the form $M(\sigma)=M(\sigma(\mu-\rho)+\rho)$ for $\sigma \in S_n$ and $\rho=(\frac{n-1}{2},\frac{n-3}{2},\dots, \frac{1-n}{2})$ is half the sum of the positive roots (for the positive Borel).  If $\mu'=\sigma(\mu-\rho)+\rho$,  then we have \begin{equation}\label{eq:muprime}
    \mu'_i=\chi_{\sigma^{-1}(i)}-(i-1)
\end{equation}
For simplicity, we will only prove this result for $\chi\in \MaxSpec_{\Z}(Z_n)$, in which case $M(\sigma)$ in an integral Gelfand-Tsetlin module for all $\sigma$, but with appropriate changes, the same results hold for any Verma module.

\begin{theorem}
We have an isomorphism $\mathbb{GT}(\bar{\Delta}(\Bi))\cong M(\sigma)$, where $\sigma(k)=r_k$ as defined above.  
\end{theorem}

\begin{proof}
For simplicity, we first consider the case where $\chi$ is regular;  we will deal with the singular case at the end of the proof.

Let $\gamma\in \MaxSpec(\Gamma)$ be the maximal ideal that kills the unique lowest weight vector of $M(\sigma)$.  Since this vector is lowest weight for $\gl_k$ for all $k=1,\dots, n$, this maximal ideal is determined by the Harish-Chandra homomorphism---the action is induced by sending $\la_{ij}$ to the scalar $a_{ij}=\chi_{\sigma^{-1}(j)}$ by \eqref{eq:muprime}.  This maximal ideal $\gamma$ has associated word $\Bi(\gamma)$ given by the original word $\Bi$.  Note that this is exactly the maximal ideal $\la^{(\Bi)}$ introduced in \cite[\S 5.3]{WebGT}.

Thus, we have that $\Wei_{\gamma}(\mathbb{GT}(\bar{\Delta}(\Bi)))\cong e(\Bi)\bar{\Delta}(\Bi)$.  As shown in the proof of Theorem \ref{thm:all-simples}(1), this is a 1-dimensional vector, spanned by the generating vector $v$ for $\bar{\Delta}(\Bi)$; let $v'$  denote the corresponding vector in $\mathbb{GT}(\bar{\Delta}(\Bi))$.  Note that the action of any $X^-_k$ on $v'$ is a sum over diagrams that move a black strand to the left and is a sum of terms $d\psi_kv$, where $\psi_kv=0$ by the definition \ref{eq:Delta}.  Note that this argument does not apply to $X^+_k$, since diagrams where a black strand moves northeast across a red strand are not set to 0. Thus, $v'$ is a lowest weight vector.  

By the fullness of the functor $\Theta$, we find that $v'$ generates $\mathbb{GT}(\bar{\Delta}(\Bi))$ since $v$ generates $\bar{\Delta}(\Bi)$.  This shows that we have a surjective homomorphism $M(\sigma)\to \mathbb{GT}(\bar{\Delta}(\Bi))$.  We need only show that this is injective.  Since it is not zero, it must be injective if the Verma module is simple, so if $\sigma=w_0$ is the longest permutation.

One way to see that it is injective is to compute the Gelfand-Kirillov dimension of $\mathbb{GT}(\bar{\Delta}(\Bi))$.  By Lemma \ref{lem:induction-basis}, the space $\Wei_{\gamma'}(\mathbb{GT}(\bar{\Delta}(\Bi)))=e(\gamma')\bar{\Delta}(\Bi))$ is non-zero whenever $\Bi(\gamma')$ can be written as a shuffle of the words $b_i$ without switching the order of two $n$'s.  In particular, this is possible if $\gamma'$ is in the support of $M(w_0)$, and any Gelfand-Tsetlin module with all these spaces non-zero has Gelfand-Kirillov dimension $(n+1)n/2$. Any proper quotient of a Verma module has strictly lower Gelfand-Kirillov dimension, so the map to $\mathbb{GT}(\bar{\Delta}(\Bi))$ must be injective.

Another approach is to note that this construction is well-defined not just for $\bbC$-valued weights, but also for weights in a complete local ring over $\bbC$ (or any other field).  
In this context, we can consider the deformed Verma module $\tilde{M}(\sigma)$ as in \cite[\S 3.1]{soergelCombinatoricsHarishChandra1992} with $T'$ the completed local ring of $U(\mathfrak{h})$ at $\sigma(\mu-\rho)+\rho$.   This has a natural map to $\mathbb{GT}(\hat{\Delta}(\Bi))$ where $\hat{\Delta}(\Bi)$ is the completion of $\Delta(\Bi)$ with respect to its grading.  After passing to the fraction field of $T'$, the module $\tilde{M}(\sigma)$ becomes simple, and so the induced map is an isomorphism.   By the flatness of the deformed Verma, our original map must be injective.  

This covers the case of a regular block. For a singular integral block, we have the usual translation functor $T$ from a regular block.  The functor $T$ sends the Verma module $M(\sigma(\mu'-\rho)+\rho)$ to the Verma module $M(\sigma(\mu-\rho)+\rho)$; note that the singularity of the block means that we will sometimes have $M(\sigma)\cong M(\sigma')$ for two different permutations.  On the other hand, by \cite[Cor. 4.14]{websterThreePerspectives2020}, we find that the functor $\mathbb{GT}$ (which depends on a choice of central character) commutes with translation onto a wall.  That is, if we use $\mathbb{GT}'$ to denote this functor for $\chi'$, then $\mathbb{GT}\cong T\circ \mathbb{GT}'$.  Thus
\[\mathbb{GT}(\Delta(\Bi))\cong T\circ \mathbb{GT}'(\Delta(\Bi))\cong T(M(\sigma(\mu'-\rho)+\rho))\cong M(\sigma).\qedhere\]
\end{proof}

Lemma \ref{lem:induction-basis} furthermore shows us how to construct a basis of the Verma module $M(\sigma)$.  Combining this with the construction of the functor $\mathbb{GT}$, we see that we have a basis indexed by {\bf semi-patterns}, choices of $a_{ij}$ for $(i,j)\in \Omega$ that satisfy $a_{ij}\geq a_{i+1,j}$, with $a_{nj}=\chi_{\sigma^{-1}(j)}$; these satisfy the NW-SE inequalities of a Gelfand-Tsetlin pattern, but not the NE-SW inequalities; any such semi-pattern has a corresponding maximal ideal $\gamma$ in $\Gamma$. Given such a semi-pattern, can construct the basis vector in $\Wei_\gamma(M(\sigma))$ as the image of a vector in $\bar{\Delta}(\Bi)$: we connect to $v$ a diagram that connects the strand labeled $i$ in the block corresponding to the word $(n,\dots, j)$ at the bottom to the longitude $a_{ij}$.  If $a_{ij}=a_{ik}$ for $j\neq k$, then these join to form a thick strand.  Note that this diagram is not unique, though there are various systematic ways of choosing a preferred one (this essentially requires choosing a preferred reduced expression for each permutation).  The semi-pattern inequalities guarantee that strands coming from the same Lyndon word don't cross.  

For example, if $n=3$, and $\chi=(0,1,2)$, then the Verma with lowest weight $\mu=0$ gives half-patterns of the form 
\[\tikz{\matrix[row sep=0mm,column sep=0mm,ampersand replacement=\&]{
\node {$0$}; \& \& \node {$1$}; \& \& \node {$2$};\\
\& \node {$a$};  \& \& \node {$b$}; \&\\
\& \&\node {$c$}; \& \&\\
};}\]
with $0\leq a\leq c$ and $1\leq b$.
Some examples of diagrams are below:
\[\tikz[very thick,baseline=15pt,yscale=1,xscale=1.2]{
\draw(1.6,0) -- node[above, at end]{$1$} (1.6,1) ; 
\draw(2.3,0) -- node[above, at end]{$2$} (2.3,1); 
\draw(1.3,0) -- node[above, at end]{$2$}(1.3,1);
\draw[red](1,0) -- node[above, at end]{$3$} (1,1) ; 
\draw[red](2,0) -- node[above, at end]{$3$} (2,1); 
\draw[red](3,0) -- node[above, at end]{$3$}(3,1);
\node at (2,-.5) {$a=0,b=1, c=0$};
}\qquad \qquad \tikz[very thick,baseline=15pt,yscale=1,xscale=1.2]{
\draw(1.6,0) to[out=45,in=-135] node[above, at end]{$1$} (3.6,1) ; 
\draw(2.3,0) -- node[above, at end]{$2$} (2.3,1); 
\draw(1.3,0) -- node[above, at end]{$2$}(1.3,1);
\draw[red](1,0) -- node[above, at end]{$3$} (1,1) ; 
\draw[red](2,0) -- node[above, at end]{$3$} (2,1); 
\draw[red](3,0) -- node[above, at end]{$3$}(3,1);
\node at (2,-.5) {$a=0,b=1, c=2$};
}\]\[\tikz[very thick,baseline=15pt,yscale=1,xscale=1.2]{
\draw(1.6,0) to[out=90,in=-90] (2.6,.8) -- node[above, at end]{$1$} (2.6,1) ; 
\draw[thickc](2.3,.8) -- node[above, at end]{$2$} (2.3,1); 
\draw(2.3,0) --  (2.3,.8); 
\draw(1.3,0) to[out=90,in=-90] (2.3,.8);
\draw[red](1,0) -- node[above, at end]{$3$} (1,1) ; 
\draw[red](2,0) -- node[above, at end]{$3$} (2,1); 
\draw[red](3,0) -- node[above, at end]{$3$}(3,1);
\node at (2,-.5) {$a=1,b=1, c=1$};
}\qquad \qquad \tikz[very thick,baseline=15pt,yscale=1,xscale=1.2]{
\draw(1.6,0) to[out=45,in=-135] node[above, at end]{$1$} (3.6,1) ; 
\draw(2.3,0) -- node[above, at end]{$2$} (2.3,1); 
\draw(1.3,0) to[out=45,in=-135] node[above, at end]{$2$}(3.3,1);
\draw[red](1,0) -- node[above, at end]{$3$} (1,1) ; 
\draw[red](2,0) -- node[above, at end]{$3$} (2,1); 
\draw[red](3,0) -- node[above, at end]{$3$}(3,1);
\node at (2,-.5) {$a=2,b=1, c=2$};
}
\]
The action on the formal span of these diagrams is given by applying the functor $\Theta$ and then acting by the resulting morphism in $\mathcal{T}_{\Z,\chi}$. In the case of $\gl_3$, this is written out in the next section; in higher rank than this, accounting for all special cases becomes overwhelming.

\begin{remark}
While Lemma \ref{lem:induction-basis} gives a basis of the module $\bar{\Delta}(\Bi)$, it is still quite challenging to calculate the structure coefficients of the action of $\tT$, and thus of the $\gl_n$-action on this space.  In our opinion, this just shows that giving hands on formulas for the action on a basis compatible with $\Gamma$ is unavoidably challenging, and makes precise the combinatorial challenge of doing so on a level close to the complexity of the Gelfand-Tsetlin formulas.
 \end{remark}

\subsection{Essential support and nilHecke algebras}

One topic which received a great deal of focus in \cite{FGRZVerma} is the {\bf essential support} of a module (\cite[Def. 5.1]{FGRZVerma}), the set of maximal ideals $\gamma$ where $\dim \Wei_{\gamma}(M)$ achieves the Futorny-Ovsienko upper bound.  For a general OGZ algebra, this bound is $\dim \Wei_{\gamma}(M)\leq \prod_{i=1}^{n-1}v_i!$ for any module $M$ cyclically generated by a single vector killed by a maximal ideal in $\Gamma$; of course, this includes all simple modules.  In the case of $\gl_n$, this bound is $(n-1)!\cdots 2!\cdot 1!$.  

By \cite[Cor. 2.19]{WebGT}, the sum of the dimensions $\Wei_{\gamma}(M)$ where $M$ ranges over all simple integral GT modules is bounded above by $\prod_{i=1}^{n-1}v_i!$, so if this bound is achieved by a single module $M$, then $M$ is the only simple GT module where this weight appears.  In fact, we can distinguish the modules where this occurs.  We call a word $\Bi$ {\bf essential} if all appearances of $i+1$ come either before or after all appearances of $i$.  For example, $333221$ is essential, whereas $332321$ or $333212$ are not.  In terms of the associated order on $\Omega$, this is the property that there is no triple $\{(i,k),(i\pm 1,m),(i,\ell)\}\subset \Omega$ with $(i,k)\prec (i\pm 1,m) \prec (i,\ell)$

\begin{theorem}\label{thm:essential}The following are equivalent for a simple $M$ GT module and $\gamma\in \MaxSpec_{\Z}(\Gamma)$ in its support:
\begin{itemize}
    \item   The weight $\gamma$ lies in the essential support of the simple module $M$.
   \item The maximal ideal $\gamma$ is non-critical, and the word $\Bi(\gamma)$ is essential.
  \item The algebra $e(\gamma)\tT e(\gamma)$ is isomorphic to the tensor product $NH_{v_1}\otimes NH_{v_2}\otimes NH_{v_{n-1}}\otimes \bbC[y_{n,1},\dots, y_{n,v_n}]$.
\end{itemize}
\end{theorem}
\begin{proof}
  $(1)\Rightarrow (2)$: If $\gamma$ is critical, then the bound cannot be achieved by \cite[Cor. 2.19]{WebGT}.  If $\gamma$ is in the essential support, then we have that $e(\gamma)\mathbb{KLR}M$ is isomorphic to $U=e(\gamma)\tT e(\gamma)/\sum_ie(\gamma)\tT e(\gamma)y_i$, the quotient  by the left ideal generated by all dots.  On the other hand, $e(\gamma)\tT e(\gamma)$ also acts on the coinvariant quotient of the polynomial ring $e(\gamma)\poly$, which must also be simple for dimension reasons.  Furthermore, this isomorphism must send $[1]\in U$ to the unique element of the coinvariant quotient killed by the dots, which has degree $\sum_{i=1}^{n-1} v_i(v_i-1)$.  If we think of this coinvariant quotient of as the cohomology of the product of flag varieties on $\bbC^{v_i}$, this is the fundamental class.
  
  Thus, for this isomorphism to hold, there must be an element of degree $-\sum_{i=1}^{n-1} v_i(v_i-1)$ which maps to the unit of the coinvariant quotient.  This must be given by a diagram in which there are no dots and all pairs of strands with the same label cross (i.e. we do the half-twist on the strands with each fixed label).  This will have the correct degree if and only if every pair of strands with different labels crossing has degree 0, i.e. if their labels are not consecutive.  This will hold if and only if the word is essential.  
  
  $(2)\Rightarrow (3)$: If the word $\Bi$ is essential, then we have a map of algebras $NH_{v_1}\otimes NH_{v_2}\otimes NH_{v_{n-1}}\otimes \bbC[y_{n,1},\dots, y_{n,v_n}]$ to $e(\gamma)\tT e(\gamma)$ by simply superimposing diagrams.  This map is an isomorphism by \cite[Prop. 4.16]{Webmerged}.  
  
  $(3)\Rightarrow (1)$: The weight space $\Wei_{\gamma}(M)$ is a simple $e(\gamma)\tT e(\gamma)$-module by applying the equivalence $\Theta$ and \cite[Thm. 18]{FOD}.  Since $NH_{v_1}\otimes NH_{v_2}\otimes NH_{v_{n-1}}\otimes \bbC[y_{n,1},\dots, y_{n,v_n}]$ is a matrix algebra of rank $\prod_{i=1}^{n-1}v_i!$ over its center by \cite[Prop. 3.5]{laudaCategorificationQuantum2010}, this simple module has dimension $\prod_{i=1}^{n-1}v_i!$, so $\gamma$ is in the essential support.
\end{proof}
Since each essential word appears in the support of a unique simple, it is natural to ask when two essential words are in the support of a single simple.  We call two essential words {\bf essentially the same} if the unique diagram joining the corresponding idempotents with no crossings of strands of the same label has degree 0 (it crosses no strands with consecutive labels), and {\bf essentially different} otherwise.  You can check this means that for each $i$, both words have all appearances of $i+1$ before all appearances of $i$, or both have all $i+1$'s after all $i$'s.  For example, the words $333122$, $313322$, and $133322$ are all essential and all essentially the same, whereas $333221$ is essentially different from these words.  You can easily check that essential sameness is an equivalence relation.
\begin{theorem}\label{thm:essentially-same}
  Non-critical $\gamma$ and $\gamma'$ will both appear in the essential support of a simple module $M$ if and only if the words $\Bi(\gamma),\Bi(\gamma')$ are both essential and essentially the same.
\end{theorem}
\begin{proof}
  If $\Bi(\gamma),\Bi(\gamma')$ are both essential and essentially the same, then the straight line diagrams provide isomorphisms between $\gamma$ and $\gamma'$ in $\mathcal{C}$ (using the equivalence $\Theta$), so the corresponding GT weight spaces have the same dimension in all GT modules.  This shows that $\gamma$ is in the essential support of $M$ if and only if $\gamma'$ is as well.
  
  Conversely, if $\gamma$ and $\gamma'$ both appear in the essential support of a simple module $M$, then we can consider the module $N=\mathbb{KLR}(M)$, and grade it so that the highest degree element $m$ appearing in $e(\gamma)N$ or $e(\gamma')N$ has degree 0.  By symmetry, we may assume that this is in $e(\gamma)N$; thus, the lowest degree element of $e(\gamma')N$ has degree $\leq -\sum_{i=1}^{n-1} v_i(v_i-1)$.  Since $m$ generates $N$, there must be a diagram in $\tT$ sending $m$ to a non-zero element of $e(\gamma')N$ which has degree $\leq -\sum_{i=1}^{n-1} v_i(v_i-1)$.  This must be the diagram with no dots that connects these idempotents crossing all strands with the same label. This is a lowest degree diagram connecting these idempotents, which has degree $ -\sum_{i=1}^{n-1} v_i(v_i-1)$.  Thus, if we remove the crossings of strands with the same label, the diagram has degree $0$, and so $\gamma$ and $\gamma'$ are essentially the same.  
\end{proof}

If a word is not essential, then we can still use the appearance of nilHecke algebras in its endomorphism algebra to constrain the possible multiplicities of GT weight spaces.
\begin{proposition}\label{prop:smaller-nilHecke}
Assume that $\gamma\in \MaxSpec_{\Z}(\Gamma)$ is non-critical.  If the word $\Bi$ has a group of $n_1$ consecutive appearances of $i_1$, $n_2$ consecutive appearances of $i_2$, etc. up to $n_p$ consecutive appearances of $i_p$ then the multiplicity $\dim\Wei_{\gamma}(M)$ is divisible by $n_1!\cdots n_p!$ for any GT module $M$.  
\end{proposition}
\begin{proof}
  Diagrams which only cross strands in the groups discussed define a non-trivial map $NH_{n_1}\otimes \cdots \otimes NH_{n_p}\to e(\gamma)\tT e(\gamma)$, which then induces a module structure on $\Wei_{\gamma}(M)$.  Since $NH_{n_1}\otimes \cdots \otimes NH_{n_p}$ has a unique graded simple module with dimension $n_1!\cdots n_p!$ so the dimension of any weakly graded module is divisible by this number.
\end{proof}

\section{Low rank computations}
\subsection{The case of $\mathfrak{gl}_2$ modules}

While Gelfand-Tsetlin modules for $\mathfrak{gl}_2$ are well-understood, we think it will be instructive to illustrate our approach in this case.  

In the $\gl_2$ case we have $Z(\mcU(\gl_1)) = e_{11}$ since it is a one-dimensional algebra and $Z(\mcU(\gl_2))$ can be calculated from the Capelli determinant as 
\begin{align}
    C(t)&=   \left| \begin{matrix}e_{11} +1 -t& e_{12}\\e_{21} & e_{22} -t \end{matrix} \right| = (e_{11}+1-t)(e_{22} -t) - e_{21}e_{12} \\
    &= t(t-1) -t(e_{11} + e_{22} ) + e_{11} e_{22} - e_{21}e_{12} + e_{22} .
\end{align} 
So we see $\Gamma = \langle e_{11}, e_{11} + e_{22},e_{11} e_{22} - e_{21}e_{12} + e_{22} \rangle$, and these elements generate $\Gamma$ as a polynomial ring. 

Consider a homomorphism $\mu\colon \Gamma\to \bbC$.  Let $a_{11}=\mu(e_{11})$; the image $\mu(C(t))$ is a quadratic polynomial and we let $a_{21},a_{22}$ be its roots.  If $a_{21}\neq a_{22}$, the completion $\widehat{\Gamma}_{\mu}$ of $\Gamma$ with respect to this maximal ideal is a polynomial ring is a power series ring $\bbC[[\la_{11}-a_{11},\la_{21}-a_{21},\la_{22}-a_{22}]]$ with the isomorphism defined by $\la_{11}\mapsto e_{11}$ and $C(t)\mapsto (t-\lambda_{21})(t-\lambda_{22})$; the inverse to this map can be constructed using Hensel's lemma to lift the factorization $\mu(C(t))=(t-a_{21})(t-a_{22})$ to the completion. If $a_{21}=a_{22}$, then the same map is injective with image given by the power series symmetric under switching $\lambda_{21}$ and $\lambda_{22}$.  

We let $\Ba=(a_{11},\{a_{21},a_{22}\})$ and $\Bb=(b_{11},\{b_{21},b_{22}\})$ index two such maximal ideals; we can let these represent the corresponding objects in the category $\mathcal{C}$.  

By definition, the morphism space in $\mathcal{C}$ between these is, by definition: \[\Hom_{\mathcal{C}}(\mathbf{a},\mathbf{b})=\mcU/(\mcU\mathfrak{m}_{\mathbf{a}}^N+\mathfrak{m}_{\mathbf{b}}^N\mcU)\]
One can easily check using the PBW theorem that $U(\mathfrak{gl}_2)$ is a free left $\Gamma$-module with basis $\{\dots, e_{12}^2,e_{12},1,e_{21},e_{21}^2,\dots\}$, and Hom space $\Hom_{\mathcal{C}}(\mathbf{a},\mathbf{b})$ is 0 unless $\{a_{21},a_{22}\}=\{b_{21},b_{22}\}$ and $n=a_{11}-b_{11}\in \Z$.  If $n>0$, then $e_{21}^n$ generates $\Hom_{\mathcal{C}}(\mathbf{a},\mathbf{b})$ as a free left module over the completion $\widehat{\Gamma}_{\mathbf{b}}$; if $n<0$, then $e_{12}^{-n}$ plays the same role, and if $n=0$, then the Hom space is simply $\widehat{\Gamma}_{\mathbf{b}}$.  

In the notation introduced earlier when $b_{11}=a_{11}-1$ and $\{a_{21},a_{22}\}=\{b_{21},b_{22}\}$, the image of $e_{21}$ is the morphism $X_i^-(\Ba,\Bb)$, and the image of $e_{12}$ is the morphism $X_i^+(\Bb,\Ba)$.  From the formulas (\ref{u_on_polys:start}--\ref{u_on_polys:end}), we have that \[e_{12}e_{21}=-\lambda_{11}^2+(\la_{21}+\la_{22})\lambda_{11}-\la_{21}\la_{22}.\]
This is invertible if $a_{11}\notin \{a_{21},a_{22}\}$: in the case, we have 
\[\tilde{X}_i^+(\Bb,\Ba)=\frac{1}{-\lambda_{11}^2+(\la_{21}+\la_{22})\lambda_{11}-\la_{21}\la_{22}}X_i^+(\Bb,\Ba)\qquad \tilde{X}_i^-(\Ba,\Bb)=X_i^-(\Ba,\Bb).\]

If $a_{11}=a_{21}\neq a_{22}$, then 
\[\tilde{X}_i^+(\Bb,\Ba)=\frac{1}{\lambda_{11}-\la_{22}}X_i^+(\Bb,\Ba)\qquad \tilde{X}_i^-(\Ba,\Bb)=X_i^-(\Ba,\Bb).\] Note that the case of $a_{11}=a_{22}\neq a_{21}$ is equivalent since this is just the relabeling of the roots of $\mu(C(t))$.  
If $a_{11}=a_{21}=a_{22}$, then 
\[\tilde{X}_i^+(\Bb,\Ba)=X_i^+(\Bb,\Ba)\qquad \tilde{X}_i^-(\Ba,\Bb)=X_i^-(\Ba,\Bb).\]
In each of these, we divide by all invertible factors of $\mu(C(t))$.

These correspond to diagrams with two red strands and one black strand; the $x$-values $\{a_{21},a_{22}\}$ with $a_{21}\leq a_{22}$ specify the position of the red strands in the idempotent $e(\Ba)$ and $a_{11}$ specifies the position of the black strand, with the tie break rule that the black strand is right of the red strand if these $x$-values agree. Thus, we have that:
\begin{equation}\label{eq:gl2-idempotents}
e(\Ba)=\begin{cases}e(1,2,2) & a_{11}<a_{21}\leq a_{22}\\
e(2,1,2) & a_{21}\leq a_{11}<a_{22}\\
e(2,2,1) & a_{21}\leq a_{22}\leq a_{11}
\end{cases}
\end{equation}
More explicitly, fix a $\tT$-module $M$. Keeping  $\{a_{21},a_{22}\}\subset \Z$ (that is, the maximal ideal of $Z_2$) fixed, we can define a Gelfand-Tsetlin module $\mathbb{GT}(M)=\bigoplus_{a_{11}\in\Z}e(\Ba)M$.  In order to describe the action, it will be useful to use the notation $y_{ij}=y_{\wp(i,j)}$ for the dot on the $j$th strand with label $i$ read from right to left.  Using the formulas for $\tilde{X}_1^{\pm}$ above, we see that the $\mathfrak{gl}_2$ action on this sum is given by 
\newseq
\begin{align*}
    e_{11}m &= (a_{11}+y_{11})m\\
    (e_{11}+e_{22})m &= (a_{21}+a_{22}+y_{21}+y_{22})m\\
\end{align*}
Now, we give the formulas for the action of $X_1^+$.  When $a_{11}\notin \{a_{21}-1,a_{22}-1\}$:
\begin{align*}\subeqn\label{eq:X1p-action}
    X^+_1m&= -{\prod_{i=1}^2(y_{11}-y_{2i}+a_{11}-1-a_{2i})}\cdot m \\
\shortintertext{When $a_{11}=a_{21}-1\neq a_{22}-1$:} X^+_1m&= (y_{22}-y_{11}+a_{22}+1-a_{11})
    \tikz[baseline,scale=.4,very thick]{\draw (1,1) -- node[below, at end]{$1$}(-1,-1);\draw[red] (-1,1) -- node[below, at end]{$2$}(1,-1);  \draw[red] (3,1) -- node[below, at end]{$2$}(3,-1); }\cdot m  \subeqn\\
    \shortintertext{When $a_{11}=a_{22}-1\neq a_{21}-1$:} 
    X^+_1m&= 
    (y_{21}-y_{11}+a_{21}+1-a_{11})
        \tikz[baseline,scale=.4,very thick]{\draw (3,1) -- node[below, at end]{$1$}(1,-1);\draw[red] (-1,1) -- node[below, at end]{$2$}(-1,-1);  \draw[red] (1,1) -- node[below, at end]{$2$}(3,-1); }\cdot m \\
        \shortintertext{When $a_{11}=a_{22}-1=a_{21}-1$:} 
    X^+_1m&= -
        \tikz[baseline,scale=.4,very thick]{\draw (3,1) -- node[below, at end]{$1$}(-1,-1);\draw[red] (-1,1) -- node[below, at end]{$2$}(1,-1);  \draw[red] (1,1) -- node[below, at end]{$2$}(3,-1); }\cdot m \subeqn
\end{align*}
For the formulas for $X_1^-$ are similar, but simpler since we have no terms in the numerator.  When $a_{11}\notin \{a_{21},a_{22}\}$:\newseq
\begin{align*}\subeqn\label{eq:X1m-action}
    X^-_1m&=   m \\ 
    \shortintertext{When $a_{11}=a_{21}\neq a_{22}$:} X^-_1m&=  
    \tikz[baseline,scale=.4,very thick]{\draw[red] (1,1) -- node[below, at end]{$2$}(-1,-1);\draw (-1,1) -- node[below, at end]{$1$}(1,-1);  \draw[red] (3,1) -- node[below, at end]{$2$}(3,-1); }\cdot m  \subeqn\\    \shortintertext{When $a_{11}=a_{22}\neq a_{21}$:} X^-_1m&=  
        \tikz[baseline,scale=.4,very thick]{\draw[red] (3,1) -- node[below, at end]{$2$}(1,-1);\draw[red] (-1,1) -- node[below, at end]{$2$}(-1,-1);  \draw (1,1) -- node[below, at end]{$1$}(3,-1); }\cdot m  \subeqn\\
    \shortintertext{When $a_{11}=a_{22}=a_{21}$:} X^-_1m&=         \tikz[baseline,scale=.4,very thick]{\draw (3,1) -- node[below, at end]{$1$}(-1,-1);\draw[red] (-1,1) -- node[below, at end]{$2$}(1,-1);  \draw[red] (1,1) -- node[below, at end]{$2$}(3,-1); }\cdot m \subeqn
\end{align*}
Note that the corresponding category of GT modules has 3 simple modules, corresponding to the 3 simple 1-d $\mathbb{\tilde{T}}$-modules.  In each of these simples, all positive degree elements act by 0 and they are distinguished by whether $e(1,2,2)$, $e(2,1,2)$ and $e(2,2,1)$ act by the identity. Since in the case, GT weight spaces are the same as those of the Cartan, the modules $\mathbb{GT}(M)$ have 1-dimensional weight spaces, which are non-zero when $a_{11}$ satisfies the equality in one of the lines of \eqref{eq:gl2-idempotents}.  

Thus, if we consider the subcategory $\mathcal{C}_{\Z,\chi}$, this means that we fix $\{a_{21},a_{22}\}\subset  \Z$ with $a_{21}\leq a_{22}$ and consider all possible $a_{11}\in \Z$.  In general, for a $\mathbb{\tilde{T}}$-module $M$, the image $e(1,2,2)M$ is isomorphic to the weight space of $\mathbb{GT}(M)$ for $\Ba$ with $a_{11}<a_{21}$, analogously with $e(2,1,2)M$ when $a_{21}\leq a_{11}<a_{22}$ and $e(2,2,1)M$ when $a_{22}\leq a_{11}$; thus, the three simples are a Verma module for the positive Borel, the unique finite-dimensional module, and a Verma module for the negative Borel, as we expect.  When $a_{21}=a_{22}$, the idempotent $e(2,1,2)$ doesn't have a corresponding weight space, so we lose the corresponding finite-dimensional simple; again, this is as expected for a singular block of $\mcU(\mathfrak{gl}_2)$-modules.

\subsection{The case of $\mathfrak{gl}_3$ modules}

Let us now discuss how these results apply to $\mathfrak{gl}_3.$

\subsubsection{Integral GT modules} First, let us expound a bit on the meaning of Theorem \ref{thm:GT-KLR} in this case.  Fix a regular integral maximal ideal $\chi\in \MaxSpec_{\Z}(Z_3)$, that is, fix integers $a_{31}<a_{32}<a_{33}$.  Choosing $\gamma$ lying over this in $\mathcal{S}=\MaxSpec_{\Z,\chi}(\Gamma)$ entails choosing $a_{11},a_{21},a_{22}\in \Z$.  Each such $\gamma$ gives an idempotent $e(\mathbf{I})\in \tT$; note that this idempotent will only include a thick strand if $a_{21}=a_{22}$. 

Thus, given a weakly graded finite-dimensional module $M$ over $\tT$, we can define the corresponding $\mcU(\mathfrak{gl}_3)$-module by considering the sum
\[\mathbb{GT}(M)=\sum_{\Ba\in \mathcal{S}}e(\Ba)M.\]
As before,  we use the notation $y_{ij}=y_{\wp(i,j)}$ for the dot on the $j$th strand with label $i$ read from right to left.
Now, let us describe the action of the Chevalley generators for $m\in e(\Ba)M$. 

The operators $X^{\pm}_1$ send $m$ to the weight space for $e(\Ba')M$; note that typically, this is the same idempotent, unless the $a_{*,*}$ line up so that two strands cross; this happens when $a_{2*}-a_{11}$ switches between 1 and 0.  These are given by the same formulas as (\ref{eq:X1p-action}--\ref{eq:X1m-action}), ignoring the position of the strands with the label $3$.

This was just a warm-up: Now we consider the operators $X_2^\pm$.  In the interest of space, let \[p_j=\prod_{a_{3i}\neq a_{2j}}(y_{2j}-y_{3i}+a_{2j}+1 - a_{3i} ).\]
First, let us consider the case where $a_{21}\neq a_{22}$.  In this case $X_2^\pm$ is a sum of elements in two weight spaces: \[\Ba'=(a_{11},a_{21}\pm 1,a_{22})\qquad \Ba''=(a_{11},a_{21},a_{22}\pm 1).\] I'll write $m_1,m_2$ for the instance of $m$ associated with these weight spaces.  If $a_{2i}+1\notin \{a_{31},a_{32},a_{33},a_{22}\}$, then $e(\Ba)=e(\Ba')=e(\Ba'')$, so $m_1$ and $m_2$ are already in the correct idempotent images (though note that these are still different elements of $\mathbb{GT}(M)$).  
In other cases, we will need to apply a diagram to them.  Similarly, if $a_{21}=a_{22}$, then only $\Ba''$ is dominant (and $\Ba'$ is a permutation of it), so we simply write $m$ for the single instance we need in this case.
To write the formula for $X_2^+$, we consider the diagrams.
\begin{equation*}
    d_i^+=\begin{dcases}
    1 &\text{ when } a_{2i}+1\notin \{a_{31},a_{32},a_{33},a_{22}\}\\
    \tikz[baseline,scale=.4,very thick]{\draw[red] (3,1) -- node[below, at end]{$3$}(3,-1);\draw[red] (1,1) -- node[below, at end]{$3$}(1,-1);\draw (-1,1) -- node[below, at end]{$2$}(-3,-1);\draw[red] (-3,1) -- node[below, at end]{$3$}(-1,-1);   } &\text{ when } a_{2i}+1=a_{31}\neq a_{22}\\
    \tikz[baseline,scale=.4,very thick]{\draw[red] (3,1) -- node[below, at end]{$3$}(3,-1);\draw[red] (-1,1) -- node[below, at end]{$3$}(1,-1);\draw[red] (-3,1) -- node[below, at end]{$3$}(-3,-1);\draw (1,1) -- node[below, at end]{$2$}(-1,-1);   } &\text{ when } a_{2i}+1=a_{32}\neq a_{22}\\
    \tikz[baseline,scale=.4,very thick]{\draw[red] (1,1) -- node[below, at end]{$3$}(3,-1);\draw[red] (-1,1) -- node[below, at end]{$3$}(-1,-1);\draw[red] (-3,1) -- node[below, at end]{$3$}(-3,-1);\draw (3,1) -- node[below, at end]{$2$}(1,-1);   } &\text{ when } a_{2i}+1=a_{33}\neq a_{22}\\
    \end{dcases} 
    \end{equation*}
    The cases $a_{21}+1=a_{22}$ and  $a_{21}=a_{22}$ are special.  In the former case, we have a special formula for $d_1^+$, and in the latter, we only need a single diagram $d^+$
    \begin{equation*}d_1^+=
    \begin{dcases}
    \tikz[baseline,scale=.4,very thick]{\draw (1,0) -- node[below, at end]{$2$}(-1,-1); \draw (1,0) -- node[below, at end]{$2$}(1,-1); \draw[thickc] (1,0) -- (1,1);  }&\text{ when } a_{21}+1=a_{22}\neq a_{3j}\\
    \tikz[baseline,scale=.4,very thick]{\draw[red] (3,1) -- node[below, at end]{$3$}(3,-1);\draw[red] (1,1) -- node[below, at end]{$3$}(1,-1);\draw (-1,0) -- node[below, at end]{$2$}(-3,-1);\draw[red] (-2,1) -- node[below, at end]{$3$}(-2,-1);  \draw (-1,0) -- node[below, at end]{$2$}(-1,-1);  \draw[thickc] (-1,0) -- (-1,1); }&\text{ when } a_{21}+1=a_{22}=a_{31}\\
\tikz[baseline,scale=.4,very thick]{\draw[red] (-5,1) -- node[below, at end]{$3$}(-5,-1);\draw[red] (1,1) -- node[below, at end]{$3$}(1,-1);\draw (-1,0) -- node[below, at end]{$2$}(-3,-1);\draw[red] (-2,1) -- node[below, at end]{$3$}(-2,-1);  \draw (-1,0) -- node[below, at end]{$2$}(-1,-1);  \draw[thickc] (-1,0) -- (-1,1); } &\text{ when } a_{21}+1=a_{22}=a_{32}\\
    \tikz[baseline,scale=.4,very thick]{\draw[red] (-5,1) -- node[below, at end]{$3$}(-5,-1);\draw[red] (-7,1) -- node[below, at end]{$3$}(-7,-1);\draw (-1,0) -- node[below, at end]{$2$}(-3,-1);\draw[red] (-2,1) -- node[below, at end]{$3$}(-2,-1);  \draw (-1,0) -- node[below, at end]{$2$}(-1,-1);  \draw[thickc] (-1,0) -- (-1,1); } &\text{ when } a_{21}+1=a_{22}=a_{33}\\
            \end{dcases}
            \end{equation*}
            \begin{equation*}            d^+=
            \begin{dcases}
        \tikz[baseline,scale=.4,very thick]{\draw (1,0) -- (3,1); \draw (1,0) -- (1,1); \draw[thickc] (1,0) -- node[below, at end]{$2$}(1,-1);  }&\text{ when } a_{21}=a_{22}\neq a_{3j}-1\\
    \tikz[baseline,scale=.4,very thick]{\draw[red] (3,1) -- node[below, at end]{$3$}(3,-1);\draw[red] (1,1) -- node[below, at end]{$3$}(1,-1);\draw[red] (-2,1) -- node[below, at end]{$3$}(-2,-1);  \draw (-3,0) -- (-1,1); \draw (-3,0) -- (-3,1); \draw[thickc] (-3,0) -- node[below, at end]{$2$}(-3,-1); }&\text{ when } a_{21}=a_{22}=a_{31}-1\\
\tikz[baseline,scale=.4,very thick]{\draw[red] (-5,1) -- node[below, at end]{$3$}(-5,-1);\draw [red](1,1) -- node[below, at end]{$3$}(1,-1);\draw[red] (-2,1) -- node[below, at end]{$3$}(-2,-1);\draw (-3,0) -- (-1,1); \draw (-3,0) -- (-3,1); \draw[thickc] (-3,0) -- node[below, at end]{$2$}(-3,-1); } &\text{ when } a_{21}=a_{22}=a_{32}-1\\
    \tikz[baseline,scale=.4,very thick]{\draw[red] (-5,1) -- node[below, at end]{$3$}(-5,-1);\draw[red] (-7,1) -- node[below, at end]{$3$}(-7,-1);\draw[red] (-2,1) -- node[below, at end]{$3$}(-2,-1);  \draw (-3,0) -- (-1,1); \draw (-3,0) -- (-3,1); \draw[thickc] (-3,0) -- node[below, at end]{$2$}(-3,-1);} &\text{ when } a_{21}=a_{22}=a_{33}-1\\
    \end{dcases}
\end{equation*}
In each of the pictures above, we have left out ``irrelevant'' strands whose position is not fixed by our conditions: the other strand with the label $2$ in the cases of $a_{21}+1\neq a_{22} $, and all strands with label $3$ if $a_{2i}+1\notin \{a_{31},a_{32},a_{33}\}$.
Since we focus on the regular case, we have left out cases where the $a_{3*}$'s are not distinct, though the reader should be able to work out the pattern.
Having made all these definitions, applying the formula \eqref{eq:tXp-KLR} shows that:
\begin{lemma}
\begin{equation*}
    X^+_2m=\begin{cases}[r]\displaystyle
        \displaystyle\frac{d_1^+p_1m_1}{y_{22}-y_{21}+a_{22} - a_{21}-1} -\frac{d_2^+p_2m_2}{y_{22}-y_{21}+a_{22} - a_{21}+1}
             &\text{ when } a_{22}-a_{21}\neq 0,1\\
        \displaystyle
        -d_1^+p_1m_1-\frac{d_2^+p_2m_2}{y_{22}-y_{21}+a_{22} - a_{21}+1} &\text{ when } a_{22}-a_{21}=1\\
        p_2d^+m &\text{ when } a_{22}-a_{21}=0
    \end{cases}
\end{equation*}
\end{lemma}

Similarly, for $X_2^-$, we define diagrams $d_i^-$ when $a_{21}\neq a_{22}$ and $d^-$ when $a_{21}=a_{22}$; these are just the reflections in a vertical line of $d^+_i$.  We let 
\[q_i=\begin{cases}1 &\text{ when } a_{2i}=a_{11}+1\\
y_{2i}-y_{11}+a_{2i}-a_{11}-1&\text{ when } a_{2i}\neq a_{11}+1\end{cases}\]
We then applying \eqref{eq:tXm-KLR} shows:
\begin{lemma}
\begin{equation*}
    X^+_2m=\begin{cases}[r]
        \displaystyle\frac{d_1^-q_1m_1}{y_{22}-y_{21}+a_{22} - a_{21}+1} -\frac{d_2^+q_2m_2}{y_{22}-y_{21}+a_{22} - a_{21}-1}
             &\text{ when } a_{22}-a_{21}\neq -1,0\\
        \displaystyle
            \displaystyle\frac{d_1^-q_1m_1}{y_{22}-y_{21}+a_{22} - a_{21}+1} +d_2^+q_2m_2
         &\text{ when } a_{22}-a_{21}=-1\\
        q_1d^-m &\text{ when } a_{22}-a_{21}=0
    \end{cases}
\end{equation*}
\end{lemma}
\subsubsection{The classification of simple modules}
As discussed in \cite{futornyGelfandTsetlinModules2018,WebGT}, in the principal block of $\mathfrak{sl}_3$ (or more generally, any regular integral block), there are 20 simple modules up to isomorphism.  These include:
\begin{enumerate}
\item one finite dimensional module
\item 6 irreducible anti-dominant Verma modules for the 6 different Borels containing the standard torus
\item 12 other simples in category $\mathcal{O}$ for some Borel; there are 4 simples in category $\mathcal{O}$ for a fixed Borel which are neither finite dimensional nor a Verma module, and each of these lies in category $\cO$ for exactly 2 of the 6 Borels.
\item one simple which does not lie in category $\cO$ for any Borel (in fact, it has infinite weight multiplicities).
\end{enumerate}
As an illustration of our approach to Gelfand-Tsetlin modules, let us explain how to arrive at this list using KLR algebras.  By Theorem \ref{thm:all-simples}, the simple $\tT$-modules for $\Bv=(1,2,3)$ are classified by red-good words.    This realizes each simple module as a quotient of a standard module.  In full generality, finding this quotient can be quite complicated, but in the case of $\mathfrak{sl}_3$, things are relatively simple.

In particular, since the group $W$ in this case is just $S_2$, all integral weights are at worst 1-singular, and we can utilize \cite[Cor. 3.6]{WebGT}; thus, every weight either appears in a unique GT-module or in two of them.  In KLR terms, we can separate this possibility by looking for the unique diagram with no dots, a single crossing of black strands with label 2, and $e(\gamma)$ at both top and bottom.  We can separate the 3 cases of \cite[Cor. 3.6]{WebGT} by the degree of this element:
\begin{enumerate}
    \item If the degree is -2, then $e(\gamma)\tT e(\gamma)$ is a 2$\times$2 matrix algebra, and so there is a unique simple GT module where $\Wei_\gamma(S)$ is 2-dimensional (and it is 0-dimensional in all other simples).  
    \item If the degree is 0, then $e(\gamma)\tT e(\gamma)$ has two distinct 1-dimensional simples, so there are two simple GT modules where this weight space is 1-dimensional (and it is 0-dimensional in all other simples).
     \item If the degree is $> 0$, then $e(\gamma)\tT e(\gamma)$ is strictly positively graded and has a unique simple module which is 1-dimensional, so there is a unique simple GT module where this weight space is 1-dimensional (and it is 0-dimensional in all other simples).
\end{enumerate}

In Figures \ref{fig:list1} and \ref{fig:list2} below, we show a basis of each module; for each of these, we write the corresponding good word as well as the indexing of the corresponding $\mcU$-module in the notation of \cite{futornyGelfandTsetlinModules2018}.  In each case, we have a box which represents a generating vector; the idempotent of which this vector is in the image can be read off from the first diagram.  Note that these are not necessarily good words, but are of the form (1) or (3) in the classification above, and thus only appear in a single simple module; the structure coefficients for diagrams acting on our basis vectors are implicit here, since they are simply all diagrams in $\tT e(\gamma)$ which generate a proper submodule of this projective left module.  We leave to the reader to verify that these modules are simple; this might sound daunting, but one can easily reduce to showing that the basis vectors shown each generate the module, using the application of idempotents to a general linear combination and the observation that each weight space is either 1-dimensional or a simple 2-dimensional module over $e(\gamma')\tT e(\gamma')$, when $\gamma'$ is of type (1).  

We note that this description also makes it simple to match with the classification of \cite{futornyGelfandTsetlinModules2018}, since one need only read off the good word from their description of the weight spaces.  Since they use the principal block, some of the idempotents in the list above might not be realized as $e(\gamma)$, but those corresponding to good words will be.

\begin{figure}\thisfloatpagestyle{empty} \centering
\begin{tabular}{|c|c|}\hline
\parbox[c]{1.5cm}{\centering 332321\\$L_{32}$} & 
\tikz[very thick,baseline=15pt,yscale=.8,xscale=.5]{
\draw(6,0) -- node[above, at end]{$1$} (6,1) ; 
\draw[thickc] (4.5,0) -- (4.5,.5);
\draw(4.5,.5) -- node[above, at end]{$2$} (4,1); 
\draw(4.5,.5) -- node[above, at end]{$2$}(5,1);
\draw(1,0) -- node[above, at end]{$3$} (1,1) ; 
\draw(2,0) -- node[above, at end]{$3$} (2,1); 
\draw(3,0) -- node[above, at end]{$3$}(3,1);
\node[draw,fill=white,inner xsep=40pt] at (3.5,0){};}\qquad
\tikz[very thick,baseline=15pt,yscale=.8,xscale=.5]{
\draw(6,0) -- node[above, at end]{$1$} (6,1) ; 
\draw[thickc] (4.5,0) -- (4.5,.5);
\draw(4.5,.5) -- node[above, at end]{$2$} node[midway,circle,fill=black,inner sep=2pt]{}(4,1); 
\draw(4.5,.5) -- node[above, at end]{$2$}(5,1);
\draw(1,0) -- node[above, at end]{$3$} (1,1) ; 
\draw(2,0) -- node[above, at end]{$3$} (2,1); 
\draw(3,0) -- node[above, at end]{$3$}(3,1);
\node[draw,fill=white,inner xsep=40pt] at (3.5,0){};}\qquad 
\tikz[very thick,baseline=15pt,yscale=.8,xscale=.5]{
\draw(6,0) -- node[above, at end]{$1$} (5,1) ; 
\draw[thickc] (4.5,0) -- (4.5,.5);
\draw(4.5,.5) -- node[above, at end]{$2$}(4,1); 
\draw(4.5,.5) -- node[above, at end]{$2$}(6,1);
\draw(1,0) -- node[above, at end]{$3$} (1,1) ; 
\draw(2,0) -- node[above, at end]{$3$} (2,1); 
\draw(3,0) -- node[above, at end]{$3$}(3,1);
\node[draw,fill=white,inner xsep=40pt] at (3.5,0){};}\qquad 
\tikz[very thick,baseline=15pt,yscale=.8,xscale=.5]{
\draw(6,0) -- node[above, at end]{$1$} (6,1) ; 
\draw[thickc] (4.5,0) -- (4.5,.5);
\draw(4.5,.5) -- node[above, at end]{$2$}(3,1); 
\draw(4.5,.5) -- node[above, at end]{$2$}(5,1);
\draw(1,0) -- node[above, at end]{$3$} (1,1) ; 
\draw(2,0) -- node[above, at end]{$3$} (2,1); 
\draw(3,0) -- node[above, at end]{$3$}(4,1);
\node[draw,fill=white,inner xsep=40pt] at (3.5,0){};}
\\
\hline
\end{tabular}
    \caption{Basis vectors for simple $\tT$-modules, part I}
    \label{fig:list1}
\end{figure}
\begin{figure}\thisfloatpagestyle{empty} 
    \centering
    \begin{tabular}{|c|c|}\hline

\parbox[c]{1.5cm}{\centering 212333\\ $L_{4}$ }& 
\tikz[very thick,baseline=15pt,yscale=.8,xscale=.5]{
\draw[thickc] (1.5,0) -- (1.5,.5);
\draw(1.5,.5) --  node[above, at end]{$2$} (1,1); 
\draw(1.5,.5) -- node[above, at end]{$2$}(2,1);
\draw(3,0) -- node[above, at end]{$1$} (3,1) ; 
\draw(4,0) -- node[above, at end]{$3$} (4,1); 
\draw(5,0) -- node[above, at end]{$3$}(5,1);
\draw(6,0) -- node[above, at end]{$3$} (6,1) ; 
\node[draw,fill=white,inner xsep=40pt] at (3.5,0){};}\qquad 
\tikz[very thick,baseline=15pt,yscale=.8,xscale=.5]{
\draw[thickc] (1.5,0) -- (1.5,.5);
\draw(1.5,.5) --   node[midway,circle,fill=black,inner sep=2pt]{} node[above, at end]{$2$} (1,1); 
\draw(1.5,.5) -- node[above, at end]{$2$}(2,1);
\draw(3,0) -- node[above, at end]{$1$} (3,1) ; 
\draw(4,0) -- node[above, at end]{$3$} (4,1); 
\draw(5,0) -- node[above, at end]{$3$}(5,1);
\draw(6,0) -- node[above, at end]{$3$} (6,1) ; 
\node[draw,fill=white,inner xsep=40pt] at (3.5,0){};}\qquad 
\tikz[very thick,baseline=15pt,yscale=.8,xscale=.5]{\draw[thickc] (1.5,0) -- (1.5,.5);
\draw(1.5,.5) --  node[above, at end]{$2$} (1,1); 
\draw(1.5,.5) -- node[above, at end]{$2$}(2,1);
\draw(3,0) -- node[above, at end]{$1$} (4,1) ; 
\draw(4,0) -- node[above, at end]{$3$} (3,1); 
\draw(5,0) -- node[above, at end]{$3$}(5,1);
\draw(6,0) -- node[above, at end]{$3$} (6,1) ; 
\node[draw,fill=white,inner xsep=40pt] at (3.5,0){};}\qquad 
\tikz[very thick,baseline=15pt,yscale=.8,xscale=.5]{
\draw[thickc] (1.5,0) -- (1.5,.5);
\draw(1.5,.5) --   node[midway,circle,fill=black,inner sep=2pt]{} node[above, at end]{$2$} (1,1); 
\draw(1.5,.5) -- node[above, at end]{$2$}(2,1);
\draw(3,0) -- node[above, at end]{$1$} (4,1) ; 
\draw(4,0) -- node[above, at end]{$3$} (3,1); 
\draw(5,0) -- node[above, at end]{$3$}(5,1);
\draw(6,0) -- node[above, at end]{$3$} (6,1) ; 
\node[draw,fill=white,inner xsep=40pt] at (3.5,0){};}\\
\hline
    \end{tabular}
    \caption{Basis vectors for simple $\tT$-modules, part II}
    \label{fig:list2}
\end{figure}

We can also organize this information a bit differently, as a table of the weight multiplicities for different simples.   This table, and those later in this section, are all based on GAP4 and Python code available online \cite{GTcode}.
 To save space, we
only indicate the word for each weight space.  So, in place of
$a_{3,1}<a_{3,2}<a_{3,3}\leq a_{2,1}<a_{2,2}\leq a_{1,1}$, we write
$333221$, in place of $    a_{3,1}<a_{3,2}<a_{3,3}\leq a_{2,1}\leq
a_{1,1}<a_{2,2}$, we write $333212$, etc.  We omit entries which are 0
for clarity of reading.
We denote simples in this chart by their symbol from \cite{futornyGelfandTsetlinModules2018}. You can easily read off the good word by looking at the last non-zero entry in each column; since words are listed in decreasing lexicographic order, this will give the desired good word.

For each row, there will be a unique column that corresponds to the canonical module for this weight space.  This, of course, will have non-zero multiplicity.  In this case, we have also included a box around the multiplicity of that weight space.  
\begin{figure}
  \thisfloatpagestyle{empty}
 \centerline{ \scalebox{.75}{   \begin{tabular}{|c|c|c|c|c|c|c|c|c|c|c|c|c|c|c|c|c|c|c|c|c|}\hline
FGR& $L_{32}$&$L_{21}$&$L_{24}$&$L_{26}$&
$L_{16}$&$L_8$&$L_{19}$&$L_{10}$&$L_{17}$&$L_{13}$&$L_{6}$&$L_{3}$&$L_{4}$&
$L_{28}$&$L_{11}$&$L_{14}$&$L_{5}$&$L_{2}$&$L_{7}$&$L_{1}$\\\hline
333221&\boxed{2}&&&&&&&&&&&&&&&&&&&\\\hline
333212&\boxed{1}&&&&&&&&&&&&&1&&&&&&\\\hline
333122&&&&&&&&&&&&&&\boxed{2}&&&&&&\\\hline
332321&\boxed{1}&&&1&&&&&&&&&&&&&&&&\\\hline
332312&&\boxed{1}&&&&&&&&&&&&&&&&&&\\\hline
332231&&&&\boxed{2}&&&&&&&&&&&&&&&&\\\hline
332213&&&&\boxed{2}&&&&&&&&&&&&&&&&\\\hline
332132&&\boxed{1}&&&&&&&&&&&&&&&&&&\\\hline
332123&&&&\boxed{1}&&&&&&&&&&&&1&&&&\\\hline
331322&&&&&&&&&&&&&&\boxed{2}&&&&&&\\\hline
331232&&&&&&&&&&&&&&\boxed{1}&&1&&&&\\\hline
331223&&&&&&&&&&&&&&&&\boxed{2}&&&&\\\hline
323321&&&\boxed{1}&&&&&&&&&&&&&&&&&\\\hline
323312&&&&&\boxed{1}&&&&&&&&&&&&&&&\\\hline
323231&&&&\boxed{1}&&&&&1&&&&&&&&&&&\\\hline
323213&&&&\boxed{1}&&&&&1&&&&&&&&&&&\\\hline
323132&&&&&\boxed{1}&&&&&&&&&&&&&&&\\\hline
323123&&&&&&\boxed{1}&&&&&&&&&&&&&&\\\hline
322331&&&&&&&&&\boxed{2}&&&&&&&&&&&\\\hline
322313&&&&&&&&&\boxed{2}&&&&&&&&&&&\\\hline
322133&&&&&&&&&\boxed{2}&&&&&&&&&&&\\\hline
321332&&&&&\boxed{1}&&&&&&&&&&&&&&&\\\hline
321323&&&&&&\boxed{1}&&&&&&&&&&&&&&\\\hline
321233&&&&&&&&&\boxed{1}&&&&&&&&&&1&\\\hline
313322&&&&&&&&&&&&&&\boxed{2}&&&&&&\\\hline
313232&&&&&&&&&&&&&&\boxed{1}&&1&&&&\\\hline
313223&&&&&&&&&&&&&&&&\boxed{2}&&&&\\\hline
312332&&&&&&&&&&&&&&&\boxed{1}&&&&&\\\hline
312323&&&&&&&&&&&&&&&&\boxed{1}&&&1&\\\hline
312233&&&&&&&&&&&&&&&&&&&\boxed{2}&\\\hline
233321&&&&&&&\boxed{1}&&&&&&&&&&&&&\\\hline
233312&&&&&&&&&&\boxed{1}&&&&&&&&&&\\\hline
233231&&&&&&&&\boxed{1}&&&&&&&&&&&&\\\hline
233213&&&&&&&&\boxed{1}&&&&&&&&&&&&\\\hline
233132&&&&&&&&&&\boxed{1}&&&&&&&&&&\\\hline
233123&&&&&&&&&&&\boxed{1}&&&&&&&&&\\\hline
232331&&&&&&&&&\boxed{1}&&&&1&&&&&&&\\\hline
232313&&&&&&&&&\boxed{1}&&&&1&&&&&&&\\\hline
232133&&&&&&&&&\boxed{1}&&&&1&&&&&&&\\\hline
231332&&&&&&&&&&\boxed{1}&&&&&&&&&&\\\hline
231323&&&&&&&&&&&\boxed{1}&&&&&&&&&\\\hline
231233&&&&&&&&&&&&\boxed{1}&&&&&&&&\\\hline
223331&&&&&&&&&&&&&\boxed{2}&&&&&&&\\\hline
223313&&&&&&&&&&&&&\boxed{2}&&&&&&&\\\hline
223133&&&&&&&&&&&&&\boxed{2}&&&&&&&\\\hline
221333&&&&&&&&&&&&&\boxed{2}&&&&&&&\\\hline
213332&&&&&&&&&&\boxed{1}&&&&&&&&&&\\\hline
213323&&&&&&&&&&&\boxed{1}&&&&&&&&&\\\hline
213233&&&&&&&&&&&&\boxed{1}&&&&&&&&\\\hline
212333&&&&&&&&&&&&&\boxed{1}&&&&&&&1\\\hline
133322&&&&&&&&&&&&&&\boxed{2}&&&&&&\\\hline
133232&&&&&&&&&&&&&&\boxed{1}&&1&&&&\\\hline
133223&&&&&&&&&&&&&&&&\boxed{2}&&&&\\\hline
132332&&&&&&&&&&&&&&&\boxed{1}&&&&&\\\hline
132323&&&&&&&&&&&&&&&&\boxed{1}&&&1&\\\hline
132233&&&&&&&&&&&&&&&&&&&\boxed{2}&\\\hline
123332&&&&&&&&&&&&&&&&&\boxed{1}&&&\\\hline
123323&&&&&&&&&&&&&&&&&&\boxed{1}&&\\\hline
123233&&&&&&&&&&&&&&&&&&&\boxed{1}&1\\\hline
122333&&&&&&&&&&&&&&&&&&&&\boxed{2}\\\hline
    \end{tabular}
}}
\caption{The table of dimensions of GT weight spaces in the integral regular case for $\mathfrak{gl}_3$}
\label{fig:sl3wt}
\end{figure}

\begin{figure}\thisfloatpagestyle{empty} 
  \centering
  \begin{tabular}{|c|c|c|c|c|c|c|c|c|c|c|c|}
  \hline FGR&$L_{16}$ &$L_7$&$L_9$&$L_{13}$&$L_5$&$L_2$&$L_6$&
$L_4$&$L_3$&$L_4$&$L_1$\\\hline
33221&\boxed{2}&&&&&&&&&&\\\hline
33212&\boxed{1}&&&&&&&1&&&\\\hline
33122&&&&&&&&\boxed{2}&&&\\\hline
32321&\boxed{1}&&&1&&&&&&&\\\hline
32312&&\boxed{1}&&&&&&&&&\\\hline
32231&&&&\boxed{2}&&&&&&&\\\hline
32213&&&&\boxed{2}&&&&&&&\\\hline
32132&&\boxed{1}&&&&&&&&&\\\hline
32123&&&&\boxed{1}&&&&&&1&\\\hline
31322&&&&&&&&\boxed{2}&&&\\\hline
31232&&&&&&&&\boxed{1}&&1&\\\hline
31223&&&&&&&&&&\boxed{2}&\\\hline
23321&&&\boxed{1}&&&&&&&&\\\hline
23312&&&&&\boxed{1}&&&&&&\\\hline
23231&&&&\boxed{1}&&&1&&&&\\\hline
23213&&&&\boxed{1}&&&1&&&&\\\hline
23132&&&&&\boxed{1}&&&&&&\\\hline
23123&&&&&&\boxed{1}&&&&&\\\hline
22331&&&&&&&\boxed{2}&&&&\\\hline
22313&&&&&&&\boxed{2}&&&&\\\hline
22133&&&&&&&\boxed{2}&&&&\\\hline
21332&&&&&\boxed{1}&&&&&&\\\hline
21323&&&&&&\boxed{1}&&&&&\\\hline
21233&&&&&&&\boxed{1}&&&&1\\\hline
13322&&&&&&&&\boxed{2}&&&\\\hline
13232&&&&&&&&\boxed{1}&&1&\\\hline
13223&&&&&&&&&&\boxed{2}&\\\hline
12332&&&&&&&&&\boxed{1}&&\\\hline
12323&&&&&&&&&&\boxed{1}&1\\\hline
12233&&&&&&&&&&&\boxed{2}\\\hline
  \end{tabular}
                                  
    \caption{The characters of GT
      modules when $\iota_1=1,\iota_3=2$ (case (C5) in the notation of \cite{futornyGelfandTsetlinModules2018})}
    \label{fig:33221}
  \end{figure}

   \begin{figure}\thisfloatpagestyle{empty} 
        \centering (1)
        \begin{tabular}{|c|c|c|c|c|c|}
\hline
FGR&$L_6$&$L_2$&$L_3$&$L_4$&$L_1$\\\hline
3221&\boxed{2}&&&&\\\hline
3212&\boxed{1}&&&1&\\\hline
3122&&&&\boxed{2}&\\\hline
2321&\boxed{1}&&1&&\\\hline
2312&&\boxed{1}&&&\\\hline
2231&&&\boxed{2}&&\\\hline
2213&&&\boxed{2}&&\\\hline
2132&&\boxed{1}&&&\\\hline
2123&&&\boxed{1}&&1\\\hline
1322&&&&\boxed{2}&\\\hline
1232&&&&\boxed{1}&1\\\hline
1223&&&&&\boxed{2}\\\hline
        \end{tabular}     
        \hspace{1cm}(2)
        \begin{tabular}{|c|c|c|c|c|}\hline
FGR&$L_5$&$L_2$&$L_3$&$L_1$\\\hline
3322&\boxed{2}&&&\\\hline
3232&\boxed{1}&&1&\\\hline
3223&&&\boxed{2}&\\\hline
2332&&\boxed{1}&&\\\hline
2323&&&\boxed{1}&1\\\hline
2233&&&&\boxed{2}\\\hline
        \end{tabular}
        \vspace{1cm}

      (3)  \begin{tabular}{|c|c|c|c|c|c|c|c|}
          \hline
FGR&$L_{10}$&$L_5$&$L_8$&$L_4$&$L_2$&$L_3$&$L_1$\\\hline
33322&\boxed{2}&&&&&&\\\hline
33232&\boxed{1}&&1&&&&\\\hline
33223&&&\boxed{2}&&&&\\\hline
32332&&\boxed{1}&&&&&\\\hline
32323&&&\boxed{1}&&&1&\\\hline
32233&&&&&&\boxed{2}&\\\hline
23332&&&&\boxed{1}&&&\\\hline
23323&&&&&\boxed{1}&&\\\hline
23233&&&&&&\boxed{1}&1\\\hline
22333&&&&&&&\boxed{2}\\\hline
        \end{tabular}
         \hspace{1cm}
         \begin{minipage}{1.5in}\centering
         (4)
                \begin{tabular}{|c|c|c|c|}\hline
FGR&$L_2$&$L_1$\\\hline
221&\boxed{2}&\\\hline
212&\boxed{1}&1\\\hline
122&&\boxed{2}\\\hline
        \end{tabular}
        \vspace{5mm}
        
        (5)               
        \begin{tabular}{|c|c|c|c|}\hline
FGR&$L_1$\\\hline
22&\boxed{2}\\\hline
        \end{tabular}
         \end{minipage}
        \caption{Similarly, here are the tables for (1) $\iota_1=\iota_3=1$ (C4)
          (2)   $\iota_1=0,\iota_3=2 $ (C6), and (3)
          $\iota_1=0, \iota_3=3$ (C9), (4)  
          $\iota_1=1,\iota_3=0$ (C2), (5) $\iota_1=\iota_3=0$ (C1); the case $\iota_1=0,\iota_3=1$ (C3) is identical to (4).}
        \label{fig:smaller-sl3}
      \end{figure}
Note some interesting observations about this table:
\begin{itemize}
\item As mentioned above, the
  last non-zero entry in a column is in the row corresponding to its
  good word.
\item Also, as mentioned above, every row contains either 1 once, 1 twice, or 2 once.  This corresponds to the trichotomy for 1-singular weights given in \cite[Cor. 3.6]{WebGT}, depending on whether the corresponding map is small, strictly semi-small, or a $\mathbb{P}^1$-bundle.
\item We have a 2 in a row if and only if the word has two consecutive 2's.  This is because this is the only situation where the corresponding simple in the KLR algebra has a negative degree element.
\item The Gelfand-Kirillov dimension of each irrep can be read from these diagrams: by Theorem \ref{thm:GK}, for each irrep, it is the maximal number of 2's and 1's not between any pair of 3's.  Thus, it is 0 for $L_8$, the finite-dimensional module, 3 for $L_1$, $L_4$, $L_5$, $L_{19}$, $L_{28}$, $L_{32}$, and $L_{13}$, and 2 for all other modules.  
\item In each of the rows of this table, the boxed entry is the furthest left.  That is, the canonical module for the weight is the simple module with the greatest good word in lex order amongst those with non-zero multiplicity.  We have yet to observe a counter-example to this pattern, but the cases we have done thus far provide scant evidence.  In many cases, patterns like this break down in higher ranks. 
\end{itemize}
A simple Gelfand-Tsetlin module is always a $\mathfrak{h}$-weight module, that is, the Cartan $\mathfrak{h}$ acts locally finitely and semi-simply.  However, it might have finite or infinite weight multiplicities.  For $\gl_3$, there is only one integral GT module in each block with an infinite $\mathfrak{h}$-weight multiplicity as shown in \cite{futornyGelfandTsetlinModules2018}.  One can see this using the approach of comparison with the KLRW algebra as follows---if we fix the $\mathfrak{h}$-weight and the central character, we have fixed the longitude of all strands with label $1$ or $3$.  Thus, we can only have an infinite-dimensional $\mathfrak{h}$-weight multiplicity if we have infinitely many GT weight spaces with $a_{12}+a_{22}$ fixed, so $a_{12}\ll \cdots \ll a_{22}$ with all other $a_{**}$'s in between, that is, the corresponding idempotent begins and ends with $2$.  We can thus read off from Figure \ref{fig:sl3wt} that $L_{13}$ in the notation of \cite{futornyGelfandTsetlinModules2018} is the only such simple in the regular block. Of course, one can see from the non-integral case below that there are also many non-integral examples.
      
\subsubsection{Non-integral modules} We also include some similar calculations for non-integral
Gelfand-Tsetlin modules.  We fix a central character $\chi$, which may or may not be integral.  
Of course, we can break the category of Gelfand-Tsetlin modules further into blocks where we fix the coset of $a_{1,1},a_{2,1},a_{2,2}$ modulo $\Z$. For a fixed choice of these cosets, let $\mathscr{S}$ be the set of corresponding GT weights.

In this case, the partial order $\prec$ on
the set $\Omega$ defined before is not a total order, and it is useful
to recall that:
\begin{proposition}\label{prop:comparison}
  If $\gamma$ and $\gamma'$ are related by a permutation $\sigma
  \colon \Omega\to \Omega$  such that:
  \begin{enumerate}
  \item For all $(i,j)\in \Omega$, we have $\sigma(i,j)=(i,m)$ for some $m$ and $\sigma(n,k)=(n,k)$.
  \item For all $i,j$, we have $a_{i,j}-a'_{\sigma(i,j)}\in \Z$.  
  \item We have $(i,j)\prec (k,\ell)$ if and only if $\sigma(i,j)\prec'(k,\ell)$.  
  \end{enumerate}
  then the functors $\Wei_{\gamma}$ and $\Wei_{\gamma'}$ are isomorphic.
\end{proposition}
If $\gamma$ is integral, then $\gamma'$ must be as well, and this
result simply requires that the orders $\prec$ and $\prec'$ coincide.  
On the other hand, if $\gamma$ is not integral, then we only need to pay
attention to the order between elements that are in the same coset of $\Z$.

If 
$a_{2,1}-a_{2,2}\notin \Z$ then the corresponding block is non-singular in a non-singular block, and  their weight multiplicities
described in \cite[Cor. 33]{FOD}.  
In particular, all GT multiplicities are 1 or 0.  Thus, we can reduce to the case where $a_{2,1}-a_{2,2}\in \Z$.  Tensoring with a complex
power of the trace representation uniformly shifts $a_{*,*}$, so we can assume that
$a_{2,1},a_{2,2}\in \Z$.

Now, fix $\iota_1\in \{0,1\}$ and $\iota_3\in \{0,1,2,3\}$.  We'll consider a block where:
\begin{itemize}
\item [(*)] For all weights in $\mathscr{S}$, $a_{2,1},a_{2,2}\in \Z$ and $\iota_i=\#\{j\mid a_{i,j}\in \Z\}$.  
\end{itemize}
By Proposition \ref{prop:comparison}, the isomorphism type of the
corresponding weight functor only depends on the order of elements of
$\Omega$ such that $a_{i,j}\in \Z$.

By Proposition \ref{prop:integrality}, we have:
\begin{proposition}
The category of GT modules with
support in $\mathscr{S}$ is equivalent to the category of modules over the OGZ
algebra for the dimension vector $(\iota_1,\iota_2,\iota_3)$, tensored
with polynomials in $4-\iota_1-\iota_3$ variables, matching GT multiplicities.
\end{proposition}
Thus, we can compute the weight multiplicities of simple modules with
support in $\mathscr{S}$ by computing these for integral GT modules
over the corresponding OGZ algebra.  This is presented in all the
relevant cases for $\mathfrak{sl}_3$ in Figures \ref{fig:33221} and
\ref{fig:smaller-sl3}.  In each case, we write the word that encodes  the induced order on the elements of $\Omega$ such that $a_{ij}\in \Z$; note that by assumption, $2$ will appear twice,  $1$ will appear $\iota_1$ times, and similarly $3$ appears $\iota_3$ times.   As mentioned
in the caption, the case $\iota_1=0, \iota_3=1$ has the same
multiplicities as $\iota_1=1,\iota_3=0$, so we omit it.  

All of these cases have been worked out in \cite[\S 7.5]{futornyGelfandTsetlinModules2018}.  We will include their labeling of different cases; for example, for them, the integral regular case is (C10).  We have only included the cases that have regular central character.  Their cases (C7), (C8), (C11), (C12), (C13), and (C14)  are translations to singular blocks of (C5), (C6), (C10), (C9), (C10), and (C19).  

The character tables for these can be obtained by just deleting from the appropriate table all the rows corresponding to words that cannot be realized with this singular character of $Z_3$.  That is to say, all the words that have a 1 or 2 between two 3's that have the same longitude under our choice of $\chi$.  Some of the columns of the table will now be empty; these correspond to the simple modules which are killed by translation to the wall, so we remove the empty columns.  

\begin{figure}
 \centerline{ \scalebox{.85}{    \begin{tabular}{|c|c|c|c|c|c|c|c|c|c|c|c|c|c|c|c|c|c|c|c|c|}\hline
FGR& $L_{20}$&&$L_{14}$&& $L_{10}$&&$L_{11}$&&$L_{8}$&$L_{6}$&&$L_{2}$&$L_{5}$&$L_{16}$&$L_{9}$&&$L_{3}$&&$L_{4}$&$L_{1}$\\\hline
333221&\boxed{2}&&&&&&&&&&&&&&&&&&&\\\hline
333212&\boxed{1}&&&&&&&&&&&&&1&&&&&&\\\hline
333122&&&&&&&&&&&&&&\boxed{2}&&&&&&\\\hline
323321&&&\boxed{1}&&&&&&&&&&&&&&&&&\\\hline
323312&&&&&\boxed{1}&&&&&&&&&&&&&&&\\\hline
322331&&&&&&&&&\boxed{2}&&&&&&&&&&&\\\hline
322133&&&&&&&&&\boxed{2}&&&&&&&&&&&\\\hline
321332&&&&&\boxed{1}&&&&&&&&&&&&&&&\\\hline
321233&&&&&&&&&\boxed{1}&&&&&&&&&&1&\\\hline
313322&&&&&&&&&&&&&&\boxed{2}&&&&&&\\\hline
312332&&&&&&&&&&&&&&&\boxed{1}&&&&&\\\hline
312233&&&&&&&&&&&&&&&&&&&\boxed{2}&\\\hline
233321&&&&&&&\boxed{1}&&&&&&&&&&&&&\\\hline
233312&&&&&&&&&&\boxed{1}&&&&&&&&&&\\\hline
232331&&&&&&&&&\boxed{1}&&&&1&&&&&&&\\\hline
232133&&&&&&&&&\boxed{1}&&&&1&&&&&&&\\\hline
231332&&&&&&&&&&\boxed{1}&&&&&&&&&&\\\hline
231233&&&&&&&&&&&&\boxed{1}&&&&&&&&\\\hline
223331&&&&&&&&&&&&&\boxed{2}&&&&&&&\\\hline
223133&&&&&&&&&&&&&\boxed{2}&&&&&&&\\\hline
221333&&&&&&&&&&&&&\boxed{2}&&&&&&&\\\hline
213332&&&&&&&&&&\boxed{1}&&&&&&&&&&\\\hline
213233&&&&&&&&&&&&\boxed{1}&&&&&&&&\\\hline
212333&&&&&&&&&&&&&\boxed{1}&&&&&&&1\\\hline
133322&&&&&&&&&&&&&&\boxed{2}&&&&&&\\\hline
132332&&&&&&&&&&&&&&&\boxed{1}&&&&&\\\hline
132233&&&&&&&&&&&&&&&&&&&\boxed{2}&\\\hline
123332&&&&&&&&&&&&&&&&&\boxed{1}&&&\\\hline
123233&&&&&&&&&&&&&&&&&&&\boxed{1}&1\\\hline
122333&&&&&&&&&&&&&&&&&&&&\boxed{2}\\\hline
    \end{tabular}
}}
\caption{The table of dimensions of GT weight spaces in an integral singular block where $a_{31}<a_{32}=a_{33}$ for $\mathfrak{gl}_3$}
\label{fig:sl3wt-sing}
\end{figure}
As an illustration, we include the case of (C11) where $a_{31}<a_{32}=a_{33}$.  We have deleted the words incompatible with this order and have left blank columns corresponding to the simples that disappear in the process.  Unfortunately, the conventions chosen in \cite{futornyGelfandTsetlinModules2018} are not especially compatible with translation to the wall, but in our tables, the $k$th column from the left in Figure \ref{fig:sl3wt-sing} is the translation to the wall of the $k$th column in Figure \ref{fig:sl3wt}.  We leave the other singular cases to the reader.

\subsection{The case of $\gl_4$-modules}

\ifslfour
Finally, we include the full list of multiplicities for the integral case
of $\mathfrak{gl}_4$ at the end of the paper after the bibliography.  In the interest of saving space, we don't show every single word: if a 31 or 42 appears in a word, the word where we switch these numbers has the same multiplicities in every GT module and is lower in lex-order; thus we only display words that have no such pairs.  For example, we leave out the words 4443423231, 4443423213, and 4443243231 since 4443243213 has the same multiplicities.  There are 259 simples in this case, so we break
these into groups of 30; as the reader can see, these are on
impractically large pages.  If the reader would prefer to not have
these tables in the file, download the source from the {\tt arXiv} and
comment out the line {\tt \textbackslash slfourtrue}.
\fi
\begin{figure}[tb]
  \centering
  \includegraphics[width=6in]{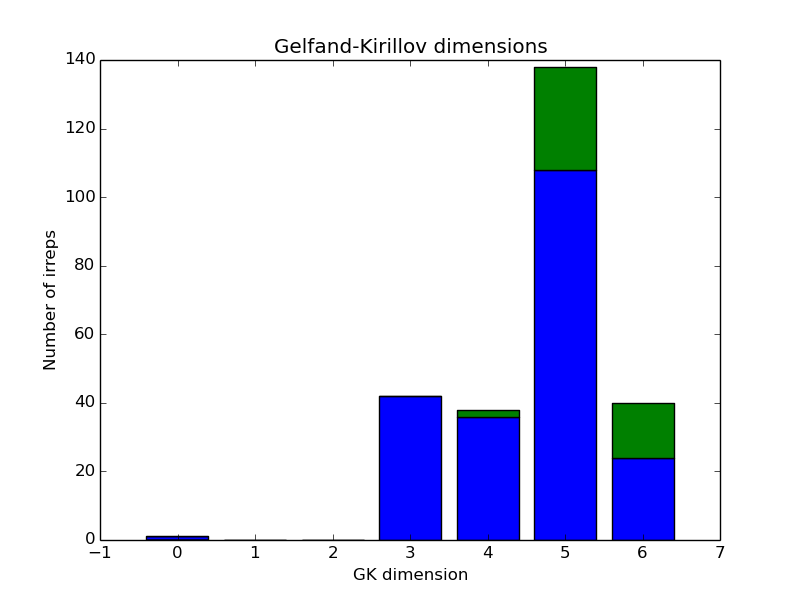}
  \caption{The number of integral $\gl_4$ GT modules of a given Gelfand-Kirillov dimension in a regular block.  Blue represents simples without an infinite-dimensional $\mathfrak{h}$-weight space, and green those which do. }\label{fig:hist}
\end{figure}
Many of the features we pointed out about the $\gl_3$ table hold here with suitable modifications:
\begin{itemize}
\item As before, the
  last non-zero entry in a column is in the row corresponding to its
  good word.
\item There are now many different possibilities for the entries in a row; we also have to account for the fact that a given word appears on several pages, so we have to consider all those rows together.  By \cite[Cor. 2.19]{WebGT}, these must have sum which is no more than $12=3!\cdot 2!$.  The different possibilities of the multiplicities with which a given maximal ideal can appear in the different integral simple Gelfand-Tsetlin modules, with an example of a word realizing them, are:\smallskip

\centerline{\begin{tabular}{ll|ll|ll}
 $ \{1\}$&$ 4432143243$&$\{4, 8\}$&$ 4443343221$&$ \{2, 2, 2, 4\}$&$ 4434343221$\\
 $ \{2\}$&$ 4432443321$&$\{6, 6\}$&$ 4444333212$&$ \{2, 2, 4, 4\}$&$ 4443343212$\\
 $\{4\}$&$ 4443322143$ &$ \{1, 1, 1\}$&$ 4432143432$&$ \{1, 1, 1, 1, 1\}$&$ 4434432132$\\
 $\{6\}$&$ 2444433321$ &$ \{1, 2, 2\}$&$ 4434432321$&$ \{1, 1, 1, 1, 2\}$&$ 4443213432$\\
 $ \{12\}$&$ 4444333221$&$\{2, 2, 2\}$&$ 4443243321$&$ \{1, 1, 1, 2, 2\}$&$ 4443234321$\\
 $\{1, 1\}$&$ 4434432123$&$ \{2, 2, 4\}$&$ 4443324321$&$ \{1, 1, 1, 1, 1, 1\}$&$ 4443432123$\\ 
 $ \{1, 2\}$&$ 4432443213$&$ \{2, 4, 4\}$&$ 4444332132$&$\{1, 1, 1, 1, 1, 2\}$&$ 4434324321$\\
 $ \{2, 2\}$&$ 4443214332$ &$\{1, 1, 1, 1\}$&$ 4434321432$&$ \{1, 1, 1, 1, 2, 2\}$&$ 4443432132$\\
 $\{2, 4\}$&$ 2444334321$&$ \{1, 1, 1, 2\}$&$ 4443243213$&$ \{1, 1, 1, 1, 1, 1, 2\}$&$ 4444321323$\\ 
 $\{2, 8\}$&$ 4444332321$&$ \{1, 1, 2, 5\}$&$ 4443432321$&$ \{1, 1, 1, 1, 1, 1, 2, 2\}$&$ 4434343212$\\ 
 $\{4, 4\}$&$ 4434433221$&$ \{2, 2, 2, 2\}$&$ 4434433212$
\end{tabular}}\smallskip
This shows that we have at least 32 different isomorphism types for the endomorphism algebras of weight functors, a dramatic contrast with $\gl_3$ where there were 3.  
\item In particular, rows where one of the entries is 12 are special and can only have one non-zero entry; this can only happen for words that are essential, by Theorem \ref{thm:essential}.  Note that we only see one 12 per column, because we don't display all words; our conventions about leaving out words containing 42 or 31 means that we only show the lex-shortest word in the essential support of a given module, and don't display the others that are essentially the same.  
\item More generally, when there is a group of $n_2$ 2's and one of $n_3$ 3's, by Proposition \ref{prop:smaller-nilHecke}, a smaller nilHecke algebra will act and the multiplicity must be a multiple of $n_2!n_3!$,  the dimension of the unique irrep of this nilHecke algebra.  For example, 3323214444 has a group of two 3's, and so the multiplicity must be divisible by $2$; in fact, it has multiplicities of 8 and 2 in two different simple modules.  To give two more examples: 4444333212 has a group of three 3's, so the multiplicities are divisible by $3!=6$, and in fact, this appears in two modules with multiplicity 6; the word    4443343221 has $n_2=2$ and $n_3=2$, so the multiplicities must be divisible by $4$, and in fact, they are 4 and 8.  
\item As in the case of $\gl_3$, we include the Gelfand-Kirillov dimension
(determined using Theorem \ref{thm:GK}) in the row labeled ``GK.'' 
In Figure \ref{fig:hist}, we give some statistics on the number of simples with a given Gelfand-Kirillov dimension, with color distinguishing whether all $\mathfrak{h}$-weight spaces are finite dimensional, or at least one is infinite dimensional.  Note that there are no modules with GK dimension 1 or 2; such a module would be annihilated by a primitive ideal whose associated variety would be a 2 or 4-dimensional nilpotent orbit in $\gl_4$, and the minimal non-trivial orbit has dimension 6.  With a bit more combinatorial effort, we could determine the precise primitive ideals annihilating these modules, but we will leave that to another time.
\item Our observation earlier that the canonical module seems to correspond to the lex-greatest good word amongst all modules where a weight appears still holds in $\gl_4$, though this is not as visually apparent since rows are split across multiple pages.  Again, we don't regard this to be extremely convincing, since the behavior of the corresponding algebras will be much more complicated in higher rank, but it remains an intriguing possibility.  
\item  For each simple, the row labeled ``IWS'' contains a ``T'' if an
infinite dimensional $\mathfrak{h}$-weight space exists and ``F'' if
it doesn't.
\end{itemize}

 We determine whether modules over $\gl_4$
have an infinite-dimensional $\mathfrak{h}$-weight space using the
Theorem below:   

\begin{theorem}\label{thm:IWS}
  If $\mcU=\mcU(\gl_4)$, then the integral GT module $\mathbb{GT}(M)$ has an infinite $\mathfrak{h}$-weight multiplicity for some weight if and only if there is a word $\Bi$ such that  $e(\Bi)M\neq 0$ which both begins and ends with one of the words $\{2,3,23,32\}$.
\end{theorem}

\begin{proof}
  The module $\mathbb{GT}(M)$ has infinite weight multiplicity for $\mu$, if and only if for some word $\Bi$ such that $e(\Bi)M\neq 0$, there are infinitely many $\gamma$ with $e(\gamma)=e(\Bi)$ of weight $\mu$ for $\mathfrak{h}$.  If a word $\Bi$ ends with $2$, with $3$, or with permutations of the word $23$, then we can subtract any fixed integer $a$ from the longitude for the start of the word and add $a$ to the longitude for the end without changing the $\mathfrak{h}$-weight or the idempotent $\Bi$, so we have found an infinite multiplicity weight space.

  Now, assume that we have $\Bi$ with infinitely many $\gamma$ with $e(\gamma)=e(\Bi)$ of weight $\mu$ for $\mathfrak{h}$.  Having fixed the weight space, the sums
  \[a_{21}+a_{22}=\mu_2\qquad a_{31}+a_{32}+a_{33}=\mu_3\]
  will be fixed.  One of $a_{21}$ or $a_{31}$ must take on infinitely many values below the averages $\mu_2/2$ and $\mu_3/3$, while $a_{11}$ and $a_{4*}$ are fixed.  So $\Bi$ must begin with a 2 or a 3.  Similarly, it must end with a 2 or 3.  If it begins with a 2 and ends with a 3, then $a_{21}<a_{31}$, and $a_{22}<a_{33}$, and we can assume that $a_{21}=-k$ for $k\gg 0$.  Thus, $a_{22}=\mu_2+k<a_{33}$.  Since $a_{31}<a_{32}$, we have the inequality:
  \[a_{31}<\frac{a_{31}+a_{32}}{2}=\frac{\mu_3-a_{33}}2<\frac{\mu_3-\mu_2-k}2.\]
  So once $k$ is sufficiently large, this shows that $a_{21}<a_{31}$ are smaller than all other $a_{**}$'s and $a_{22}<a_{33}$ are larger than all others.  Thus, $\Bi$ must begin with $23$ and end with $23$.  On the other hand, if the word begins with $3$ and ends with $2$, then the same argument shows that it begins and ends with $32$.
\end{proof}
In general, it is clear that if $\Bi$ begins and ends with permutations of the same word which includes no appearances of $n$, there will be an infinite multiplicity $\mathfrak{h}$-weight space, but it is not clear if this is a necessary condition for $n>4$; the authors suspect it is not.
\begin{figure}
  \centering
  \begin{tabular}{{|c|c|c|c|c|c|c|c|c|c|c|}}\hline
 43221&\boxed{2}&&&&&&&&&\\\hline
43212&\boxed{1}&&&&&&1&&&\\\hline
32214&&\boxed{2}&&&&&&&&\\\hline
32124&&\boxed{1}&&&&&&1&&\\\hline
24321&\boxed{1}&&&&1&&&&&\\\hline
23214&&\boxed{1}&&&&1&&&&\\\hline
22143&&&&&\boxed{2}&&&&&\\\hline
22134&&&&&&\boxed{2}&&&&\\\hline
21432&&&\boxed{1}&&&&&&&\\\hline
21324&&&&\boxed{1}&&&&&&\\\hline
21243&&&&&\boxed{1}&&&&1&\\\hline
21234&&&&&&\boxed{1}&&&&1\\\hline
14322&&&&&&&\boxed{2}&&&\\\hline
13224&&&&&&&&\boxed{2}&&\\\hline
12432&&&&&&&\boxed{1}&&1&\\\hline
12324&&&&&&&&\boxed{1}&&1\\\hline
12243&&&&&&&&&\boxed{2}&\\\hline
12234&&&&&&&&&&\boxed{2}\\\hline
  \end{tabular}\hspace{1cm}
  \begin{tabular}{{|c|c|c|c|c|c|c|c|}}\hline
    44433&\boxed{2}&&&&&&\\\hline
44343&\boxed{1}&&1&&&&\\\hline
44334&&&\boxed{2}&&&&\\\hline
43443&&\boxed{1}&&&&&\\\hline
43434&&&\boxed{1}&&&1&\\\hline
43344&&&&&&\boxed{2}&\\\hline
34443&&&&\boxed{1}&&&\\\hline
34434&&&&&\boxed{1}&&\\\hline
34344&&&&&&\boxed{1}&1\\\hline
33444&&&&&&&\boxed{2}\\\hline
  \end{tabular}
  \caption{Tables of multiplicities for integral modules over the OGZ algebras $\mcU_{(1,2,1,1)}$ and $\mcU_{(0,0,2,3)}$.  }
  \label{fig:OGZ}
\end{figure}

Finally, let us illustrate how to understand non-integral cases with a $\gl_4$ example: Consider the block of Gelfand-Tsetlin modules where $\Bv^{(\bar{0})}=(1,2,1,1)$ and $\Bv^{(\overline{\nicefrac{1}{2}})}=(0,0,2,3)$.  In this case, every simple is the tensor product of two simple Gelfand-Tsetlin modules over the corresponding OGZ algebras for these vectors.  The multiplicities of these blocks are shown in Figure \ref{fig:OGZ}.  Every simple module over $\gl_4$ in this block corresponds to a pair of columns in the two tables; thus, there are $70=10\cdot 7$ of them.  Each GT weight corresponds to a pair of rows, corresponding to the order of the numbers $a_{*,*}$ that are congruent mod $\Z$.  For example in the module $L$ corresponding to the first column of the first table and the last column of the second, either of the weights corresponding to:
\[\tikz{\matrix[row sep=0mm,column sep=0mm,ampersand replacement=\&]{
      \node {$0$}; \& \& \node {$2\nicefrac{1}{2}$}; \& \& \node {$4\nicefrac{1}{2}$}; \& \& \node {$6\nicefrac{1}{2}$};\\
\& \node {$\nicefrac{1}{2}$}; \& \& \node {$1$}; \& \& \node {$1\nicefrac{1}{2}$}; \& \\
\& \& \node {$2$};  \& \& \node {$4$}; \&\& \\
\& \& \&\node {$7$}; \& \&\& \\
};}\qquad \tikz{\matrix[row sep=0mm,column sep=0mm,ampersand replacement=\&]{
      \node {$-2\nicefrac{1}{2}$}; \& \& \node {$-\nicefrac{1}{2}$}; \& \& \node {$5$}; \& \& \node {$8\nicefrac{1}{2}$};\\
\& \node {$-4\nicefrac{1}{2}$}; \& \& \node {$-3\nicefrac{1}{2}$}; \& \& \node {$7$}; \& \\
\& \& \node {$9$};  \& \& \node {$12$}; \&\& \\
\& \& \&\node {$13$}; \& \&\& \\
};}\]
give the words $43221$ and $33444$ since, in the first case, we have $0<1<2<4<7$ and $\nicefrac{1}{2}<1 \nicefrac{1}{2}<2 \nicefrac{1}{2}<4 \nicefrac{1}{2}<6 \nicefrac{1}{2}$.  Thus, the corresponding GT weight space is $4=2\cdot 2$ dimensional, and since both 2's are boxed, $L$ is the canonical module for these weights.

On the other hand, the first weight below
\begin{equation}
\tikz[baseline]{\matrix[row sep=0mm,column sep=0mm,ampersand replacement=\&]{
      \node {$0$}; \& \& \node {$2\nicefrac{1}{2}$}; \& \& \node {$4\nicefrac{1}{2}$}; \& \& \node {$6\nicefrac{1}{2}$};\\
\& \node {$\nicefrac{1}{2}$}; \& \& \node {$3$}; \& \& \node {$1\nicefrac{1}{2}$}; \& \\
\& \& \node {$2$};  \& \& \node {$4$}; \&\& \\
\& \& \&\node {$7$}; \& \&\& \\
};}\qquad \tikz[baseline]{\matrix[row sep=0mm,column sep=0mm,ampersand replacement=\&]{
      \node {$-2\nicefrac{1}{2}$}; \& \& \node {$-\nicefrac{1}{2}$}; \& \& \node {$5$}; \& \& \node {$8\nicefrac{1}{2}$};\\
\& \node {$-4\nicefrac{1}{2}$}; \& \& \node {$-3\nicefrac{1}{2}$}; \& \& \node {$7$}; \& \\
\& \& \node {$9$};  \& \& \node {$12$}; \&\& \\
\& \& \&\node {$-2$}; \& \&\& \\
};}\label{eq:sl4-example}
\end{equation}
has corresponding words 42321 and 33444; the word 42321 doesn't appear in our table, since we can swap the initial 4 and 2 without changing the weight functor.  So, we should consider the row for 24321, and, we now have a $2=1\cdot 2$ dimensional weight space in $L$.  Note that we get the same multiplicity if we replace $L$ by $L'$, the simple for the 7th column in the first table and the last in the second.  We can see that again $L$ is the canonical module of the first weight from \eqref{eq:sl4-example}.  On the other hand, the second weight in \eqref{eq:sl4-example} has corresponding words 14322 and 33444, so $L$ has multiplicity 0 in this case, and $L'$ has multiplicity $4=2\cdot 2$ and is the canonical module of this weight.

\section*{Declarations}

{\bf Ethical Approval}: not applicable

{\bf Funding}: T. S. was supported by NSERC and the University of Waterloo through an Undergraduate Student Research Award.  B. W. is supported by an NSERC Discovery Grant.  This research was supported in part by Perimeter Institute for Theoretical Physics. Research at Perimeter Institute is supported by the Government of Canada through the Department of Innovation, Science and Economic Development Canada and by the Province of Ontario through the Ministry of Research, Innovation and Science.

{\bf Availability of data and materials}: All code used in the production of data in this file is available at: 
\url{https://github.com/bwebste/gelfand-tsetlin-public}

\bibliography{./gen3}

\newcommand{\etalchar}[1]{$^{#1}$}
\providecommand{\bysame}{\leavevmode\hbox to3em{\hrulefill}\thinspace}
\providecommand{\MR}{\relax\ifhmode\unskip\space\fi MR }
\providecommand{\MRhref}[2]{%
  \href{http://www.ams.org/mathscinet-getitem?mr=#1}{#2}
}
\providecommand{\href}[2]{#2}
\begin{thebibliography}{FGRZ20b}

\bibitem[BKM14]{brundanHomologicalProperties2014}
Jonathan Brundan, Alexander Kleshchev, and Peter~J. McNamara, \emph{Homological
  properties of finite-type {{Khovanov}}--{{Lauda}}--{{Rouquier}} algebras},
  Duke Mathematical Journal \textbf{163} (2014), no.~7.

\bibitem[B{\'o}n22]{MR4596199}
Mikl{\'o}s B{\'o}na, \emph{Combinatorics of permutations}, 3 ed., Discrete
  mathematics and its applications (boca raton), CRC Press, Boca Raton, FL,
  2022. \MR{4596199}

\bibitem[BSW]{GTcode}
Jon Brundan, Turner Silverthorne, and Ben Webster, \emph{Gap4 and python code
  for computations}, available for download at
  \url{https://github.com/bwebste/gelfand-tsetlin-public}.

\bibitem[DFO94]{FOD}
{\relax Yu}.~A. Drozd, V.~M. Futorny, and S.~A. Ovsienko,
  \emph{Harish-{{Chandra}} subalgebras and {{Gelfand-Zetlin}} modules},
  Finite-dimensional algebras and related topics ({{Ottawa}}, {{ON}}, 1992),
  {{NATO}} adv. {{Sci}}. {{Inst}}. {{Ser}}. {{C}} math. {{Phys}}. {{Sci}}.,
  vol. 424, Kluwer Acad. Publ., Dordrecht, 1994, pp.~79--93. \MR{1308982}

\bibitem[DOF91]{drozdGelFandZetlin1991}
{\relax Yu}.~A. Drozd, S.~A. Ovsienko, and V.~M. Futorny, \emph{On
  {{Gel}}'fand-{{Zetlin}} modules}, Proceedings of the {{Winter School}} on
  {{Geometry}} and {{Physics}} ({{Srn{\'i}}}, 1990), Rendiconti del {{Circolo
  Matematico}} di {{Palermo}}, no.~26, 1991, pp.~143--147. \MR{1151899}

\bibitem[EMV20]{EMV}
Nick Early, Volodymyr Mazorchuk, and Elizaveta Vishnyakova, \emph{Canonical
  {{Gelfand}}--{{Zeitlin Modules}} over {{Orthogonal Gelfand}}--{{Zeitlin
  Algebras}}}, International Mathematics Research Notices \textbf{2020} (2020),
  no.~20, 6947--6966.

\bibitem[FGR18]{futornyGelfandTsetlinModules2018}
Vyacheslav Futorny, Dimitar Grantcharov, and Luis~Enrique Ramirez,
  \emph{Gelfand-{{Tsetlin}} modules of {{sl(3)}} in the principal block},
  Representations of {{Lie}} algebras, quantum groups and related topics,
  Contemp. {{Math}}., vol. 713, Amer. Math. Soc., Providence, RI, 2018,
  pp.~121--134. \MR{3845911}

\bibitem[FGRZ20a]{FGRZVerma}
Vyacheslav Futorny, Dimitar Grantcharov, Luis~Enrique Ramirez, and Pablo
  Zadunaisky, \emph{Bounds of {{Gelfand-Tsetlin}} multiplicities and tableaux
  realizations of {{Verma}} modules}, Journal of Algebra \textbf{556} (2020),
  412--436.

\bibitem[FGRZ20b]{FGRZGalois}
\bysame, \emph{Gelfand-{{Tsetlin}} theory for rational {{Galois}} algebras},
  Israel Journal of Mathematics \textbf{239} (2020), no.~1, 99--128.

\bibitem[GC50]{gelfandFinitedimensionalRepresentations1950}
I.~M. Gelfand and M.~L. Cetlin, \emph{Finite-dimensional representations of the
  group of unimodular matrices}, Doklady Akad. Nauk SSSR (N.S.) \textbf{71}
  (1950), 825--828. \MR{35774}

\bibitem[Har20]{Hartwig}
Jonas~T. Hartwig, \emph{Principal {{Galois}} orders and {{Gelfand-Zeitlin}}
  modules}, Advances in Mathematics \textbf{359} (2020), 106806.

\bibitem[KL09]{KLI}
Mikhail Khovanov and Aaron~D. Lauda, \emph{A diagrammatic approach to
  categorification of quantum groups. {{I}}}, Representation Theory \textbf{13}
  (2009), 309--347.

\bibitem[KLMS12]{khovanovExtendedGraphical2012}
Mikhail Khovanov, Aaron~D. Lauda, Marco Mackaay, and Marko Sto{\v s}i{\'c},
  \emph{Extended graphical calculus for categorified quantum {{sl(2)}}},
  Memoirs of the American Mathematical Society \textbf{219} (2012), no.~1029,
  vi+87. \MR{2963085}

\bibitem[KTW{\etalchar{+}}19]{KTWWYO}
Joel Kamnitzer, Peter Tingley, Ben Webster, Alex Weekes, and Oded Yacobi,
  \emph{On category {{O}} for affine {{Grassmannian}} slices and categorified
  tensor products}, Proceedings of the London Mathematical Society \textbf{119}
  (2019), no.~5, 1179--1233.

\bibitem[Lau10]{laudaCategorificationQuantum2010}
Aaron~D. Lauda, \emph{A categorification of quantum $\mathfrak{sl}(2)$},
  Advances in Mathematics \textbf{225} (2010), no.~6, 3327--3424. \MR{2729010
  (2012b:17036)}

\bibitem[Lec04]{Lecshuf}
Bernard Leclerc, \emph{Dual canonical bases, quantum shuffles and
  {{q}}-characters}, Mathematische Zeitschrift \textbf{246} (2004), no.~4,
  691--732. \MR{2045836 (2005c:17019)}

\bibitem[Maz99]{mazorchukOGZ}
Volodymyr Mazorchuk, \emph{Orthogonal {{Gelfand-Zetlin}} algebras {{I}}},
  Contributions to Algebra and Geometry \textbf{40} (1999), no.~2, 399--415.

\bibitem[McN15]{mcnamaraFiniteDimensional2015}
Peter~J. McNamara, \emph{Finite dimensional representations of
  {{Khovanov}}--{{Lauda}}--{{Rouquier}} algebras {{I}}: {{Finite}} type},
  Journal f{\"u}r die reine und angewandte Mathematik (Crelles Journal)
  \textbf{2015} (2015), no.~707, 103--124.

\bibitem[MV21]{MVHC}
Volodymyr Mazorchuk and Elizaveta Vishnyakova, \emph{Harish-{C}handra modules
  over invariant subalgebras in a skew-group ring}, Asian J. Math. \textbf{25}
  (2021), no.~3, 431--454. \MR{4395608}

\bibitem[Soe92]{soergelCombinatoricsHarishChandra1992}
Wolfgang Soergel, \emph{The combinatorics of {{Harish-Chandra}} bimodules},
  Journal Fur Die Reine Und Angewandte Mathematik \textbf{429} (1992), 49--74.
  \MR{MR1173115 (94b:17011)}

\bibitem[SW]{SWschur}
Catharina Stroppel and Ben Webster, \emph{Quiver {Schur} algebras and q-{Fock}
  space}, \arxiv{1110.1115}.

\bibitem[VV11]{varagnoloCanonicalBases2011a}
M.~Varagnolo and E.~Vasserot, \emph{Canonical bases and affine {{Hecke}}
  algebras of type {{B}}}, Inventiones Mathematicae \textbf{185} (2011), no.~3,
  593--693. \MR{2827096}

\bibitem[Weba]{WebGT}
Ben Webster, \emph{Gelfand-{{Tsetlin}} modules in the {{Coulomb}} context},
  \arxiv{1904.05415}.

\bibitem[Webb]{websterKoszulDuality2019}
\bysame, \emph{Koszul duality between {{Higgs}} and {{Coulomb}} categories
  $\mathcal{O}$}, \arxiv{1611.06541}.

\bibitem[Webc]{websterThreePerspectives2020}
\bysame, \emph{Three perspectives on categorical symmetric {{Howe}} duality},
  \arxiv{2001.07584}.

\bibitem[Web15]{WebCB}
\bysame, \emph{Canonical bases and higher representation theory}, Compositio
  Mathematica \textbf{151} (2015), no.~1, 121--166.

\bibitem[Web17a]{Webmerged}
\bysame, \emph{Knot invariants and higher representation theory}, Memoirs of
  the American Mathematical Society \textbf{250} (2017), no.~1191, 141.

\bibitem[Web17b]{Webqui}
\bysame, \emph{On generalized category {{O}} for a quiver variety},
  Mathematische Annalen \textbf{368} (2017), no.~1, 483--536.

\bibitem[Web19]{WebwKLR}
\bysame, \emph{Weighted {{Khovanov-Lauda-Rouquier}} algebras}, Documenta
  Mathematica \textbf{24} (2019), 209--250. \MR{3946709}

\bibitem[{Web}20]{WebBKnote}
{Webster}, \emph{On graded presentations of {{Hecke}} algebras and their
  generalizations}, Algebraic Combinatorics \textbf{3} (2020), no.~1, 1--38.

\end{thebibliography}
\bibliographystyle{amsalpha}
\vfill
\ifslfour

\ifslfour
\pagestyle{empty}
\eject
\pdfpagewidth=12in \pdfpageheight=35in

\begin{tabular}{|c|c|c|c|c|c|c|c|c|c|c|c|c|c|c|c|c|c|c|c|c|c|c|c|c|c|c|c|c|c|c|}\hline
IWS&F&F&F&F&F&F&F&F&F&F&F&F&F&F&F&F&F&F&F&F&F&F&F&F&F&F&F&F&F&F\\\hline
  GK dim &6& 5& 5& 5& 5& 3& 5& 4& 5& 5& 4& 3& 5& 4& 3& 3& 4& 3& 5& 3& 3& 3& 3& 0& 6& 5& 5& 5& 5& 3 \\\hline
  4444333221&\boxed{12}&&&&&&&&&&&&&&&&&&&&&&&&&&&&&\\\hline
4444333212&\boxed{6}&&&&&&&&&&&&&&&&&&&&&&&&&&&&&\\\hline
4444332321&\boxed{8}&&&&&&&&&&&&&&&&&&&&&&&&&&&&&\\\hline
4444332213&\boxed{4}&&&&&&&&&&&&&&&&&&&&&&&&&&&&&\\\hline
4444332132&\boxed{4}&&&&&&&&&&&&&&&&&&&&&&&&&&&&&\\\hline
4444332123&\boxed{2}&&&&&&&&&&&&&&&&&&&&&&&&&&&&&\\
\hline
\end{tabular}
  \eject \pdfpagewidth=12in \pdfpageheight=42in

  \begin{tabular}{|c|c|c|c|c|c|c|c|c|c|c|c|c|c|c|c|c|c|c|c|c|c|c|c|c|c|c|c|c|c|c|}\hline
    IWS&F&F&F&F&F&F&T&T&F&F&T&F&F&F&F&T&F&F&F&F&F&F&T&T&T&T&T&F&T&F\\\hline
    GK dim &5& 4& 4& 5& 3& 3& 6& 5& 5& 5& 5& 3& 5& 4& 4& 5& 3& 3& 6& 5& 5& 5& 6& 5& 5& 5& 5& 3& 5& 4\\\hline
  4434343221&&&2&&&&&&&&&&&&&&&&&&&&&&&&&&&\\\hline
4434343212&&&1&&&&&&&&&&&&&&&&&&&&&&&&&&&\\\hline
4434332214&&&4&&&&&&&&&&&&&&&&&&&&&&&&&&&\\\hline
4434332124&&&2&&&&&&&&&&&&&&&&&&&&&&&&&&&\\\hline
4434324321&&&1&&&&&&&&&&&&&&&&&&&&&&&&&&&\\\hline
4434323214&&&2&&&&&&&&&&&&&&&&&&&&&&&&&&&\\
\hline
  \end{tabular}
  \eject \pdfpagewidth=12in \pdfpageheight=43in
  \begin{tabular}{|c|c|c|c|c|c|c|c|c|c|c|c|c|c|c|c|c|c|c|c|c|c|c|c|c|c|c|c|c|c|c|}\hline
    IWS&F&T&F&F&T&F&F&F&T&F&F&F&F&F&F&F&F&F&F&F&F&F&F&F&F&F&F&F&F&F\\\hline
 GK dim & 4& 5& 3& 3& 6& 5& 5& 5& 6& 5& 5& 5& 6& 6& 5& 5& 5& 4& 5& 5& 4& 3& 5& 3& 3& 6& 5& 5& 5& 5\\\hline
    4444332321&&&&&&&&&&&&&&2&&&&&&&&&&&&&&&&\\\hline
4444332213&&&&&&&&&&&&&&4&&&&&&&&&&&&&&&&\\\hline
4444332123&&&&&&&&&&&&&&2&&&&&&&&&&&&&&&&\\\hline
4444323321&&&&&&&&&&&&&&4&&&&&&&&&&&&&&&&\\\hline
4444323213&&&&&&&&&&&&&&4&&&&&&&&&&&&&&&&\\\hline
4444322133&&&&&&&&&&&&&&\boxed{4}&&&&&&&&&&&&&&&&\\
\hline
\end{tabular}

  \eject \pdfpagewidth=12in \pdfpageheight=48.5in
  \begin{tabular}{|c|c|c|c|c|c|c|c|c|c|c|c|c|c|c|c|c|c|c|c|c|c|c|c|c|c|c|c|c|c|c|}\hline 
    IWS&F&F&F&F&F&F&F&F&F&F&F&T&T&T&T&F&F&T&F&F&F&F&T&F&F&T&F&T&T&F\\\hline
    GK dim & 3& 5& 4& 4& 5& 3& 3& 6& 5& 5& 5& 6& 5& 5& 5& 4& 4& 6& 5& 5& 5& 6& 6& 5& 5& 5& 4& 5& 5& 4\\\hline
    4444332132&&&&&&&&&&&&&&&&&&&&&&&2&&&&&&&\\\hline
4444321332&&&&&&&&&&&&&&&&&&&&&&&4&&&&&&&\\\hline
4444321323&&&&&&&&&&&&&&&&&&&&&&&2&&&&&&&\\\hline
4443432132&&&&&&&&&&&&&&&&&&&&&&&1&&&&&&&\\\hline
4443321324&&&&&&&&&&&&&&&&&&&&&&&&&&&&&2&\\\hline
4443214332&&&&&&&&&&&&&&&&&&&&&&&2&&&&&&&\\
\hline
  \end{tabular}
   \eject \pdfpagewidth=12in \pdfpageheight=26in
  \begin{tabular}{|c|c|c|c|c|c|c|c|c|c|c|c|c|c|c|c|c|c|c|c|c|c|c|c|c|c|c|c|c|c|c|}\hline 
    IWS& F&F&F&F&T&T&T&T&T&F&T&T&T&T&F&F&T&T&T&T&T&F&F&F&F&F&T&F&F&F\\\hline
    GK dim & 3& 5& 3& 3& 6& 5& 5& 5& 5& 3& 5& 4& 4& 5& 3& 3& 6& 5& 5& 5& 6& 5& 5& 5& 4& 4& 6& 5& 5& 5\\\hline
    4433213244&&&&&&&&&&&&&2&&&&&&&&&&&&&&&&&\\\hline
4432134432&&&&&&&&&&&&&1&&&&&&&&&&&&&&&&&\\\hline
4432134324&&&&&&&&&&&&&2&&&&&&&&&&&&&&&&&\\\hline
4432133244&&&&&&&&&&&&&4&&&&&&&&&&&&&&&&&\\\hline
4432132434&1&&&&&&&&&&&&1&&&&&&&&&&&&&&&&&\\\hline
4432132344&&&&&&&&&&&&&2&&&&&&&&&&&&&&&&&\\
\hline
    \end{tabular}
 
    \eject \pdfpagewidth=12in \pdfpageheight=68in
  \begin{tabular}{|c|c|c|c|c|c|c|c|c|c|c|c|c|c|c|c|c|c|c|c|c|c|c|c|c|c|c|c|c|c|c|}\hline 
    IWS&T&F&F&F&F&F&F&F&F&F&F&F&F&F&F&F&F&F&F&F&F&F&F&F&F&F&F&F&F&F\\\hline
    GK dim & 6& 6& 5& 5& 5& 6& 5& 5& 5& 4& 4& 6& 5& 5& 5& 6& 6& 5& 5& 5& 5& 5& 5& 4& 3& 3& 4& 3& 6& 5\\\hline
    4444333212&&&&&&&&&&&&&&&&&6&&&&&&&&&&&&&\\\hline
4444332132&&&&&&&&&&&&&&&&&4&&&&&&&&&&&&&\\\hline
4444332123&&&&&&&&&&&&&&&&&2&&&&&&&&&&&&&\\\hline
4444323213&&2&&&&&&&&&&&&&&&&&&&&&&&&&&&&\\\hline
4444322133&&4&&&&&&&&&&&&&&&&&&&&&&&&&&&&\\\hline
4444321332&&&&&&&&&&&&&&&&&2&&&&&&&&&&&&&\\
\hline 
    \end{tabular}
     \eject \pdfpagewidth=12in \pdfpageheight=53in
     \begin{tabular}{|c|c|c|c|c|c|c|c|c|c|c|c|c|c|c|c|c|c|c|c|c|c|c|c|c|c|c|c|c|c|c|}\hline
       IWS&F&F&F&F&T&T&F&F&T&F&F&F&F&T&F&F&F&F&F&F&T&F&F&F&F&F&F&F&F&F\\\hline
       GK dim & 5& 5& 4& 4& 6& 5& 5& 5& 5& 3& 5& 4& 4& 5& 3& 3& 6& 5& 5& 5& 6& 5& 5& 5& 6& 6& 5& 5& 5& 4\\\hline
       4444332123&&&&&&&&&&&&&&&&&&&&&&&&&&2&&&&\\\hline
4444321323&&&&&&&&&&&&&&&&&&&&&&&&&&1&&&&\\\hline
4444321233&&&&&&&&&&&&&&&&&&&&&&&&&&2&&&&\\\hline
4443432123&&&&&&&&&&&&&&&&&&&&&&&&&&1&&&&\\\hline
4443214323&&&&&&&&&&&&&&&&&&&&&&&&&&&1&&&\\\hline
4443212433&&&&&&&&&&&&&&&&&&&&&&&&&&&2&&&\\
\hline
     \end{tabular}
   
    \eject \pdfpagewidth=12in \pdfpageheight=46in
 
\begin{tabular}{|c|c|c|c|c|c|c|c|c|c|c|c|c|c|c|c|c|c|c|c|c|c|c|c|c|c|c|c|c|c|c|}\hline
  IWS&F&F&F&F&F&F&F&F&F&F&F&F&F&F&F&F&F&F&F&F&F&F&F&T&T&T&T&F&F&T\\\hline
  GK dim &5& 5& 4& 3& 5& 3& 3& 6& 5& 5& 5& 5& 3& 5& 4& 4& 5& 3& 3& 6& 5& 5& 5& 6& 5& 5& 5& 4& 4& 6 \\\hline
  4443321234&&2&&&&&&&&&&&&&&&&&&&&&&&&&&&&\\\hline
4443213234&&1&&&&&&&&&&&&&&&&&&&&&&&&&&&&\\\hline
4443212343&&1&&&&&&&&&&&&&&&&&&&&&&&&&&&&\\\hline
4443212334&&2&&&&&&&&&&&&&&&&&&&&&&&&&&&&\\\hline
4434321234&&1&&&&&&&&&&&&&&&&&&&&&&&&&&&&\\\hline
4433212344&&&&&&&&&&&&&&&&2&&&&&&&&&&&&&&\\
\hline
  
     \end{tabular}

     \eject \pdfpagewidth=8.5in \pdfpageheight=46in
 
    \begin{tabular}{|c|c|c|c|c|c|c|c|c|c|c|c|c|c|c|c|c|c|c|c|c|c|c|c|c|c|}\hline
      IWS&F&F&F&F&F&F&F&F&F&F&F&F&F&F&F&F&F&F&F\\\hline
      GK dim & 5& 5& 5& 6& 6& 5& 5& 5& 6& 5& 5& 5& 4& 4& 6& 5& 5& 5& 6\\\hline
4444321323&&&&&1&&&&&&&&&&&&&&\\\hline
4444321233&&&&&2&&&&&&&&&&&&&&\\\hline
4443213234&&&&&&&&1&&&&&&&&&&&\\\hline
4443212343&&&&&&&&1&&&&&&&&&&&\\\hline
4443212334&&&&&&&&2&&&&&&&&&&&\\\hline
4432132344&&&&&&&&&&&&&&1&&&&&\\
\hline
    \end{tabular}
  \eject \pdfpagewidth=8.5in \pdfpageheight=11in
  \fi
 \fi
\end{document}